\documentclass[11pt]{article}
\title{The satisfiability threshold of random linear equations over finite commutative rings}

\author{Pu Gao\thanks{Research supported by NSERC RGPIN-04173-2019.} \\ University of Waterloo \\ pu.gao@uwaterloo.ca \and Theodore Morrison\thanks{Research supported by an Ontario Graduate Scholarship} \\ University of Waterloo\\ tmorriso@uwaterloo.ca}

\date{}
\usepackage{amsfonts, amsmath, amssymb, amsthm, bigints, bbm,tikz,scrextend,mathrsfs}
\usepackage[shortlabels]{enumitem}
\usepackage[normalem]{ulem}
\usetikzlibrary{graphs, quotes}
\usepackage[margin=1in]{geometry}
\usepackage[dvipsnames]{xcolor}
\newtheorem{theorem}{Theorem}[section]

\newtheorem{lemma}[theorem]{Lemma}

\newtheorem{definition}[theorem]{Definition}
\newtheorem{observation}[theorem]{Observation}
\newtheorem{conjecture}[theorem]{Conjecture}
\newtheorem{example}[theorem]{Example}
\newtheorem{problem}[theorem]{Problem}

\newtheorem{claim}[theorem]{Claim}

\newcommand{\ZZ}{\mathbb{Z}}
\newcommand{\FF}{\mathbb{F}}

\newcommand{\NN}{\mathbb{N}}
\newcommand{\eps}{\varepsilon}

\newcommand{\lt}{\left}
\newcommand{\rt}{\right}
\newcommand{\rv}{\boldsymbol}

\newcommand{\Var}{\operatorname{Var}}
\newcommand{\mc}{\mathcal}

\renewcommand{\subset}{\subseteq}
\renewcommand{\supset}{\supseteq}

\newcommand{\Spec}{\operatorname{Spec}}

\newcommand{\supp}{\operatorname{supp}}

\newcommand{\Po}{\operatorname{Po}}
\newcommand{\Bin}{\operatorname{Bin}}

\newcommand{\Ann}{\operatorname{Ann}}

\newcommand{\Mult}{\operatorname{Mult}}

\newcommand{\contig}{\boldsymbol{\triangleright}}
\newcommand{\mcontig}{\boldsymbol{\triangleleft}\boldsymbol{\triangleright}}

\theoremstyle{definition}
\newcommand{\ind}[1]{\boldsymbol{1}_{\{#1\}}}
\def\deg{{\text{deg}}}
\def\da{{\boldsymbol{\delta}}}
\def\ka{{\boldsymbol{\kappa}}}

\newcommand\remove[1]{{}}

\DeclareSymbolFont{symbolsC}{U}{txsyc}{m}{n}
\DeclareMathSymbol{\strictif}{\mathrel}{symbolsC}{74}
\DeclareMathSymbol{\strictfi}{\mathrel}{symbolsC}{75}
\DeclareMathSymbol{\strictiff}{\mathrel}{symbolsC}{76}
\usepackage{fullpage}

\def\A{{\bf A}}
\def\B{{\bf B}}
\def\b{{\bf b}}
\def\C{{\bf C}}

\def\d{{\bf d}}
\def\bb{{\beta}}

\def\pr{{\mathbb P}}
\def\ex{{\mathbb E}}

\begin{document}

\maketitle

\begin{abstract}
    We extend the study of random linear equations over finite fields to equations over finite commutative rings. We characterize precisely when the satisfiability threshold occurs at a sublinear scale; namely, when the random system become unsatisfiable with high probability with a number of constraints $m$ that is sublinear in $n$, the number of variables. In this regime, we determine the exact value of the satisfiability threshold. 
    
    In the complementary regime where the satisfiability threshold is linear in $n$, we determine its precise value when $R$ is a principal ring. Interestingly, this value is independent of the choice of $R$, mirroring the same phenomenon when $R$ is a finite field. We further prove that this independence of $R$ breaks down if $R$ is nonprincipal. In particular, we investigate a classical family of nonprincipal rings and determine the satisfiability thresholds for all rings in this family. Remarkably, in this setting, the satisfiability threshold depends not only on the underlying ring, but also on other parameters defining the random linear equation model.  
\end{abstract}

\newpage

\section{Introduction}

Random constraint satisfaction problems (CSPs) are a central topic at the intersection of theoretical computer science, probability, combinatorics, statistical physics, and optimization. A CSP consists of a set of $n$ variables and a collection of $m$ constraints, and the goal is to determine whether there exists an assignment of the variables that satisfies all constraints simultaneously. In the random setting, the variables and constraints are generated according to a probability distribution. Among the most widely studied random CSP models are those in which each constraint involves exactly $k$ variables chosen uniformly at random from all of the $n$ variables. For many such models, it has been established  that as $n$ grows, the probability of the random CSP being satisfiable undergoes a sharp transition from  one to zero, when the ratio $m/n$ is around some critical value $d_k$ called the satisfiability threshold. The value of $d_k$ depends on both the type of CSP and the underlying random model.

Determining the satisfiability threshold is one of the central problems in the study of random CSPs. This research  began with random $k$-SAT in the early 1990s~\cite{Cheeseman1991,Selman1992,CRAWFORD1996}, motivated by extensive experimental and theoretical investigations into the phase transition phenomenon.  A major milestone was the work of Friedgut and Bourgain~\cite{friedgut1999sharp}, who established the existence of a sharp satisfiability threshold for random $k$-SAT. However the exact value of the threshold remained unknown for each fixed $k\ge 3$. The first breakthrough result in determining the satisfiability of a nontrivial random CSP was achieved by Dubois and Mander~\cite{duboismandler2002XORSAT}  in the study of random 3-XORSAT in 2002. 
Later, this work was extended to $k$-XORSAT for any fixed $k\ge 3$ by Dietzfelbinger et al.~\cite{dietzfelbinger2010XORSATthreshold} and independently by Pittel and Sorkin~\cite{pittelsorkin2016XORSATthreshold}. Since the constraints in a $k$-XORSAT are simply  linear equations over $\FF_2$, it is natural to ask how the satisfiability threshold behaves for random linear equations over other algebraic structures. The special case of $\FF_3$ and $\FF_4$ was solved by Goerdt and Falke~\cite{goerdt2012beyondxorsat}, and the satisfiability threshold of random linear equations over a general finite field was determined by Ayre et al.\ in~\cite{ayre2020satisfiability}. A striking phenomenon emerging from these results is that, despite the different algebraic structures of the underlying finite fields, the satisfiability threshold is universal: it takes the same value  $d_k$ for every finite field (see its definition in~\eqref{def:dk} in Theorem~\ref{thm:field-sat} below). In parallel,  continual progress has been made towards determining the satisfiability threshold of the random $k$-SAT~\cite{Ding2016,GOERDT1996469,Kirousis1998,Achlioptas2002,CojaOghlan2016}. Much of the work was significantly influenced by ideas and techniques from statistical physics. In particular, the precise value of the $k$-SAT satisfiability threshold was predicted using the so-called cavity method from statistical physics.  This problem was eventually resolved by of Ding, Sly and Sun~\cite{dingslysun2022kSAT} for sufficiently large $k$.  Beyond random $k$-SAT and linear equations,  satisfiability thresholds have also been investigated for a wide range of other random CSPs, including $k$-NAE-SAT, $1$-in-$k$-SAT, graph and hypergraph colouring, etc.~\cite{Ding2016,Achlioptas2001}. For some of these problems, the threshold has been rigorously determined, while for some of the others, precise predictions have been obtained through the cavity method and related techniques from statistical physics. 

Recall the striking phenomenon that the satisfiability threshold of random linear equations is universal regardless of the finite field over which the equations are defined.  
Connamacher and Molloy  boldly conjectured that this is due to a special property these CSP all share: the uniquely extendability, meaning that  for any constraint, fixing the values of any $k-1$ variables uniquely determines the value of the remaining variable needed to satisfy the constraint.
In their work~\cite{connamacher2012satisfiability} they conjectured that any random UE-CSP (uniquely extendable CSP) instance has the same satisfiability threshold $d_k$ as in~\eqref{def:dk},  and they confirmed this conjecture in the special case $(k,r)=(3,4)$ where $r$ is the number of different values a  variable can take. This is the smallest pair of parameters not covered by the theory of linear equations over finite fields. More recently, Gao and Morrison~\cite{gao2026satisfiabilitythresholdsolutionspace} confirmed this conjecture for the class of so-called ``reducible'' UE-CSPs. 

In this paper we extend this study to random linear equations over a finite commutative ring $(R,+,\cdot)$. There are several motivations for considering this broader setting. Linear equations over rings have a rich history across many areas of mathematics, including number theory, algebra, additive combinatorics, and optimisation. In particular, a wide range of combinatorial problems can be formulated as the problem of finding integer or modular solutions to systems of linear equations, where the algebraic structure of the underlying ring captures essential combinatorial constraints.
From the perspective of random CSPs, linear equations over a finite field are examples of UE-constraints. However, when the underlying field is replaced by a general finite ring, this property fails: the equations may remain extendable but no longer uniquely so. They may also cease to be extendable. Thus, random linear equations over finite rings provide a natural framework that interpolates between different classes of random CSPs and encompasses a broad range of constraints with different nature. This leads to several fundamental questions. How does the satisfiability threshold depend on the choice of the ring $R$? Does the universality phenomenon observed for finite fields extend to broader classes of finite rings? If not, to what extent is the threshold determined by the algebraic properties of $R$, and which structural features of the ring govern this dependence? Addressing these questions is the main objective of this paper.

Beyond the perspective of random CSPs, random linear equations over rings are also of independent interest in random matrix theory, a field with far-reaching applications in areas such as coding theory, combinatorics, cryptography, and information theory. A central theme in this area is the study of the rank of random matrices~\cite{Tikhomirov2020,Bordenave2009,Costello2010,glasgow2023}.   In fact, for certain classes of random matrices $B$, the rank can be determined through the satisfiability of random linear systems $Bx=b$ where $b$ is chosen according to an appropriate distribution; see for instance~\cite{ayre2020satisfiability}. Random matrices over rings were studied in~\cite{glasgow2023,Cheong2021} and universality results on the asymptotic distribution of the co-kernel were established. However, compared with the extensive theory over finite fields, many fundamental questions about random matrices over rings remain open. We discuss several related open problems in Section~\ref{sec:open}.

We adopt a model of random linear equations similar to~\cite{ayre2020satisfiability}. Let $k\ge 3$ be an integer, and $\hat P$ be a distribution over $R^k$. We say that $\hat P$ is permutation-invariant if $P(\sigma_1,\ldots, \sigma_k)=P(\sigma_{\pi(1)},\ldots, \sigma_{\pi(k)})$ for all $(\sigma_1,\ldots, \sigma_k)\in R^k$ and all permutations $\pi$ of $\{1,\ldots, k\}$. For each $m,n\in \NN$, let $C_n=\{a_1,\ldots, a_m\}$ and $V_n=\{v_1,\ldots, v_n\}$ be disjoint index sets, and define a random $m\times n$ matrix $\B\in R^{C_m\times V_n}$ over $R$ as follows: independently for each $a\in C_m$, choose $j_1<j_2<\cdots <j_k$ uniformly at random from $[n]$; sample $(\sigma_1\ldots, \sigma_k)$ according to $\hat P$;  
and let $\B_{j,j_i}=\sigma_i$ for all $1\le i\le k$ and $\B_{j,h}=0$ for all $h\notin \{j_1,j_2,\ldots,j_k\}$. Let $\hat Q$ be a distribution over $R^m$ and let $\b$ be sampled by $\hat Q$. Consider the random linear equations over $R$ given by 

\begin{equation}\label{eq:system}
\B x=\b.
\end{equation}
We say the system is satisfiable (SAT) if there exists $x\in R^n$ satisfying all the equations in (\ref{eq:system}). Otherwise, we say (\ref{eq:system}) is unsatisfiable (UNSAT). The satisfiability threshold in the case where $R$ is a field is given as follows.

\begin{theorem}[\cite{ayre2020satisfiability}]\label{thm:field-sat} Let $k\ge 3$ be a fixed integer, and let 
   \begin{align}
            \rho_{k,d} &= \sup\lt\{x\in [0,1]: x=1-\exp(-dx^{k-1})\rt\}\nonumber\\
            d_k&=\inf\lt\{d>0:\rho_{d,k}-d\rho_{d,k}^{k-1}+(1-1/k)d\rho_{d,k}^k<0\rt\}.\label{def:dk}
        \end{align}
Suppose that $R$ is a finite field, that the entries of $\b$ are drawn uniformly and independently from $R$, and that $\hat P$ is permutation-invariant. Then, for every $\eps>0$, a.a.s.\ \eqref{eq:system} is satisfiable if $m<(d_k/k-\eps)n$ and is unsatisfiable if $m>(d_k/k+\eps)n$.

\end{theorem}

We only consider finite commutative rings $R$, and we investigate the satisfiability (SAT) threshold for $\B x=\b$ over $R$. Moreover, we only consider the model where the linear equations are independent. In other words, the components in $\b$ are drawn independently.
Suppose that $Q$ is a distribution over $R$. Let $Q^{\otimes m}$ denote the $m$-fold product measure on $R^m$; i.e.\ $Q^{\otimes m}(b_1,\ldots,b_m)=\prod_{i=1}^m Q(b_i)$ for all $(b_1,\ldots,b_m)\in R^m$. Hence, in this paper, we always consider $\hat Q=Q^{\otimes k}$ for some distribution $Q$ over $R$. The probability measure $\hat P$ does not need to be a product measure. However, the choice $\hat P=P^{\otimes k}$ for some distribution $P$ over $R$ is is among the most natural and important cases, and thus will be our primary focus. 

The proof of Theorem~\ref{thm:field-sat} critically uses the ``unique extendability'' property of linear equations over a field. That is, for any assignment of $x_1,\ldots, x_{k-1}\in\FF_q$ in a equation $\sum_{i=1}^ka_ix_i=b$ over $\FF_q$, there exists a unique value $x_k\in \FF_q$ that satisfies the equation. This property fails when $R$ is not a field, which presents an obstacle to adapting the proof of Theorem~\ref{thm:field-sat}. We begin to address this problem by omitting the most trivial case, which is when the equation $\sum_{i=1}^ka_ix_i=b$ over $R$ has no solutions. 
\begin{example}\label{trivialUNSAT}
    Suppose $(a_1,\ldots, a_k)\in \supp(\hat P)$ and $b\in \supp(Q)$ are such that $b$ is not in the ideal generated by $\{a_1,\ldots, a_k\}$, then there is no solution to $\sum_{i=1}^k a_i x_i = b$ over $R$.
\end{example}

\begin{definition}
    Let $\hat P$ be a distribution on $R^k$ and let $Q$ be a distribution on $R$. We say that $(\hat P,Q)$ is inconsistent if, for all $(a_1,\ldots, a_k)\in \supp(\hat P)$ and all $b\in \supp(Q)$, $b$ is not in the ideal generated by $\{a_1,\ldots, a_k\}$.  
\end{definition}
If $(\hat P,Q)$ are inconsistent, then every equation $\sum_{i=1}^ka_ix_i=b$ with $(a_1,\ldots, a_k)\in \supp(\hat P),b\in \supp(Q)$ is unsatisfiable, so~\eqref{eq:system} is trivially unsatisfiable.

 Next we demonstrate an example where the system is trivially always satisfiable. 
\begin{definition}
 Let $\hat P$ be a distribution on $R^k$ and let $Q$ be a distribution on $R$. We say that $(\hat P,Q)$ admits a $k$-ary constant solution $\sigma\in R$ if, for all $(a_1,\ldots, a_k)\in \supp(\hat P)$ and $b\in \supp(Q)$,   
    \begin{equation*}
        \sum_{i=1}^k a_i \sigma=b.
    \end{equation*}
\end{definition}
Clearly, if $(P,Q)$ admits a $k$-ry constant solution $\sigma$, then the constant solution $x=\sigma^n$ (i.e. $x_i=\sigma$ for all $i\in [n]$) satisfies~\eqref{eq:system}. To study SAT threshold, we certainly would exclude cases like this to avoid triviality. This is similar to a homogeneous system with $b=0$, but there are examples of $(P,Q)$ with $\supp(Q)\neq \{0\}$ that admit a $k$-ary constant solution. For example $R=\ZZ_4$ with $\supp(P)=\ZZ_4$, $\supp(Q)=\{3\}$, and $k=3$ admits a constant solution $1$.

Throughout this paper we assume the following assumptions on $(\hat P, Q)$ to avoid trivialities demonstrated above.
\begin{description}
    \item[(A1)] $(\hat P,Q)$ is not inconsistent.
    \item[(A2)] $(\hat P,Q)$ does not support any constant solution.
\end{description}

\subsection{Main results}

First, we consider the case where $\hat P=P^{\otimes k}$ for some distribution $P$ over $R$. Our first result provides a complete characterization of the sublinear SAT threshold, determines the associated coarse threshold, and establishes that the phase transition is not sharp in this regime. To state the result we need a technical definition. We use the notation $(I:J)$ to denote the ideal quotient $(I:J)=\{t\in R:Jt\subset I\}$ for two ideals $I,J\subset R$.

\begin{definition}\label{def:I}
    Let $P$ and $Q$ be distributions over $R$. Define a sequence of ideals $I_0(P),I_1(P),\ldots$ inductively by setting $I_0(P)=\bigcap_{r\in\supp(P)}(r)$, and setting 
    \begin{equation*}
        I_\ell(P)=\bigcap_{r\in\supp(P)}r \left(\bigcap_{s\in \supp(P)}(I_{\ell-1}(P):s)\right)
    \end{equation*}
    for positive integers $\ell$. Define
    \begin{equation*}
        \ell(P,Q)=\inf\{\ell\geq 0: I_\ell(P)\not\supset \supp(Q)\}    
    \end{equation*}
    with the convention that $\inf (\emptyset) = \infty$.
\end{definition}

\subsection{Principal $R$}

\begin{theorem}\label{thm:principal-sublinear}
    Let $P$ and $Q$ be distributions on a principal ring $R$ and let $\hat P=P^{\otimes k}$. Suppose $(\hat P,Q)$ satisfies {\bf (A1)} and {\bf (A2)}. 
    Set
    \begin{equation}
        \alpha_{P,Q} = \frac{\sum_{i=1}^{\ell(P,Q)}k(k-1)^{i-1}}{1+\sum_{i=1}^{\ell(P,Q)}k(k-1)^{i-1}}. \label{def:alphaPQ}
    \end{equation}
    with $\alpha_{P,Q}=1$ if $\ell(P,Q)=\infty$. Then,  
    \begin{align*}
        \lim_{n\to\infty}\pr[\text{System~\eqref{eq:system} is satisfiable}]=\begin{cases}
            1&m \ll n^{\alpha_{P,Q}},\\
            0&m \gg n^{\alpha_{P,Q}}.
        \end{cases}
    \end{align*}
   Moreover, if $\alpha_{P,Q}<1$ and $m=Cn^{\alpha_{P,Q}}$, then for all fixed $C>0$, 
    \begin{align}
0<        \liminf_{n\to\infty}P[ \text{System~\eqref{eq:system} is satisfiable}]\le        \limsup_{n\to\infty}P[\text{System~\eqref{eq:system} is satisfiable}]<1. \label{eq:window}
    \end{align}
\end{theorem}

\noindent {\bf Remark}. If $(P,Q)$ is such that $(\hat P,Q)$ is inconsistent, then $\alpha_{P,Q}=0$, and~\eqref{eq:system} is trivially unsatisfiable for every $m\ge 1$. Thus, the inequality~\eqref{eq:window} does not hold. Hence it is necessary to assume {\bf (A1)} in this theorem.

This theorem shows that the SAT phase transition is not sharp when $\alpha_{P,Q}<1$, in contrast to Theorem~\ref{thm:field-sat} in the field case. What remains to be understood is the case where $\alpha_{P,Q}=1$. Is there a sharp phase transition for SAT in this regime? We confirm it in the case that $R$ is principal. Moreover, we show that the SAT threshold coincides exactly the threshold $d_k$ in Theorem~\ref{thm:field-sat} in the field case. Hence, the SAT threshold is independent of the structure and the order of $R$ provided that $R$ is principal.

\begin{theorem}\label{thm:principal}
    Suppose that $R$ is principal,  that $P,Q$ are distributions over $R$ such that $\alpha_{P,Q}=1$. Suppose $\hat P=P^{\otimes k}$ and $(\hat P,Q)$ satisfies {\bf (A2)}. Let $d_k$ be as defined in Theorem~\ref{thm:field-sat}.
    Then, for every fixed $\eps>0$,
    \begin{equation*}
        \lim_{n\to\infty}P[\text{System~\eqref{eq:system} is satisfiable}]=\begin{cases}
            1& \text{if $m<(d_k-\eps)n/k$},\\
            0&\text{if $m>(d_k+\eps)n/k$}.
        \end{cases}
    \end{equation*}
\end{theorem}

\noindent {\bf Remark}. Note that {\bf (A1)} is satisfied by the assumption that $\alpha_{P,Q}=1$.

\subsection{Nonprincipal $R$}
\label{sec:main_result_nonprincipal}

Given a matrix $B\in R^{C_m\times V_n}$, let its associated factor graph $F=F(B)$ be a bipartite graph on vertex set $V_n\cup C_m$ with edge set $E$ defined as follows: the vertices in $V_n$ correspond to the columns of $B$, and the vertices in $C_m$ correspond to the rows of $B$. There is an edge $(v,c)\in V_n\times C_m$ if $B_{c,v}\neq 0$. 

Given a linear system $Bx=b$ over $R$ we can associate a weighted factor graph $(F,w)$ to it by letting $F=F(B)$ and $w$ be the function on $C_m\cup E$ defined by $w_{c}=b_c$ for every vertex $c\in C$ and $w_{cv}=B_{c,v}$ for every edge $cv\in E$.

Conversely, given a weighted factor graph $(F,w)$, let $(B,b)=(B,b)_{(F,w)}$ be the linear system over $R$ whose weighted factor graph is $(F,w)$. We say $(F,w)$ is unsatisfiable if $B x= b$ has no solutions over $R$.
Given distributions $P$ and $Q$ over a ring $R$, let $\mc T(P,Q)$ be the set of unsatisfiable connected acyclic factor graphs whose edge weights are in $\supp(P)$ and whose vertex weights are in $\supp(Q)$.
\begin{theorem}\label{thm:nonprincipal-sublinear}
    Let $P$ and $Q$ be distributions on a ring $R$, and let $\hat P=P^{\otimes k}$. Suppose $\hat P$ and $\hat Q$ satisfy {\bf (A1)} and {\bf (A2)}. Define
    \begin{equation*}
        \alpha_{P,Q}=1-\frac{1}{\inf\lt\{|C(T)|:T\in\mc T(P,Q)\rt\}}.
    \end{equation*}
    with $\alpha_{P,Q}=1$ if $\mc T(P,Q)=\emptyset$. Then
    \begin{align*}
        \lim_{n\to\infty}\pr[\text{System~\eqref{eq:system} is satisfiable}]=\begin{cases}
            1&m \ll n^{\alpha_{P,Q}},\\
            0&m \gg n^{\alpha_{P,Q}}.
        \end{cases}
    \end{align*}
    Moreover, if $\alpha_{P,Q}<1$ and $m=Cn^{\alpha_{P,Q}}$, then for all fixed $C>0$, 
    \begin{align}
0<        \liminf_{n\to\infty}P[ \text{System~\eqref{eq:system} is satisfiable}]\le        \limsup_{n\to\infty}P[\text{System~\eqref{eq:system} is satisfiable}]<1. \label{eq:window-nonprincipal}
    \end{align}
\end{theorem}

Our next theorem characterises precisely when $\mc T(P,Q)$ is empty, and relates this charaterisation to $\ell(P,Q)$, introduced in Theorem~\ref{thm:principal-sublinear}.

\begin{theorem}\label{thm:nonprincipal-linear-characterisation}
    Let $P$ and $Q$ be distributions on a ring $R$. Then $\alpha_{P,Q}=1$ if and only if $\ell(P,Q)=\infty$.
\end{theorem}

However, for nonprincipal rings, the relation between $\alpha_{P,Q}$ and $\ell(P,Q)$ is not the same as in~\eqref{def:alphaPQ}. In fact, the right hand side of~\eqref{def:alphaPQ} is an upper bound for $\alpha_{P,Q}$ for nonprincipal $R$, as demonstrated in the following theorem.

\begin{theorem}\label{thm:nonprincipal-bounds}
    Let $P$ and $Q$ be distributions on a ring $R$, and let $p=|\supp(P)|$. If $\ell(P,Q)<\infty$, then
    \begin{equation*}
        \frac{\sum_{i=1}^{\ell(P,Q)}k(k-1)^{i-1}}{1+\sum_{i=1}^{\ell(P,Q)}k(k-1)^{i-1}}\leq \alpha_{P,Q}\leq \frac{\sum_{i=1}^{\ell(P,Q)}k(k-1)^{i-1}p^{2i}}{1+\sum_{i=1}^{\ell(P,Q)}k(k-1)^{i-1}p^{2i}}
    \end{equation*}
\end{theorem}

We conjecture that there is a sharp SAT threshold when $\alpha_{P,Q}=1$ if $R$ is nonprincipal, although the value of the threshold may not coincide with $d_k$.

\begin{conjecture}\label{conj}
Suppose that $P,Q$ are distributions over $R$ such that $\alpha_{P,Q}=1$ and $\hat P=P^{\otimes k}$ satisfies {\bf{(A2)}}.
    There exists a constant $\beta=\beta(P,Q,R,k)>0$ such that for all fixed $\eps>0$,
    \begin{equation*}
        \lim_{n\to\infty}P[\text{System~\eqref{eq:system} is satisfiable}]=\begin{cases}
            1&m<(\beta-\eps)n/k,\\
            0&m>(\beta+\eps)n/k.
        \end{cases}
    \end{equation*}
\end{conjecture}

We did not attempt to determine the value of $\beta(P,Q,R,k)$ for general nonprincipal rings in this paper. In the theorem below, we determined $\beta(P,Q,R,k)$ for some examples of $(P,Q,R,k)$ where $R$ is nonprincipal, which establish that $\beta(P,Q,R,k)$ can be different from $d_k$ in Theorem~\ref{thm:principal}, and indeed can depend on the distribution of $P$. This is in contrast to the case that $R$ is a finite field, or a finite commutative principal ring.

Given a prime power $q$, let $\FF_q$ denote the finite field of order $q$. Let $\FF_q[X,Y]$ denote the polynomial ring in variables $X$ and $Y$ over $\FF_q$. For the theorem below, we consider $R=\FF_q[X,Y]/(X^2,Y^2)$, i.e.\ the quotient ring of $\FF_q[X,Y]$ by the ideal generated by $(X^2,Y^2)$. For any $p\in \FF_q[X,Y]$, let $\bar p$ denote its projection in $R$.

\begin{theorem}\label{thm:nonprincipal-linear}
    Let $R=\FF_q[X,Y]/(X^2,Y^2)$ and let $\rho\in [0,1]$. Suppose that $P$ is supported on $\{\bar X,\bar Y\}$, and that $P$ takes value $\bar X$ with probability $1-\rho$ and value $\bar Y$ with probability $\rho$. Suppose that $\supp(Q)=\{\bar X\bar Y,0\}$. Define $d_{k,\rho}$ by
    \begin{equation*}
        d_{k,\rho}=\sup\{d>0: \Phi_{d,k,\rho}(z)<\Phi_{d,k,\rho}(0)\text{ for all }z\in (0,1]\},\\
    \end{equation*}
    where
    \begin{equation*}
        \Phi_{d,k,\rho}(z)=e^{-\rho dz^{k-1}}+e^{-(1-\rho)dz^{k-1}}-d(1-1/k)z^k+dz^{k-1}-d/k.
    \end{equation*}
    Then, for any fixed $\eps>0$,
    \begin{equation*}
        \lim_{n\to\infty}P[\text{System~\eqref{eq:system} is satisfiable}]=\begin{cases}
            1& \text{if $m<(d_{k,\rho}-\eps)n/k$},\\
            0&\text{if $m>(d_{k,\rho}+\eps)n/k$}.
        \end{cases}
    \end{equation*}
\end{theorem}

\subsection{Other related work}

As mentioned earlier, the study of random CSPs, and in particular the aim to determine their satisfiability thresholds, played a central role in motivating the development of the cavity method in statistical physics. In turn, the insights and techniques arising from the cavity method have had a profound influence on the study of random CSPs, inspiring investigations beyond satisfiability alone, including the geometry and structure of their solution spaces. Other phase transitions such as ``clustering'', ``freezing'', ``condensation'' other than satisfiability also became the central focuses in the studies of random CSPs. These perspectives have led to a deeper understanding of the landscape of the solution space and have contributed to the development of creative and often highly nontrivial statistical-physics-inspired proof techniques for determining satisfiability thresholds in a variety of random CSP models. For a comprehensive account of the rich interplay between statistical physics and the theory of random CSPs, we refer the interested reader to~\cite{dingslysun2022kSAT,Ding2016,Panagiotou_Pasch_2025,achlioptas2006twomomentsNAESAT,ayre2020satisfiability} and the references therein.

\subsection{Open problems}
\label{sec:open}

Theorem~\ref{thm:principal-sublinear} we determined the satisfiability threshold of random system~\eqref{eq:system} to be $n^{\alpha_{P,Q}}$ where $\alpha_{P,Q}$ is given explicitly as an easily computable function of $P,Q$ and $R$ where $R$ is a finite commutative principal ring. This explicit characterization does not extend to the case of nonprincipal rings; see Theorem~\ref{thm:nonprincipal-sublinear}. Although $\alpha_{P,Q}$ is well defined as well for the nonprincipal rings, its definition involves the solution of an optimization problem over the space $\mathcal{T}(P,Q)$. Evaluating this optimization problem is challenging, largely because the members of $\mathcal{T}(P,Q)$ are not well understood. Our first open problem is to find an explicit expression for $n^{\alpha_{P,Q}}$ for nonprincipal rings.

\begin{problem}
Find an explicit expression for $\alpha_{P,Q}$ as an easily computable function  $P,Q$ and $R$  when $R$ is not principal.
\end{problem}

Conjecture~\ref{conj} states that a sharp satisfiability threshold exists for distributions $(P,Q)$ over a nonprincipal $R$ if $\alpha_{P,Q}=1$. Our next open problem is to determine the value of this threshold, assuming the validity of Conjecture~\ref{conj}.

\begin{problem}
Determine $\beta(P,Q,R,k)$ in Conjecture~\ref{conj} for general nonprincipal rings.
\end{problem}

Our next open problem concerns the ``rank'' of random matrices over rings. Unlike the case of matrices over fields, the rank of a matrix over a general ring is not always well defined. A natural analogue of rank in this setting is the size  of the kernel of the matrix.

\begin{problem}
Let $\B$ be the random matrix defined as above~\eqref{eq:system} and suppose that $m\sim dn$ where $d>0$ is fixed. Determine the asymptotic values of $$\limsup_{n\to \infty} n^{-1}\log |\text{ker}(\B)|\ \text{and}\ \liminf_{n\to \infty} n^{-1}\log |\text{ker}(\B)|.$$
\end{problem}

When $R$ is a field, the above two limits are equal, and their precise value as a function of $d$ was determined in~\cite{coja2020rank}.

\section{Random models and contiguity}\label{sec:contiguity}

Recall from Section~\ref{sec:main_result_nonprincipal} that we can specify a random matrix $\B$ by first specifying the distribution of its factor graph $F(\B)$, and then the distribution of the weight function $w$ on the edges in $F(B)$. We usually assume that $w$ and $F(\B)$ are independent, which is the case in our paper, and in most random models in the literature.

Coja-Oghlan, Gao, Hahn-Klimroth, Lee, M\"uller, and Rolvien in~\cite{coja2024full} determined the satisfiablity of $\B x=\b$ over finite fields for a random $(\B,\b)$ in great generality as follows.
Let $\rv d,\rv k\geq 0$ be integer valued random variables such that $\ex[\rv k^{2+\eps}],\ex[\rv d^{2+\eps}]<\infty$ for some $\eps>0$. Set $ d=\ex[\rv d]$ and $ k=\ex[\rv k]$. Let $\rv m,\{\rv k_{a_i}\}_{i\in\NN},\{\rv d_{v_i}\}_{i\in\NN}$ be mutually independent random variables, where  $\{\rv k_{a}\}_{a\in C_m}$ are i.i.d.\ copies of $\rv k$, and $\{\rv d_{v}\}_{v\in V_n}$ are i.i.d.\ copies of $\rv d$. Let $\mc D_{n}$ be the event
\begin{equation}\label{eq:D}
    \sum_{i=1}^n \rv d_{v_i}=\sum_{i=1}^{\rv m}\rv k_{a_i}.
\end{equation}
Let $G^{(1)}_{n,\rv m,\rv k,\rv d}$ be the random factor graph chosen from the uniform distribution on simple bipartite graphs on $C_{\rv m}\cup V_n$ with degree sequence $\{\rv k_a\}_{a\in C_{\rv m}}$, $\{\rv d_v\}_{v\in V_n}$ conditional on $\mc D_n$.
This is precisely the random graph model used in~\cite{coja2024full} to construct the random linear system $\B x=b$. 

Many parts of our results rely on reductions to random linear equations $\B x=b$ over some finite fields where $F(\B)$ is distributed approximately as $G^{(1)}_{n,\rv m,\rv k,\rv d}$ by choosing proper distributions for $\rv d$ and $\rv k$. However, there are several drawbacks of the model $G^{(1)}_{n,\rv m,\rv k,\rv d}$ which make it hard to apply the result in~\cite{coja2024full}. One is the i.i.d.\ conditions on $\rv d$. For instance, by taking $\B$ with distribution in~\eqref{eq:system}, the distribution of variable vertex degrees $\{\deg_{F(\B)}(v)\}_{v\in V_n}$ of $F(\B)$ is approximately independent Poisson but not exactly. The second is the conditioning on event $\mc D_{n}$, which is an event with vanishing probability. 

Our last contributions of this paper are to derive contiguity results among several models of random (factor) graphs, which are of independent interest for translating results from one model to the other. Next we define the models. Let $\rv m$, $\{\rv k_{a}\}_{a\in C_m},\{\rv d_{v}\}_{v\in V_n}$ be given as above.
\begin{itemize}
    \item $G_{n,\rv m,\boldsymbol{k}}$: Independently for each $a\in C_{\rv m}$, choose a uniform random subset of $V_n$ of size $\rv k_a$ and add and edge from $a$ to each variable in this set.
    \item $ M_{n,\rv m,\boldsymbol{k}}$: Independently for each $a\in C_{\rv m}$, generate $(u_1,\ldots, u_{\rv k_a})$ where each $u_i$ is chosen uniformly and independently from $V_n$. Then, add the $\rv k_a$ edges $\{au_i: i\in [\rv k_a]\}$. Note that the resulting random graph $M_{n,\rv m,\rv k}$ can be a multigraph since the elements in $\{u_1,\ldots, u_{\rv k_a}\}$ are not necessarily distinct for every $a\in C_{\rv m}$.
    \item $M^{(2)}_{n,\rv m,\rv k,\rv d}$: Condition on $\mc D_n$. Choose a random bipartite multigraph with degree sequence $\{\rv k_a\}_{a\in C_{\rv m}}$, $\{\rv d_v\}_{v\in V_n}$ generated by the pairing model, described as follows. Let $\Gamma$ be a random uniformly chosen perfect matching of the complete bipartite graph with vertex classes $\bigcup_{a\in C_{\rv m}} \{a\}\times [\rv k_a]$ and $\bigcup_{v\in V_n} \{v\}\times [\rv d_v]$. For each $a\in C_{\rv m},v\in V_n$, add an edge between $a$ and $v$ in $M^{(2)}_{n,\rv m,\rv k,\rv d}$ for each edge between the vertex sets $\{a\}\times [\rv k_a]$ and $\{v\}\times [\rv d_v]$ in $\Gamma$. By convention, the elements in $\{a\}\times [\rv k_a]$ are called the vertex-copies of $a$, and the the elements in $\{v\}\times [\rv d_v]$ are called the vertex-copies of $v$.
    \item $G^{(2)}_{n,\rv m,\rv k,\rv d}$:  Choose from the conditional distribution of $M^{(2)}_{n,\rv m,\rv k,\rv d}$ (hence conditioned on $\mc D_n$ as well) given that $M^{(2)}_{n,\rv m,\rv k,\rv d}$ has no parallel edges. 
\end{itemize}
We impose the following condition  to ensure that the above models  are  well defined.
\begin{enumerate}[label=({\bfseries H}\arabic*)]
 \item If $\Var (\rv d)=0$ then $\gcd(\rv k)\mid  d n$.
\end{enumerate}

 We will derive contiguity results for these models for two types of $\rv m$: when $\rv m$ is distributed as a Poisson variable, and when $\rv m$ has an atomic distribution on a single value. When $\rv m\sim \Po(\mu)$ it is always assumed that 
 \[
 \mu=dn/k
 \]
 so that the probability of the event $\mc D_n$ is not too small. This assumption is maintained throughout the paper and we do not repeat it. When $\rv m$ has an atomic distribution, we write $m=m(n)$ instead of $\rv m$ to distinguish this case from the Poisson case. In this case, we impose condition ({\bf H}1) above, as well as the following conditions ({\bf H}2)--({\bf H}4) to ensure that these models (with deterministic $m$) are well defined, and that the probability of $\mc D_n$ is not too small.

\begin{enumerate}
    \item[({\bf H}2)] $m=m(n)$ is a sequence with $|m- dn/ k|=O(\sqrt{n})$. 
    \item[({\bf H}3)] If $\Var(\rv k)=0$ then $\gcd(\rv d)\mid  km$.
 \item[({\bf H}4)] If $\Var(\rv k)=\Var(\rv d)=0$ then $ d n= km$.
\end{enumerate}

\noindent {\bf Remark}. (a)
Note that for $\B$ distributed as in~\eqref{eq:system}, $F(\B)$ has the same distribution as $G_{n,m,\rv k}$ where $\rv k=k$.

(b) ~\cite{coja2020rank} assumed that $G^{(1)}_{n,m,\rv k,\rv d}$ and $G^{(2)}_{n,m,\rv k,\rv d}$ have the same distribution, which is untrue. We fix this error here by showing that these two models are asymptotically equivalent (see Theorem~\ref{thm:asympequiv} below), although being contiguous is enough for~\cite{coja2020rank}, which is given in Theorem~\ref{thm:contiguity} below.

(c) Observe that  $G^{(1)}_{n,m,\rv k,\rv d}$, $G_{n,m,\boldsymbol{k}}$, $M_{n,m,\boldsymbol{k}}$, $M^{(2)}_{n,m,\rv k,\rv d}$, and  $G^{(2)}_{n,m,\rv k,\rv d}$ are exactly the models with Poisson $\rv m$ by conditioning the respective models above on having $m$ constraint vertices.

\begin{definition}
    Let $\{\mu_n\}_{n\in\NN}$ and $\{\nu_n\}_{n\in\NN}$ be distributions on a sequence of probability spaces $\{\Omega_n\}_{n\in\NN}$. We say that $\{\mu_n\}_{n\in \NN}$ is contiguous with respect to $\{\nu_n\}_{n\in \NN}$, and write $\{\mu_n\}_{n\in \NN}\contig\{\nu_n\}_{n\in \NN}$, if $\lim_{n\to\infty}\mu_n(\mc E_n)=0$ implies $\lim_{n\to\infty}\nu_n(\mc E_n)=0$ for all sequences events $\mc E_n\subset \Omega_n$. If $\{\mu_n\}_{n\in\NN}\contig\{\nu_n\}_{n\in\NN}$ and $\{\nu_n\}_{n\in\NN}\contig\{\mu_n\}_{n\in\NN}$, we say that $\{\mu_n\}_{n\in\NN},\{\nu_n\}_{n\in\NN}$ are mutually contiguous, and write $\{\mu_n\}_{n\in\NN}\mcontig\{\nu_n\}_{n\in\NN}$.
\end{definition}

\begin{theorem}\label{thm:contiguity}
   Suppose $\rv d$ is a Poisson random variable. Then $G_{n,m,\rv k}\mcontig G^{(1)}_{n,m,\rv k,\rv d}\mcontig G^{(2)}_{n,m,\rv k,\rv d}$, and $G_{n,\rv m,\rv k}\mcontig G^{(1)}_{n,\rv m,\rv k,\rv d}\mcontig G^{(2)}_{n,\rv m,\rv k,\rv d}$. 
\end{theorem}

\begin{theorem}\label{thm:asympequiv}
    Let $\{\mc E_n\}_{n\in\NN}$ be a sequence of sets of factor graphs. Then
    \begin{equation*}
        \lim_{n\to\infty}\lt|\pr[G^{(1)}_{n,m,\rv k,\rv d}\in \mc E_n]-\pr[G^{(2)}_{n,m,\rv k,\rv d}\in \mc E_n]\rt|=0,\quad \lim_{n\to\infty}\lt|\pr[G^{(1)}_{n,\rv m,\rv k,\rv d}\in \mc E_n]-\pr[G^{(2)}_{n,\rv m,\rv k,\rv d}\in \mc E_n]\rt|=0.
    \end{equation*}
\end{theorem}

\section{Overview of proofs}\label{sec:proofoverview}

\subsection{Theorems~\ref{thm:principal-sublinear}, \ref{thm:nonprincipal-sublinear}, \ref{thm:nonprincipal-linear-characterisation}, and~\ref{thm:nonprincipal-bounds}}
To prove Theorem~\ref{thm:nonprincipal-sublinear},  we use the first and the second moment methods to determine that the threshold at which a given tree $T$ whose constraint vertices all have degree $k$ appears in $F(\B)$ is $\Theta(n^{1-1/c(T)})$, where $c(T)$ is the number of constraint vertices in $T$. 
Consequently, if $m\ll n^{\alpha_{P,Q}}$ then a.a.s.\ there will be no $T\in \mc T(P,Q)$ occuring in $F(\B)$ and thus every component of $F(\B)$ is a tree that is satisfiable for any assignment of entries of $(\B,b)$, and so
the system~\eqref{eq:system} is satisfiable. On the other hand, if $m\gg n^{\alpha_{P,Q}}$ then there will be $\omega(1)$ disjoint copies of $T\in \mc T(P,Q)$ occurring in $F(\B)$. Since the nonzero entries of $\B$ are assigned independently of $F(\B)$, independently for each occurrence of $T\in \mc T(P,Q)$ in $F(\B)$, there is a positive probability away from zero such that  $(T,w)$ is unsatisfiable, where $(T,w)$ is the weighted factor graph of $\B x=b$ induced by the set of vertices in $T$. This implies that a.a.s.\ the system~\eqref{eq:system} is unsatisfiable.

By Theorem~\ref{thm:nonprincipal-sublinear}, it suffices to estimate $\inf\{c(T):T\in \mc T(P,Q)\}$ to prove Theorem~\ref{thm:nonprincipal-bounds}.  We inductively define a pair of rooted trees $T_{\ell(P,Q)},\ T^+_{\ell(P,Q)}$, both of height $2\ell(P,Q)+1$, such that any $T\in \mc T(P,Q)$ contains $T^+_{\ell(P,Q)}$ (Lemma~\ref{lem:sublinearSATcondition}), and $(T_{\ell(P,Q)}, w)$ is unsatisfiable for some weight function $w$ (Lemma~\ref{lem:Tobstruction}).  Consequently,
$c(T^+_{\ell(P,Q)})\le   \inf\{c(T):T\in \mc T(P,Q)\}\le c(T_{\ell(P,Q)}),
$ which yields Theorem~\ref{thm:nonprincipal-bounds}.
The appropriate weight function $w$ on $T_{\ell(P,Q)}$ is derived from $Q$ and the sequence of ideals $I_{0}(P),I_1(P),\ldots, I_{\ell(P,Q)}(P)$.

Unfortunately, $T_{\ell(P,Q)}$ and $T^+_{\ell(P,Q)}$ are not the same size, and therefore do not give tight bounds on $\inf\{c(T):T\in \mc T(P,Q)\}$ and $\alpha_{P,Q}$. 
However, for Theorem~\ref{thm:nonprincipal-linear-characterisation}, it suffices to show that $\mc T(P,Q)\neq \emptyset$ if and only if $\ell(P,Q)=\infty$, which immediately follows from Theorem~\ref{thm:nonprincipal-bounds}.

Finally, we prove Theorem~\ref{thm:principal-sublinear} by finding the smallest obstruction in $\mc T(P,Q)$, which gives the SAT threshold by Theorem~\ref{thm:nonprincipal-sublinear}. We begin by defining a second inductive sequence of ideals $I^+_0(P), I_1^+(P),\ldots$ similar to $I_0(P), I_1(P),\ldots$ (Definition~\ref{def:I^+}), which agrees with $I_0(P),I_1(P),\ldots$ when $R$ is a principal ideal ring (Lemma~\ref{lem:I-I+}). Letting $\ell^+(P,Q)=\inf\{\ell\geq 0: I_\ell(P)\not\subset \supp(Q)\}$, we use the sequence $I_0(P),I_\ell(P)\ldots$ to derive a weight function $w$ on $T^+_{\ell^+(P,Q)}$ such that $(T^+_{\ell^+(P,Q)},w)$ is unsatisfiable (Lemma~\ref{lem:T^+obstruction}). Since $\ell^+(P,Q)=\ell(P,Q)$ when $R$ is principal, this shows that the bound on $\alpha_{P,Q}$ given by Lemma~\ref{lem:sublinearSATcondition} is tight, and proves Theorem~\ref{thm:principal-sublinear}.

\subsection{Theorem~\ref{thm:principal}}
We first prove a special case of Theorem~\ref{thm:principal}, which is given as Theorem~\ref{thm:units} below, and then give an overview of how to reduce Theorem~\ref{thm:principal} to this special case. 
It is easy to show that $\ell(P,Q)=\infty$ and $\alpha_{P,Q}=1$ when $P$ is supported on the units of $R$.

\begin{theorem}\label{thm:units}
    Suppose $P,Q$ are distributions on a local ring $R$, and let $\hat P=P^{\otimes k}$. Suppose $(\hat P,Q)$ satisfies {\bf (A2)} and $P$ is supported on the units of $R$. Then, for fixed $\eps>0$, we have
    \begin{equation*}
        \lim_{n\to\infty}P[\text{System~\eqref{eq:system} is satisfiable}]=\begin{cases}
            1&m<(d_k-\eps)n/k\\
            0&m>(d_k+\eps)n/k.
        \end{cases}
    \end{equation*}
\end{theorem}

\proof Let $M$ be the unique maximal ideal of $R$. Let $\bar \B$ and $\bar \b$ be the matrix and vector obtained from $
\B$ and $\b$ by taking entry-wise projections on to the residue filed $R/M$. By the assumption that $P$ is supported on the units of $R$, every nonzero element of $\B$ is mapped to a nonzero field element in $R/M$ by the projection.
Let $C$ and $V$ denote the set of rows and columns of $\B$.
Given $C'\subset C$ and $V'\subseteq V$, let $\B{[C',V']}$ denote the submatrix of $\B$ induced by rows in $C'$ and coloumns in $V'$.

In the subcritical case $m<(d_k-\eps)n/k$, by following the identical to the argument in~\cite{ayre2020satisfiability}, we can show that there exists a subset $V'\subseteq V$ with $|V'|=|C|$ such that $\bar\B[C,V']$ is invertible over $R/M$. Hence, $\det(\bar\B[C,V'])\neq 0$ in $R/M$, which implies that $\det(\B[C,V'])$ is a unit in $R$ since $R$ is local. Therefore, $\B[C,V']$ is invertible over $R$, and so $\B x=\b$ is a.a.s.\ satisfiable.

The proof for the supercritical case $m>(d_k+\eps)n/k$ follows closely that of~\cite[Theorem 1.5]{ayre2020satisfiability}, and so we only briefly outline the main ideas here. Start with a system $\B'x=\b'$ of $m'=(d_k-\eps_n)n/k$ equations for some $\eps_n\to 0$ sufficiently slowly, with rows in $\B'$ and entries in $\b'$ drawn from the same distribution as in $\B x=\b$. Then, consider the set of solutions $\Sigma$ of the subsystem $\B'_{\text{2-core}}x=\b'_{\text{2-core}}$, where $\B'_{\text{2-core}}$ is the maximum submatrix of $\B'$ where every column contains at least 2 non-zero entries and every row contains $k$ non-zero entries. We can show that $|\Sigma|=|R|^{O(\eps_nn)}$ by the same argument as Claim 2.2 in~\cite{gao2026satisfiabilitythresholdsolutionspace}, using the fact that any invertible submatrix of $\bar \B'$ (over $R/M$) is also an invertible submatrix of $\B'$ (over $R$). Then, we add $m-m'>\eps n/k$ equations to $\B'x=\b'$ to get a system with the same distribution as $\B x=\b$. Let $R'$ be the subset of equations that were added to $\B'$ which have non-zero entries only in the columns that are contained in the submatrix $\B'_{\text{2-core}}$. As argued in~\cite{gao2026satisfiabilitythresholdsolutionspace}, we can show that a.a.s.\  $|R'|=\Omega(n)$. Using \textbf{(A2)} and the same argument as Claim 2.3 of~\cite{gao2026satisfiabilitythresholdsolutionspace}, there is a $\delta>0$ such that, given an $x\in \Sigma$, the probability that $x$ satisfies each additional equation in $R'$ is at most $1-\delta$. Then we can show that the expected number of solutions in $\Sigma$ that satisfy all the additional $m-m'$ equations is bounded by $q^{O(\eps_nn)}(1-\delta)^{\Omega(\eps n/k)}=\exp(-\Omega(n))$. This immediately implies that a.a.s.\ $\B x= \b$ is unsatisfiable. \qed

In the first step in proving both Theorem~\ref{thm:principal}, we note that any finite ring $R$ can be written as a product of local rings (Lemma~\ref{lem:localproduct}), and therefore it suffices to prove Theorem~\ref{thm:principal} in the case where $R$ is a local ring with a unique maximal ideal $M$.
We show that $\supp(P)$ and $\supp(Q)$ have a very simple structure when $\alpha_{P,Q}=1$ (Lemma~\ref{lem:oneideal}). In particular, either every $r\in \supp(P)$ generates the same ideal, or $\supp(Q)=\{0\}$. Satisfiability is trivial if $\supp(Q)=\{0\}$. In the case where every $r\in \supp(P)$ generates the same ideal $I$, the quotient $I/IM$ is a one dimensional vector space over the residue field $R/M$, with $r\notin IM$ for all $r\in \supp(P)$. This allows us to apply Theorem~\ref{thm:units} to show that satisfiability threshold is $m\sim d_kn/k$.


\subsection{Theorem~\ref{thm:nonprincipal-linear}}
We show that system~\eqref{eq:system} is satisfiable if and only if a related system $\B_{XY}\, x=\b_{XY}$ is satisfiable, where $\B_{XY}$ is an $m\times 2n$ $\{0,1\}$-matrix over $\FF_q$ whose distribution is contiguous to the one given as follows. Independently for each row $a$, let $k^X_a$ be a random variable drawn from distribution $\Bin(k,\rho)$ and let $k_a^Y=k-k_a^X$. Then, uniformly choose $k_a^X$ entries from the first $n$ columns and $k_a^Y$ entries from the last $n$ columns. These $k$ chosen entries form the set of nonzero entries of row $a$. 

Let $d_1,\ldots,d_{2n}$ denote the number of nonzero entries in the columns of this random matrix. The distribution of $(d_1,\ldots,d_{2n})$ is approximately that of i.i.d.\ copies of the random variable $\boldsymbol{d}$ where with probability $1/2$ it is a Poisson variable with mean $\rho km/n$ and with probability $1/2$ it is a Poisson variable with mean $(1-\rho) km/n$. The threshold $d_{k,\rho}$ in this theorem corresponds exactly to the threshold of satisfiability of $\B x=\b$ if $\B$ was drawn from the model in~\cite{coja2024full} described in Section~\ref{sec:contiguity}. That is,  $(d_1,\ldots,d_{2n})$ is drawn from i.i.d.\ $\boldsymbol{d}$, and $\boldsymbol{m}$ is a Poisson variable with mean $m$, and $\boldsymbol{k}=k$.

However, the distribution of our random matrix is not exactly the model above; in particular, $d_1,\ldots,d_{2n}$ are not independent.
The main challenge in proving the satisfiablity statement of Theorem~\ref{thm:nonprincipal-linear} is to come up with a coupling scheme that allows us to compare our random linear equations with the model above with i.i.d.\ $d_i$s and conclude that if the latter is satisfiable then our system is also satisfiable. The detailed description of the coupling scheme as well as the full proof of Theorem~\ref{thm:nonprincipal-linear} is given in Section~\ref{sec:nonprincipal-linear}.

\section{Sublinear SAT thresholds}\label{sec:sublinearSAT}

\subsection{Proof of Theorems~\ref{thm:nonprincipal-sublinear}}

Let $R$ be a general finite commutative ring, and suppose that $(\hat P,\hat Q)$ satisfies the assumptions of Theorem~\ref{thm:nonprincipal-sublinear}. 
Obviously, if $F=F(B,b)$ contains an unsatisfying weighted subgraph $F'$ as a subgraph then $B x=b$ is not satisfiable. In this case we call $F'$ \textit{unsatisfiable}, and an \textit{obstruction} for the satisfiability of $B x=b$. Theorem~\ref{thm:nonprincipal-sublinear} shows that the course satisfiability threshold for the system $\B x=\b$ is determined by the size of the smallest obstruction. We will identify the smallest obstruction for principal ideal rings in section~\ref{sec:proofoverview-sublinear}, which leads to the proof of Theorem~\ref{thm:principal-sublinear}.

To prove Theorem~\ref{thm:nonprincipal-sublinear}, we need to know when any given acyclic subgraph appears in $F(\B)$ as $m$ increses. This is given by Claim~\ref{claim:treethresholds}. We also use Claims~\ref{claim:noconstantUB} and~\ref{claim:acyclicsat}, which are related to the assumptions {\bf(A1)} and {\bf(A2)}.

\begin{claim}\label{claim:treethresholds}
    Suppose $T$ is a connected acyclic factor graph, and let $X_T$ be the number of isomorphic copies of $T$ in $F(\rv B)$. If $m\ll n^{1-1/c(T)}$, then $\pr[X_T=0]=1-o(1)$. If $m=\Theta(n^{1-1/c(T)})$, then $\ex[X_T]=O(1)$ and $\pr[X_T>0]=\Omega(1)$. If $m\gg n^{1-1/c(T)}$, then $\pr[X_T=\omega(1)]=1-o(1)$.
\end{claim}
\begin{proof}
Since $T$ is a tree, we have $v(T)=(k-1)c(T)+1$. Hence
\begin{equation}
    \ex[X_T]=\frac{m!n!}{(m-c(T))!(n-v(T))!a(T)}\binom{n}{k}^{-c(T)}=\Theta(m^{c(T)}n^{v(T)-kc(T)})=\Theta(m^{c(T)}n^{-c(T)+1}).
\end{equation}
Therefore $\pr[X_T=0]=1-o(1)$ if $m=o(n^{1-1/c(T)})$. For the second moment, write
\begin{align*}
    \ex[X_T^2]&=\sum_{H\subset T} \sum _{\substack{T_1,T_2\cong T\\T_1\cap T_2\cong H}}\pr[T_1\subset F(\rv B),T_2\subset F(\rv B)]
\end{align*}
Since $T$ is acyclic, any $H\subset T$ with $c(H)>0$ satisfies $v(H)\geq (k-1)c(H)+1$. Hence, for $H\subset T$ with $c(H)>0$,
\begin{align*}
    \sum _{\substack{T_1,T_2\cong T\\T_1\cap T_2\cong H}}\pr[T_1\subset F(\rv B),T_2\subset F(\rv B)]&=O(m^{2c(T)-c(H)}n^{2v(T)-v(H)})\binom{n}{k}^{-2c(T)+c(H)}\\
    &=O(m^{2c(T)-c(H)}n^{-2c(T)+c(H)+1})\\
    &=(\ex[X_T])^2O(m^{-c(H)}n^{c(H)-1}).
\end{align*}
For $m=\Omega(n^{1-1/c(T)})$, we have $m^{-c(H)}n^{c(H)-1}=O(n^{c(H)/c(T)-1})$. Therefore
\begin{equation*}
    \ex[X_T^2]\leq (\ex[X_T])^2+\ex[X_T]+o(1)(\ex[X_T])^2
\end{equation*}
and so, for $C>0$
\begin{equation*}
    \pr[X_T\geq C\ex[X_T]]\geq (1-C)^2\frac{(\ex [X_T])^2}{\ex[X^2]}=(1-C)\frac{(\ex [X_T])^2}{(\ex [X_T])^2+\ex [X]}+o(1).
\end{equation*}
For $m=\Theta(n^{1-1/c(T)})$, we have $\ex[X]=\Theta(1)$, and so $\pr[X_T>0]=\Omega(1)$. For $m=\omega(n^{1-1/c(T)})$, we have $\ex[X_T]=\omega(1)$, and so $\pr[X_T\geq\sqrt{\ex[X_T]}]=1-o(1)$.
\end{proof}
\begin{claim}\label{claim:noconstantUB}
For any distributions $P$ and $Q$ that satisfy {\bf (A2)}, if $m\gg n$ then
\begin{equation*}
    \lim_{n\to\infty}\pr[\text{System~\eqref{eq:system} is satisfiable}]=0
\end{equation*}
\end{claim}

\begin{proof}
Let
\begin{equation*}
    \gamma=\min_{r\in R} \pr[\sum_{v\in V_n}\B_{a,v}r\neq \rv \b_a].
\end{equation*}
Since the rows of $(\rv B, \rv b)$ are i.i.d., $\gamma$ does not depend on the row $a\in C_m$. Since $P$ and $Q$ satisfy {\bf (A2)}, $\gamma>0$. Now, given $x\in R^{V_n}$, let $r_x\in R$ be such that $|\{v\in V_n: x_v=r_x\}|\geq n/|R|$. Then
    \begin{equation*}
        \pr[(\B x)_a\neq b_a]\geq\gamma \pr[ N(a)\subset \{v\in V_n:x_v=r_x\}]\geq (1-o(1))\gamma|R|^{-k}
    \end{equation*}
    Since the rows of $(\rv B,\rv b)$ are independent, the expected number of solutions to $\rv Bx=\rv b$ is therefore at most $(1-\gamma |R|^{-k}+o(1))^m|R|^n$. If $m\gg n$, we get $(1-\gamma|R|^{-k}+o(1))^m|R|^n=o(1)$, and therefore~\eqref{eq:system} is unsatisfiable a.a.s.
\end{proof}
\begin{claim}\label{claim:acyclicsat}
    For any distributions $P$ and $Q$ that satisfy {\bf (A1)} and any acyclic factor graph $F$, there is a weighting $w$ with edge weights in $\supp(P)$ and constraint-vertex weights in $\supp(Q)$ such that $(F,w)$ is satisfiable.
\end{claim}
\begin{proof}
    Suppose $P$ and $Q$ satisfy {\bf (A1)}, and let $r_1,\ldots, r_k\in \supp(P)$ and $b\in \supp(Q)$ be such that $b$ is in the generated ideal $(r_1,\ldots, r_k)$. We prove the claim for connected $F$ by induction on $|C(F)|$.

    For $|C(F)|=1$, assign weight $b$ to the one constraint vertex in $C(F)$, and assign weights $r_1,\ldots, r_k$ to its $k$ incident edges. Since $b\in (r_1,\ldots, r_k)$, $(F,w)$ is satisfiable.  

    Now suppose $|C(F)|>0$. Since $F$ is a tree, there is an $a\in C(F)$ such that $N_F(a)=\{v_1,\ldots, v_k\}$ where $\deg_F(v_1)\geq 2$ and $\deg_F(v_i)=1$ for $i=2,\ldots, k$. Let $F'$ be the factor graph obtained by removing $a$ and $v_2,\ldots v_k$ from $F$. By the induction hypothesis, there is weighting $w'$ of $F'$ such that $(F',w)$ is satisfiable. Let $a'\in C(F')$ be a constraint vertex adjacent to $v_1$ in $F'$, and let $v_2',\ldots, v_k'\in V(F')$ be such that $N_{F'}(a')=\{v_1,v_2',\ldots, v_k'\}$. Extend $w'$ to a weighting $w$ of $F$ by setting $w(a)=w(a')$ and $w(av_1)=w(a'v_1)$ and $w(av_i)=w(av_i')$ for $i=2,\ldots, k$. Then any satisfying assignment $x$ of $(F',w)$ can be extended to a satisfying assignment of $(F,w)$ by setting $x_{v_i}=x_{v_i'}$ for $i=2,\ldots, k$. Therefore $(F,w)$ is satisfiable.
\end{proof}
\begin{proof}[Proof of Theorem~\ref{thm:nonprincipal-sublinear}]
    First consider the case where $\alpha_{P,Q}=1$. If $m\ll n^{\alpha_{P,Q}}$, then the first moment method shows that $F(\B)$ is acyclic. Since $\mc T(P,Q)=\emptyset$, system~\eqref{eq:system} is therefore satisfiable a.a.s. If $m\gg n$, then Claim~\ref{claim:noconstantUB} shows that~\eqref{eq:system} is unsatisfiable a.a.s.

    Now suppose that $\alpha_{P,Q}<1$, and let
    \begin{equation*}
        c^*=\inf\{|C(T)|:T\in \mc T(P,Q)\}, \quad \mc T^*(P,Q)=\{T\in \mc T(P,Q):c(T)=c^*\}.
    \end{equation*}
    Suppose $m=\Omega(n^{\alpha_{P,Q}})$. Then Claim~\ref{claim:treethresholds} shows that $F(\B)$ does not contain any tree with more than $c^*$ constraint vertices a.a.s., and there is an $M>0$ such that
    \begin{equation*}
        \pr[|\{T\subset F(\B):c(T)=c^*\}|\leq M]=\Omega(1).
    \end{equation*}
    By Claim~\ref{claim:acyclicsat} each unweighted tree in $\mc T^*(P,Q)$ is satisfiable for some assignment of edge weights in $\supp(P)$ and constraint vertex weights in $\supp(Q)$. Therefore, conditioned on $\{|\{T\subset F(\B):c(T)=c^*\}|\leq M\}$ the probability that all trees in $\{T\subset F(\B):c(T)=c^*\}$ are satisfiable is positive. Therefore~\eqref{eq:system} is satisfiable with probability $\Omega(1)$. On the other hand, if $m\ll n^{\alpha_{P,Q}}$, then Claim~\ref{claim:treethresholds} shows that $F(\B)$ does not contain any trees with $c^*$ constraint vertices a.a.s., and therefore~\eqref{eq:system} is satisfiable a.a.s.
    
    Now suppose $m=\Omega(m^{\alpha_{P,Q}})$. Then Claim~\ref{claim:treethresholds} shows that $F(\B)$ contains an unweighted graph in $\mc T(P,Q)$ with probability $\Omega(1)$. This graph is unsatisfiable for some assignment of weights, so~\ref{eq:system} is unsatisfiable with probability $\Omega(1)$. If $m\gg n$, then Claim~\ref{claim:treethresholds} shows that the number of unweighted trees in $\mc T^*(P,Q)$ contained in $F(\rv B)$ is unbounded a.a.s. Since the assignment of weights to each of these trees is independent, we get that~\eqref{eq:system} is unsatisfiable a.a.s.
\end{proof}

\subsection{Proofs of Theorems~\ref{thm:principal-sublinear},~\ref{thm:nonprincipal-linear-characterisation}, and~\ref{thm:nonprincipal-bounds}: an overview}\label{sec:proofoverview-sublinear}

By Theorem~\ref{thm:nonprincipal-sublinear}, it suffices to show that 
if $R$ is principal and $\mc T(P,Q)\neq \emptyset$ then 
\begin{equation}
\inf\lt\{|C(T)|:T\in\mc T(P,Q)\rt\}=1+\sum_{i=1}^{\ell(P,Q)}k(k-1)^{i-1}. \label{eq:tree-bounds}
\end{equation}

Suppose that $T_{\min}$ is a weighted tree factor graph in $\mc T(P,Q)$ which attains the infimum of the left hand side of~\eqref{eq:tree-bounds}. We seek tight upper and lower bounds for $|C(T_{\min})|$ by identifying tree factors that are members of $\mc T(P,Q)$. This motivates the sequences of unweighted tree factors $(T_i)$ and $(T_i^+)$ defined as follows.

\begin{definition}
    We define two sequences of graphs $T_0,T_1,\ldots$ and $T^+_0,T^+_1,\ldots$ inductively as follows. Let $T_0$ and $T^+_0$ be the graph with one constraint vertex and $k$ adjacent variable vertices. Let $p=|\supp P|$. Let $L(T)$ denote the set of leaves of a tree $T$.
    
    For $\ell\geq 1$, construct $T_\ell$ from $T_{\ell-1}$ by adding $p^2$ constraint vertices $a_{v,1},\ldots, a_{v,p^2}$ adjacent to each leaf $v\in L(T_{\ell-1})$, and adding $k-1$ new variable vertices $u_{a_{v,i},1},\ldots, u_{a_{v,i},k-1}$ adjacent to each $a_{v,i}$ for $v\in L(T_{\ell-1}),i\in [p^2]$.

    For $\ell\geq 1$, construct $T^+_\ell$ from $T^+_{\ell-1}$ by adding a constraint vertex $a_v$ adjacent to each leaf $v\in L(T^+_{\ell-1})$, and adding $k-1$ new variable vertices $u_{a_{v},1},\ldots, u_{a_{v},k-1}$ adjacent to $a_{v}$ for each $v\in L(T^+_{\ell-1})$.
\end{definition}

As an illustration, $T_0=T_0^+$ is given in Figure~\ref{fig:T_0} below for the case $k=3$ (ignoring the weight $b$ on the constraint vertices, drawn as a square, and the weights $r$ on the edges. The variable vertices are drawn as small disks). Figures~\ref{fig:T_1^+} and~\ref{fig:T_1} illustrate how to obtain $T_1^+$ and $T_1$ from $T_0^+=T_0$ (again ignoring the weights).

\begin{figure}[!ht]
    \centering
\begin{minipage}[b]{0.4\textwidth}
    \centering
    \begin{tikzpicture}[scale = 2, every edge quotes/.style={fill=white,font=\footnotesize}, every label quotes/.style={above,font=\footnotesize}]
            \tikzstyle{var} = [circle, minimum width = 3mm, fill, inner sep = 0pt]
            \tikzstyle{con} = [minimum width = 3mm, minimum height = 3mm, fill, inner sep = 0pt]

            \node[con, "$b$"] (r) at (0,0){};
            
            \node[var] (1) at (1,1){};
            \node[var] (0) at (1,0){};
            \node[var] (-1) at (1,-1){};

            \draw (r) edge ["$r$"] (1){};
            \draw (r) edge ["$r$"] (0){};
            \draw (r) edge ["$r$"] (-1){};
        \end{tikzpicture}
    \caption{$T_0=T_0^+$}
    \label{fig:T_0}

\vspace{1cm}

    \centering
     \begin{tikzpicture}[scale = 2, every edge quotes/.style={fill=white,font=\footnotesize}, every label quotes/.style={above,font=\footnotesize}]
            \tikzstyle{var} = [circle, minimum width = 3mm, fill, inner sep = 0pt]
            \tikzstyle{con} = [minimum width = 3mm, minimum height = 3mm, fill, inner sep = 0pt]

            \node[con, "$b_1$"] (r) at (0,0){};
            
            \node[var] (1) at (1,1){};
            \node[var] (0) at (1,0){};
            \node[var] (-1) at (1,-1){};
            
            \node[con, "$b_2$"] (10) at (2,1){};
            \node[con, "$b_2$"] (00) at (2,0){};
            \node[con, "$b_2$"] (-10) at (2,-1){};
        
            \node[var] (101) at (3,5/4){};
            \node[var] (10-1) at (3,3/4){};
            \node[var] (001) at (3,1/4){};
            \node[var] (00-1) at (3,-1/4){};
            \node[var] (-101) at (3,-3/4){};
            \node[var] (-10-1) at (3,-5/4){};

            \draw (r) edge ["$x_0$"] (1);
            \draw (r) edge ["$x_0$"] (0);
            \draw (r) edge ["$x_0$"] (-1);

            \draw (1) edge ["$x_1$"] (10);
            \draw (0) edge ["$x_1$"] (00);
            \draw (-1) edge ["$x_1$"] (-10);

            \draw (10) edge ["$x_2$"] (101);
            \draw (10) edge ["$x_2$"] (10-1);

            \draw (00) edge ["$x_2$"] (001);
            \draw (00) edge ["$x_2$"] (00-1);

            \draw (-10) edge ["$x_2$"] (-101);
            \draw (-10) edge ["$x_2$"] (-10-1);
        \end{tikzpicture}
    \caption{$T_1^+$}
    \label{fig:T_1^+}
    \end{minipage}
\hfill
\begin{minipage}[b]{0.4\textwidth}
    \centering
    \begin{tikzpicture}[scale = 2, every edge quotes/.style={fill=white,font=\footnotesize}, every label quotes/.style={above,font=\footnotesize}]
            \tikzstyle{var} = [circle, minimum width = 3mm, fill, inner sep = 0pt]
            \tikzstyle{con} = [minimum width = 3mm, minimum height = 3mm, fill, inner sep = 0pt]

            \node[con, "$b_1$"] (r) at (0,0){};
            
            \node[var] (1) at (1,2){};
            \node[var] (0) at (1,0){};
            \node[var] (-1) at (1,-2){};
            
            \node[con, "$b_1$"] (12) at (2,11/4){};
            \node[con, "$b_1$"] (11) at (2,9/4){};
            \node[con, "$b_1$"] (1-1) at (2,7/4){};
            \node[con, "$b_1$"] (1-2) at (2,5/4){};

            \node[con, "$b_1$"] (02) at (2,3/4){};
            \node[con, "$b_1$"] (01) at (2,1/4){};
            \node[con, "$b_1$"] (0-1) at (2,-1/4){};
            \node[con, "$b_1$"] (0-2) at (2,-3/4){};

            \node[con, "$b_1$"] (-12) at (2,-5/4){};
            \node[con, "$b_1$"] (-11) at (2,-7/4){};
            \node[con, "$b_1$"] (-1-1) at (2,-9/4){};
            \node[con, "$b_1$"] (-1-2) at (2,-11/4){};
            
            \node[var] (121) at (3,11/4+1/9){};
            \node[var] (12-1) at (3,11/4-1/9){};
            \node[var] (111) at (3,9/4+1/9){};
            \node[var] (11-1) at (3,9/4-1/9){};
            \node[var] (1-11) at (3,7/4+1/9){};
            \node[var] (1-1-1) at (3,7/4-1/9){};
            \node[var] (1-21) at (3,5/4+1/9){};
            \node[var] (1-2-1) at (3,5/4-1/9){};

            \node[var] (021) at (3,3/4+1/9){};
            \node[var] (02-1) at (3,3/4-1/9){};
            \node[var] (011) at (3,1/4+1/9){};
            \node[var] (01-1) at (3,1/4-1/9){};
            \node[var] (0-11) at (3,-1/4+1/9){};
            \node[var] (0-1-1) at (3,-1/4-1/9){};
            \node[var] (0-21) at (3,-3/4+1/9){};
            \node[var] (0-2-1) at (3,-3/4-1/9){};

            \node[var] (-121) at (3,-5/4+1/9){};
            \node[var] (-12-1) at (3,-5/4-1/9){};
            \node[var] (-111) at (3,-7/4+1/9){};
            \node[var] (-11-1) at (3,-7/4-1/9){};
            \node[var] (-1-11) at (3,-9/4+1/9){};
            \node[var] (-1-1-1) at (3,-9/4-1/9){};
            \node[var] (-1-21) at (3,-11/4+1/9){};
            \node[var] (-1-2-1) at (3,-11/4-1/9){};

            \draw (r) edge ["$r_1$"] (1);
            \draw (r) edge ["$r_1$"] (0);
            \draw (r) edge ["$r_1$"] (-1);
            
            \draw (1) edge ["$r_1$"] (12);
            \draw (1) edge ["$r_1$"] (11);
            \draw (1) edge ["$r_2$"] (1-1);
            \draw (1) edge ["$r_2$"] (1-2);

            \draw (0) edge ["$r_1$"] (02);
            \draw (0) edge ["$r_1$"] (01);
            \draw (0) edge ["$r_2$"] (0-1);
            \draw (0) edge ["$r_2$"] (0-2);

            \draw (-1) edge ["$r_1$"] (-12);
            \draw (-1) edge ["$r_1$"] (-11);
            \draw (-1) edge ["$r_2$"] (-1-1);
            \draw (-1) edge ["$r_2$"] (-1-2);

            \draw (12) edge ["$r_1$" near end] (121);
            \draw (12) edge ["$r_1$" near end] (12-1);
            \draw (11) edge ["$r_2$" near end] (111);
            \draw (11) edge ["$r_2$" near end] (11-1);
            \draw (1-1) edge ["$r_1$" near end] (1-11);
            \draw (1-1) edge ["$r_1$" near end] (1-1-1);
            \draw (1-2) edge ["$r_2$" near end] (1-21);
            \draw (1-2) edge ["$r_2$" near end] (1-2-1);

            \draw (02) edge ["$r_1$" near end] (021);
            \draw (02) edge ["$r_1$" near end] (02-1);
            \draw (01) edge ["$r_2$" near end] (011);
            \draw (01) edge ["$r_2$" near end] (01-1);
            \draw (0-1) edge ["$r_1$" near end] (0-11);
            \draw (0-1) edge ["$r_1$" near end] (0-1-1);
            \draw (0-2) edge ["$r_2$" near end] (0-21);
            \draw (0-2) edge ["$r_2$" near end] (0-2-1);

            \draw (-12) edge ["$r_1$" near end] (-121);
            \draw (-12) edge ["$r_1$" near end] (-12-1);
            \draw (-11) edge ["$r_2$" near end] (-111);
            \draw (-11) edge ["$r_2$" near end] (-11-1);
            \draw (-1-1) edge ["$r_1$" near end] (-1-11);
            \draw (-1-1) edge ["$r_1$" near end] (-1-1-1);
            \draw (-1-2) edge ["$r_2$" near end] (-1-21);
            \draw (-1-2) edge ["$r_2$" near end] (-1-2-1);

        \end{tikzpicture}
    \caption{$T_1$}
    \label{fig:T_1}
    \end{minipage}
\end{figure}

We briefly illustrate how to assign weights to $T_0=T_0^+$, $T_1^+$ and $T_1$, and their relations to obstructions $T\in \mc T(P,Q)$.
Suppose that $\supp Q \not\subseteq I_0(P)$ then there exists $(b,r)\in (\supp Q,\supp P)$ such that
$b\notin (r)$. Assign weights to $T_0$ as demonstrated in Figure~\ref{fig:T_0}. Obviously, this weighted $T_0$ is an obstruction for~\eqref{eq:system}. It is easy to see if $m$ is some constant, i.e.\ for $m\approx n^{0}$, $T_0$ starts to appear in $F(\B,\b)$ with positive probability. Slightly less trivially, suppose that $\supp (Q) \subseteq I_0(P)$, but $\supp(Q)\not\subset I_1(P)$. Then $b\in (r)$ for all $b\in \supp(Q)$ and $r\in \supp(P)$, so $T_0$ is satisfiable for all weights $r$ and $b$. However, there is a $b_1\in \supp(Q)$ such that 
\begin{equation*}
    b_1\notin I_1(P)=\bigcap_{r\in \supp(P)}r\bigcap_{s\in \supp(P)}\{t\in R: st\in I_0(P)\}.
\end{equation*}
Therefore there is an $r_1\in\supp(P)$ such that
\begin{equation}\label{eq:lPQ=1}
    b_1\notin r_1\bigcap_{s\in \supp(P)}\{t\in R: st\in I_0(P)\}=r_1\bigcap_{s,r\in\supp(P)}\{t\in R:st\in (r)\}.
\end{equation}
We weight $T_1$ by first weighting $T_0$ with $b_1$ and $r_1$, and then weighting the $p^2$ branches added to each leaf of $T_1$ with a distinct pair $s,r\in\supp(P)$. The constraint vertices added to $T_1$ are weighted arbitrarily. Figure~\ref{fig:T_1} shows the weighted graph $T_1$ with $\supp(P)=\{r_1,r_2\}$ and with the added constraint vertices assigned weight $b_1$. To see that this weighted graph in an obstruction for~\eqref{eq:system}, suppose a variable assignment $x$ satisfies all twelve of the constraints on the right in figure~\ref{fig:T_1}. Since $b_1\in \bigcap_{r\in\supp(P)}(r)$, $x$ must assign a value in $\bigcap_{s,r\in\supp(P)}\{t\in R:st\in (r)\}$ to each of the three left variables in Figure~\ref{fig:T_1}. Since $b_1$ satisfies~\eqref{eq:lPQ=1}, this implies that $x$ does not satisfy the left most constraint in Figure~\ref{fig:T_1}.

We will generalize the obstruction described above to yield the following lemma. 
\begin{lemma}\label{lem:Tobstruction}
    Suppose $\ell(P,Q)<\infty$. Then there is an assignment of weights $$w:(E(T_{\ell(P,Q)}),V(T_{\ell(P,Q)}))\to (\supp(P),\supp(Q))$$such that the factor graph $(T_{\ell(P,Q)},w)$ is unsatisfiable.
\end{lemma}

This shows that $\mc T(P,Q)\neq \emptyset$ when $\ell(P,Q)<\infty$, and therefore gives one direction of Theorem~\ref{thm:nonprincipal-linear-characterisation} when combined with Theorem~\ref{thm:nonprincipal-sublinear}. However, obstructions of this form do not generally give a tight bound on the infimum in~\eqref{eq:tree-bounds}. In order to chase a tight bound, we utilise $T^+_{\ell}$ instead of $T_\ell$. By assigning proper weights to $T^+_{\ell}$, we will obtain a tight bound in the case where $R$ is a principal ideal ring.

We begin with introducing the sequence of ideals $I_{\ell}^+(P)$ that are defined in a similar manner as $I_{\ell}(P)$ in Definition~\ref{def:I}. These ideals are related to the weights we will assign to $T^+_{\ell}$. These two sequences of ideals are in general different--see an illustration in Example~\ref{example:nonprincipalI,I+} below--but they agree when $R$ is a principal ideal ring.

\begin{definition}\label{def:I^+}
    Given $x\in R^\NN$, define a sequence of ideals $I_0^+(x),I_1^+(x),\ldots$ inductively by setting $I^+_0(x)=x_0R$ and
    \begin{equation*}
        I_\ell^+(x)=\{x_{2\ell}t:x_{2\ell-1}t\in I_{\ell-1}^+(x),t\in R\}
    \end{equation*}
    for each positive integer $\ell$. Define
    \begin{equation*}
        I_\ell^+(P)=\bigcap_{x\in \supp(P)^\NN}I_\ell^+(x)
    \end{equation*}
    and
    \begin{equation*}
        \ell^+(P,Q)=\inf\{\ell\geq 0: \supp(Q)\not\subset I_\ell^+(P)\}.
    \end{equation*}
\end{definition}
\begin{lemma}\label{lem:I-I+}
    For any ring $R$ and distribution $P$ on $R$, $I_\ell(P)\subset I_\ell^+(P)$ for all $\ell\geq 0$. If $R$ is a principal ring, then $I_\ell(P)=I^+_\ell(P)$ for all $\ell\geq 0$.
\end{lemma}

The lemma below, an analog of Lemma~\ref{lem:Tobstruction}, shows that $(T^+_{\ell^+(P,Q)},w)$ is an obstruction, and thus serves to establish another upper bound for the left hand side of~\eqref{eq:tree-bounds}.

\begin{lemma}\label{lem:T^+obstruction}
    Suppose $\ell^+(P,Q)<\infty$. Then there is an assignment of weights $$w:(E(T^+_{\ell^+(P,Q)}), V(T^+_{\ell^+(P,Q)}))\to (\supp(P),\supp(Q))$$ such that the factor graph $(T^+_{\ell^+(P,Q)},w)$ is unsatisfiable.
\end{lemma}

The lemma below allows us to establish the lower bound for the left hand side of~\eqref{eq:tree-bounds} that we desire.

\begin{lemma}\label{lem:sublinearSATcondition}
    Suppose $F$ is an acyclic factor graph, and $w$ is a weight function on $F$ with edge weights in $\supp(P)$ and constraint-vertex weights in $\supp(Q)$. Suppose that either $\ell(P,Q)=\infty$, or that $\ell(P,Q)<\infty$ and $F$ does not contain $T^+_{\ell(P,Q)}$ as a subgraph. Then $(F,w)$ is satisfiable.
\end{lemma}

Note that Lemmas~\ref{lem:sublinearSATcondition} and~\ref{lem:T^+obstruction} hold for general finite commutative rings. They establish bounds on the left side of~\eqref{eq:tree-bounds} for general rings, but these bounds agree only if $\ell(P,Q)=\ell^+(P,Q)$, which is guaranteed to hold for principal ideal rings by Lemma~\ref{lem:I-I+}.

We complete the proof for Theorems~\ref{thm:nonprincipal-linear-characterisation},~\ref{thm:nonprincipal-bounds} and~\ref{thm:principal-sublinear} assuming Lemmas~\ref{lem:Tobstruction}, \ref{lem:I-I+}, \ref{lem:T^+obstruction}, and~\ref{lem:sublinearSATcondition}, whose proof will be delayed to the next section. 

\begin{proof}[Proof of Theorem~\ref{thm:nonprincipal-linear-characterisation}]
By definition of $\alpha_{P,Q}$, we have $\alpha_{P,Q}=1$ if and only if $\mc T(P,Q)=\emptyset$. If $\ell(P,Q)=\infty$, then $\mc T(P,Q)=\emptyset$ by Lemma~\ref{lem:sublinearSATcondition}. Conversely, if $\ell(P,Q)<\infty$, then $\mc T(P,Q)\neq \emptyset$ by Lemma~\ref{lem:Tobstruction}. Therefore $\mc T(P,Q)=\emptyset$ if and only if $\ell(P,Q)=\infty$.
\end{proof}
\begin{proof}[Proof of Theorem~\ref{thm:nonprincipal-bounds}]
Let $p=|\supp(P)|$ and suppose $\ell(P,Q)<\infty$. For any $T\in \mc T(P,Q)$ Lemma~\ref{lem:sublinearSATcondition} shows that $T$ contains $T^+_{\ell(P,Q)}$ as a subgraph. Thus
\begin{equation*}
    \inf\{|C(T)|:T\in \mc T(P,Q)\}\geq |C(T^+_{\ell(P,Q)})|
\end{equation*}
By definition of $T^+_{\ell(P,Q)}$, we have
\begin{equation*}
|C(T^+_{\ell(P,Q)})|=1+\sum_{i=1}^{\ell(P,Q)}k(k-1)^{i-1},
\end{equation*}
and hence
\begin{equation*}
    \alpha_{P,Q}=1-\frac{1}{\inf\{|C(T)|:T\in \mc T(P,Q)\}}\geq \frac{\sum_{i=1}^{\ell(P,Q)}k(k-1)^{i-1}}{1+\sum_{i=1}^{\ell(P,Q)}k(k-1)^{i-1}}.
\end{equation*}
This proves the lower bound in Theorem~\ref{thm:nonprincipal-bounds}. For the upper bound in Theorem~\ref{thm:nonprincipal-bounds}, Lemma~\ref{lem:Tobstruction} shows that $(T_\ell(P,Q),w)\in \mc T(P,Q)$ for some weight function $w$. Hence
\begin{equation*}
\inf\{|C(T)|:T\in \mc T(P,Q)\}\leq |C(T_{\ell(P,Q)})|.    
\end{equation*}
By definition of $T_{\ell(P,Q)}$, we have
\begin{equation*}
    |C(T_{\ell(P,Q)})|=1+\sum_{i=1}^{\ell(P,Q)}k(k-1)^{i-1}p^{2i},
\end{equation*}
and hence
\begin{equation*}
\alpha_{P,Q}\leq\frac{\sum_{i=1}^{\ell(P,Q)}k(k-1)^{i-1}p^{2i}}{1+\sum_{i=1}^{\ell(P,Q)}k(k-1)^{i-1}p^{2i}}.
\end{equation*}
\end{proof}
\begin{proof}[Proof of Theorem~\ref{thm:principal-sublinear}] The case $\ell(P,Q)=\infty$ follows immediately from Theorems~\ref{thm:nonprincipal-sublinear} and~\ref{thm:nonprincipal-linear-characterisation}.
Suppose that $\ell(P,Q)<\infty$. By Theorem~\ref{thm:nonprincipal-linear-characterisation}, $\mc T(P,Q)\neq \emptyset$. Since $R$ is principal, Lemma~\ref{lem:I-I+} shows that $\ell^+(P,Q)=\ell(P,Q)<\infty$. By the definition of $(T_{\ell}^+)_{\ell\ge 0}$, we have
\[
|C(T^+_{\ell(P,Q)})|=1+\sum_{i=1}^{\ell(P,Q)}k(k-1)^{i-1}.
\]
In light of~\eqref{eq:tree-bounds}, it suffices to show that
\begin{equation*}
    \inf\{|C(T)|:T\in \mc T(P,Q)\}=|C(T^+_{\ell(P,Q)})|.
\end{equation*}
By Lemma~\ref{lem:T^+obstruction}, there is a weight function $w$ such that $(T^+_{\ell(P,Q)},w)$ is an obstruction to~\eqref{eq:system}, which implies that $\inf\{|C(T)|:T\in \mc T(P,Q)\}\leq|C(T^+_{\ell(P,Q)})|$. For the reverse inequality, suppose that $(T,w)\in \mc T(P,Q)$. By Lemma~\ref{lem:sublinearSATcondition}, $T$ contains $T^+_{\ell(P,Q)}$ as a subgraph, and in particular $|C(T)|\geq|C(T^+_{\ell(P,Q)})|$. Therefore $\inf\{|C(T)|:T\in \mc T(P,Q)\}\geq|C(T^+_{\ell(P,Q)})|$. This completes the verification of~\eqref{eq:tree-bounds} and thus confirms Theorem~\ref{thm:principal-sublinear}.
\end{proof}

We close this section with an example that demonstrates that $I_\ell(P)$ is in general not equal to $ I_\ell^+(P)$. 
\begin{example}\label{example:nonprincipalI,I+}
    Let $R=\FF_q[X,Y,Z]/(XY-ZY,XY-XZ,X^3,Y^3,Z^3)$. Let $P$ be such that $\supp(P)=\{X,Y,Z\}$. Then
    \begin{gather*}
    I_0(P)=(XY),\quad I_\ell(P)=(XYZ),\quad \ell\geq 1,\\
    I_\ell^+(P)=(XY),\quad \ell\geq 0.
    \end{gather*}
    Consequently, if $\supp(Q)\subset(XY)$ and $\supp(Q)\not\subset (XYZ)$, then $\ell(P,Q)=1$ and $\ell^+(P,Q)=\infty$.
\end{example}
\begin{proof}
We have
    \begin{equation*}
        I_0(P)=(X)\cap (Y)\cap (Z)=(XY)
    \end{equation*}
and
    \begin{equation*}
        I_1(P)=(XA_1(P))\cap (YA_1(P))\cap (ZA_1(P)),\quad A_1(P)=\{t\in R:Xt,Yt,Zt\in (XY)\}=(XY,X^2,Y^2,Z^2).
    \end{equation*}
    Therefore $I_1(P)=(XYZ)$.
    
    \begin{equation*}
        I_2(P)=(XA_2(P))\cap (YA_2(P))\cap (ZA_2(P)),\quad A_2(P)=\{t\in R:Xt,Yt,Zt\in (XYZ)\}=(XY,X^2,Y^2,Z^2)
    \end{equation*}
    Therefore $I_\ell(P)=(XYZ)$ for $\ell\geq 1$.
    
    Now consider $I_\ell(P)$. Given $x\in \supp(P)^\NN$, we show that $I^+_\ell(x)\supset (XY)$ by induction on $\ell\geq 0$. For $\ell=0$, we have $I_0^+(x)=(x_0)\supset (XY)$. Supposing that $\ell>0$ and $I_{\ell-1}(x)\supset (XY)$, we have
    \begin{equation*}
        I_\ell^+(x)\supset \{x_{2\ell}t:x_{2\ell-1}t\in (XY)\}.
    \end{equation*}
    Since $\supp(P)=\{X,Y,Z\}$, there is a $y\in \supp(P)\setminus \{x_{2,\ell},x_{2\ell-1}\}$. Then $x_{2\ell-1}y\in (XY)$, so $I_\ell^+(x)\supset (x_{2\ell}y)=(XY)$. This completes the induction step. Taking an intersection over $x\in \supp(P)^\NN$ then gives $I_\ell^+(P)\supset (XY)$ for all $\ell\geq 0$. For the reverse inclusion, we have $I_0^+(P)=(X)\cap (Y)\cap (Z)=(XY)$, so $I_\ell^+(P)\subset (XY)$ for all $\ell\geq 0$ by Observation~\ref{obs:I^-monotone}. Therefore $I^+_\ell(P)=(XY)$ for $\ell\geq 0$.
\end{proof}

\subsection{Proofs of Lemmas~\ref{lem:Tobstruction},~\ref{lem:sublinearSATcondition},~\ref{lem:T^+obstruction}, and~\ref{lem:I-I+}}

We begin by proving Lemma~\ref{lem:I-I+} and some related facts about the sequences $\{I_\ell(P)\}_{\ell\geq 0},\{I_\ell^+(P)\}_{\ell\geq 0}$. We will need the following basic facts about finite commutative rings, which will also be useful in section~\ref{sec:linearprincipal}.

\begin{definition}
    A commutative ring $R$ is local if it has a unique maximal ideal.
\end{definition}
\begin{lemma}\label{lem:localproduct}
    Let $R$ be a finite commutative ring. Then there are finite commutative local rings $R_1,\ldots, R_h$ such that
    \begin{equation*}
        R \cong R_1\times R_2\times \cdots \times R_h.
    \end{equation*}
If $R$ is principal then each $R_j$ is principal.
\end{lemma}
\begin{lemma}\label{lem:rings-ideals}
Let $R$ be a finite commutative local principal ring with maximal ideal $M$. Then every ideal of $R$ is of the form $M^i$ for some $i\geq 0$. 

\end{lemma}

A proof of~\ref{lem:localproduct} can be found in~\cite{Bini2002}. The following key fact about principal rings allows us to prove the equality in Lemma~\ref{lem:I-I+}.

\begin{lemma}\label{lem:principalscommute}
    Let $\{I_a\}_{a\in C}$ be a collection of ideals of a finite commutative ring $R$, and let $r\in R$. Then
    \begin{eqnarray*}
        r\bigcap_{a\in C}I_a&\subseteq&\bigcap_{a\in C} rI_a,\\
        r\bigcap_{a\in C}I_a&=&\bigcap_{a\in C} rI_a,\quad \text{if $R$ is principal}.
    \end{eqnarray*}
\end{lemma}

\begin{proof}
    If $s\in \bigcap_{a\in C} I_a$, then $rs\in \bigcap_{a\in C} I_a$ for all $a\in C$. Therefore the first assertion holds.

    For the second assertion, by Lemma~\ref{lem:localproduct} it suffices to consider local principal rings. Suppose $R$ is a local ring with maximal ideal $M$. By Lemma~\ref{lem:rings-ideals}, each $I_a$ is of the form $M^{i}$ for some $i\geq 0$. Therefore there is a smallest ideal $I^*\in \{I_a\}_{a\in C}$ by inclusion, and $\bigcap_{a\in C} I_i=I^*$. We also have $rI^*\subset rI_a$ for all $a\in C$, and so $\bigcap_{a\in C}rI_a=rI^*$. Therefore $r\bigcap_{a\in C}I_a=\bigcap_{a\in C}rI_a$.
\end{proof}

\begin{proof}[Proof of Lemma~\ref{lem:I-I+}]
    Both claims can be proved by induction on $\ell\geq 0$. We first show that $I_\ell(P)\subset I_\ell^+(P)$ for $\ell\geq 0$ without assuming that $R$ is principal. For $\ell=0$, we have
    \begin{equation*}
    I_0(P)=\bigcap_{r\in\supp(P)}rR=\bigcap_{x\in \supp(P)^\NN}I^{+}_0(x)=I_0^+(P).
    \end{equation*}
    Suppose $\ell>0$ and that the claim holds for $\ell-1$. We have
        \begin{align}
        \bigcap_{x\in \supp(P)^\NN} I_\ell^+(x)&=\bigcap_{x\in\supp(P)^\NN}\{x_{2\ell}t:x_{2\ell-1}t\in I_{\ell-1}^+(x),t\in R\}\nonumber\\
&=\bigcap_{r\in\supp(P)}\bigcap_{s\in\supp(P)}\bigcap_{x\in \supp(P)^\NN}\{rt:st\in I^+_{\ell-1}(x),t\in R\}\nonumber\\        &\supset\bigcap_{r\in\supp(P)}r\left(\bigcap_{s\in\supp(P)}\bigcap_{x\in \supp(P)^\NN}\{t\in R:st\in I_{\ell-1}^+(x)\}\right)\label{eq:I+>I-ineq}\\
&=\bigcap_{r\in\supp(P)}r\left(\bigcap_{s\in\supp(P)}\{t\in R:st\in I_{\ell-1}^+(P)\}\right)
    \end{align}
Using the induction hypothesis,
    \begin{align}
        \bigcap_{r\in\supp(P)}r\left( \bigcap_{s\in\supp(P)}\lt\{t\in R:st\in I_{\ell-1}^+(P)\rt\} \right)&\supset \bigcap_{r\in\supp(P)}r\left( \bigcap_{s\in\supp(P)}\lt\{t\in R:st\in I_{\ell-1}(P)\rt\}\right)\label{eq:I+>I-IH}\\
        &=I_\ell(P).\nonumber
    \end{align}
    Therefore $I_\ell(P)\subset I_\ell^+(P)$ for all $\ell\geq 0$.
 If $R$ is principal, then     by Lemma~\ref{lem:principalscommute}, \eqref{eq:I+>I-ineq} holds with equality. Consequently, $I_\ell(P)=I_\ell^+(P)$ for $\ell\geq 0$ follows by induction with~\eqref{eq:I+>I-IH} assumed as an equality in the inductive hypothesis.
    \end{proof}

The following further observation about $\{I_\ell(P)\}_{\ell\geq 0}$ and $\{I_\ell^+(P)\}_{\ell\geq 0}$ will be useful in the proofs of Lemmas~\ref{lem:sublinearSATcondition},~\ref{lem:T^+obstruction}, and~\ref{lem:Tobstruction}. 
\begin{observation}\label{obs:I^-monotone}
    Let $P$ be a distribution on a finite commutative ring $R$. Then for all $\ell\geq 0$, $I_\ell(P)\supset I_{\ell+1}(P)$, and $I_\ell^+(P)\supset I_{\ell+1}^+(P)$.
\end{observation}
\begin{proof}
    Let $\ell\geq 0$. For any $r\in \supp(P)$, we have
    \begin{equation*}
        I_{\ell+1}(P)\subseteq \{rt:rt\in I_{\ell}(P)\} \subseteq I_{\ell}(P).
    \end{equation*}
On the other hand,
    \begin{align*}
        I_{\ell+1}^+(P)&=\bigcap_{x\in \supp(P)^\NN}I_{\ell+1}^+(x)=\bigcap_{x\in\supp(P)^\NN}\{x_{2\ell+2}t:x_{2\ell+1}t\in I_{\ell}^+(x), t\in R\} \subseteq I^+_{\ell}(P).
    \end{align*}
\end{proof}

\begin{definition}
    Let $F$ be a factor graph where every constraint vertex has degree $k$. A $2$-subtree of $F$ is a connected acyclic subgraph of $F$ in which every constraint node has degree $k$ and every variable node has degree at most $2$. Let $\mc T_2(F)$ denote the set of $2$-subtrees of $F$.
\end{definition}
Recall that for a tree $T$, $L(T)$ denotes the set of leaves of $T$.

\begin{definition}\label{def:2-subtreeheight}
    Let $F$ be an acyclic factor graph. For $T\in \mc T_2(F)$ and $a\in C(T)$, define
    \begin{equation*}
        h_T(a)=\min\{(d_T(a,v)-1)/2:v\in L(T)\},
    \end{equation*}
    where $d_T$ is the usual graph distance metric in the subgraph $T$. For $a\in C(F)$, define
    \begin{equation*}
        \bar h_{F}(a)=\max \{h_T(a): T\in \mc T_2(F)\}.
    \end{equation*}

\end{definition}

\begin{lemma}\label{lem:maximalI_a}
    Given $B\in R^{C\times V}$, let $\{I_a\}_{a\in C}$ be the unique maximal collection of ideals such that
    \begin{equation*}
        \bigoplus_{a\in C} I_a\subset \{Bx:x\in R^V\}.
    \end{equation*}
    Then $\{I_a\}_{a\in C}$ satisfies
    \begin{equation*}
        I_a\supset B_{a,v}\bigcap_{a'\in C\setminus \{a\}}\{t\in R: B_{a',v}t\in I_{a'}\}
    \end{equation*}
    for all $a,\in C$, $v\in V$.
\end{lemma}
\begin{proof}
    Let $a\in C$ and $v\in V$. Let $t\in R$ be such that $B_{a',v}t\in I_{a'}$ for all $a'\in C\setminus \{a\}$. We show that $B_{a,v}t e_a\in \{Bx:x\in R^V\}$, which implies that $B_{a,v}t\in I_a$ by maximality of $\{I_a\}_{a\in C}$. 
    
    For each $a'\in C\setminus\{a\}$, since $B_{a',v}t\in I_{a'}$, there is an $x^{(a')}\in R^V$ such that $Bx^{(a')}=B_{a',v}te_{a'}$. Letting
    \begin{equation*}
        x=te_v-\sum_{a'\in C\setminus \{a\}}x^{(a')}
    \end{equation*}
    we then have
    \begin{equation*}
        Bx=Bte_v-\sum_{a'\in C\setminus \{a\}}B_{a',v}te_{a'}=B_{a,v}te_a.
    \end{equation*}
    Therefore $B_{a,v}t e_a\in \{Bx:x\in R^V\}$, as desired.
\end{proof}

\begin{lemma}\label{lem:h(a)bound}
    Let $F$ be a connected acyclic factor graph and let $\ell^*\geq 0$. If $F$ does not contain $T^+_{\ell^*}$, then $\bar h_{F}(a)<\ell^*$ for all $a\in C(F)$.
\end{lemma}
\begin{proof}
    We prove the contrapositive. Suppose there is an $a\in C(F)$ such that $\bar h_F(a)\geq \ell^*$. Let $T^*\in \mc T_2(F)$ be such that $h_{T^*}(a)=\bar h_{F}(a)$, and for each $0\le\ell\le \ell^*$ let $T^*_{\ell}$ be the subgraph of $T^*$ induced by the vertices $\{u\in V(T^*)\cup C(T^*): d_{T^*}(a,u)\leq 2\ell +1\}$. We  show that $T^*_{\ell}\cong T^+_{\ell}$ for  all $0\le \ell\le \ell^*$ by induction on $\ell$. Note that $T^*_{\ell^*} \cong T^+_{\ell^*}$ implies the contrapositive of the lemma.
    
    The graph $T^*_0$ is the subgraph induced by $a$ and its $k$ adjacent variable vertices, and so trivially $T^*_0\cong T^+_0$. 

    Suppose $1\leq \ell\leq \ell^*$ and that $T^*_{\ell-1}\cong T^+_{\ell-1}$. Then by definition of $T^+_{\ell-1}$, $d_{T^*}(a,u)=2(\ell-1)+1$ for all $u\in L(T^*_{\ell-1})$. If there is a $u\in L(T^*)\cap L(T^*_{\ell-1})$ (i.e.\ if some leaf in $T^*_{\ell-1}$ happens to be a leaf in $T^*$), then $h_{T^*}(a)\leq \ell-1<\bar h_F(a)$, which contradicts the choice of $T^*$. Therefore, since $T^*$ is a 2-subtree of $F$, all $u\in L(T^*_{\ell-1})$ must have degree 2 in $T^*$. Let $a_u$ be the neighbour of $u\in L(T^*_{\ell-1})$ that is in $T^*\setminus T^*_{\ell-1}$. Then, $T^*_{\ell}$ is obtained from $T^*_{\ell-1}$ by adding $a_u$ for each $u\in L(T^*_{\ell-1})$ along with the $k-1$ neighbours of $a_u$ that are not contained in $T^*_{\ell-1}$. Therefore $T^*_{\ell}\cong T^+_\ell$.
\end{proof}

\begin{proof}[Proof of Lemma~\ref{lem:sublinearSATcondition}]
    Suppose that $\ell(P,Q)=\infty$, or that $\ell(P,Q)<\infty$ and $F$ does not contain $T^+_{\ell(P,Q)}$ as a subgraph. Then there is a finite $\ell^*\leq\ell(P,Q)$ such that $F$ does not contain $T^+_{\ell^*}$ as a subgraph. By Lemma~\ref{lem:h(a)bound},  $\bar h_{F}(a)<\ell^*$ for all $a\in C(F)$. Let $B\in R^{C\times V}$ be the matrix corresponding to $F$, so that $F(B)=B$, and let $\{I_a\}_{a\in C}$ be the unique maximal collection of ideals such that $\bigoplus_{a\in C}I_a\subset\{Bx:x\in R^V\}$. To show that $F$ is satisfiable, it suffices to show that $I_a\supset \supp(Q)$ for all $a\in C(F)$. We  show that 
  \begin{equation}  
     I_a\supseteq I_{\bar h_F(a)}(P),\quad  \text{for all $a\in C(F)$}. \label{eq:ideal-containment}
    \end{equation}
    By the definition of $\ell(P,Q)$ and the fact that $\bar h_F(a)<\ell(P,Q)$ for every $a\in C(F)$, it follows that $ I_{\bar h_F(a)}(P)\supseteq \supp Q$ for every $a\in C(F)$, which implies   the satisfiability of $F$.

    We prove~\eqref{eq:ideal-containment} by induction on the value of $\bar h_F(a)$. If $\bar h_F(a)=0$, then there is a $v\in L(F)$ adjacent to $a$. Since $v$ is a leaf, $I_a\supset (B_{a,v})$ by maximality of $\{I_a\}_{a\in C}$. Since $B_{a,v}\in \supp(P)$, we have $(B_{a,v})\supset \bigcap_{r\in \supp(P)} rR=I_0(P)$. This verifies~\eqref{eq:ideal-containment} for all $a\in C(F)$ for which $\bar h_F(a)=0$.
    
    Now suppose $\ell\ge 1$ and suppose that~\eqref{eq:ideal-containment} holds for all $a\in C(F)$ for which $\bar h_F(a)\le \ell-1$. Let $a\in C(F)$ such that $\bar h_F(a)=\ell$. Let $v_1,\ldots, v_k$ be the neighbours of $a$. Since $\bar h_F(a)=\ell\ge 1$, the variable vertices $v_1,\dots, v_k$ are not leaves of $F$. We claim that there exists $1\le i\le k$ for which the following holds: 
   \begin{equation}
   \bar h_F(a')\le \ell-1,\quad \text{for all $a'\neq a$ adjacent to $v_i$ in $F$}\label{eq:inductive-step}
   \end{equation} 
    
    Suppose for a contradiction that each $v_i$ is adjacent to some $a_i\neq a$ with $\bar h_F (a_i)\geq \ell$. For each $a_i$, let $T_i\in {\mathcal T}_2(F)$ be such that $a_i\in T_i$ and
    \begin{equation}
h_{T_i}(a_i)       =\bar h_F(a_i), \label{def:choice-of-Ti}
    \end{equation}
recalling that $h_{T_i}(a_i)= \min_{v\in L(T_i)}\{(d_{T_i}(a_i,v)-1)/2\}$.

     For each $1\le i\le k$, let $T_i'$ be the component of $T_i-av_i$ that contains $a_i$ (that is, $T_i'$ is obtained by  deleting the edge $av_i$ and taking the component that contains $a_i$).  Note that $T_1',\ldots, T_k'$ are disjoint, as otherwise $F$ would contain a cycle. Let $T=\bigcup_{i\in[k]} (T_i'+av_i)$. We first show that $T\in \mc T_2(F)$. The connected subgraphs $T_1',\dots,T_k'$ are connected in $T$ by paths through $a$, and so $T$ is connected. It is obvious that every $a'\neq a$ in $T$ has degree $k$ in $T$ since they each has degree $k$ in the subtrees $T_1',\ldots, T_k'$.
Immediately, all the constraint vertices in $T$, including $a$, have degree $k$ in $T$.
For the variable vertices in $T$, they all have degree two or one in the subtrees $T_1',\ldots,T_k'$. Moreover, $v_1,\ldots,v_k$ were leaves in the subtrees, who become degree two vertices in $T$.  Therefore $T\in \mc T_2(F)$. Let $\pi$ be a shortest $a,v$-path in $T$ among $v\in L(T)$. Since $v_1,\dots, v_k$ are not leaves of $T$, $\pi$ must contain a path $\pi'$ in $T_i$ from $a_i$ to $v$ for some $i\in[k]$ such that $|\pi|=|\pi'|+2$. But now, by~\eqref{def:choice-of-Ti},
    \begin{align*}
        \ell=\bar h_F(a)\geq (|\pi|-1)/2=(|\pi'|-1)/2 +1\geq h_{T_i}(a_i)+1 = \bar h_F(a_i)+1\geq \ell+1.
    \end{align*}
    This is a contradiction, which completes the proof for~\eqref{eq:inductive-step}.

Let $v'\in\{v_1,\ldots,v_k\}$ be a neighbour of $a$ that satisfies~\eqref{eq:inductive-step}. 
By Lemma~\ref{lem:maximalI_a},   
    \begin{align*}
        I_a&\supset B_{a,v'}\bigcap_{a'\in N(v')\setminus\{a\}}\{t\in R: B_{a',v'}t\in I_{a'}\}.
    \end{align*} 
     By the induction hypothesis and the choice of $v'$ so that it satisfies~\eqref{eq:inductive-step}, we conclude that $I_{a'}\supset I_{\bar h_F(a')}(P)$ for all $a'\in N(v')\setminus \{a\}$. By Observation \ref{obs:I^-monotone}, $I_{\bar h_F(a')}(P)\supset I_{\ell-1}(P)$ since $\bar h_F(a')\le \ell-1$. Thus, $I_{a'}\supset I_{\ell-1}(P)$ and so 
    \begin{align*}  
        B_{a,v}\bigcap_{a'\in N(v')\setminus\{a\}}\{t\in R: B_{a',v}t\in I_{a'}\}&\supset B_{a,v}\bigcap_{a'\in N(v')\setminus\{a\}}\{t\in R: B_{a',v}t\in I_{\ell-1}(P)\}\\
        &\supset \bigcap_{r\in \supp(P)}r\left(\bigcap_{s\in \supp(P)}\{t\in R: st\in I_{\ell-1}(P)\}\right)\\
        &=I_{\ell}(P).
    \end{align*}
    Therefore $I_a\supset I_{\ell}(P)$, which confirms~\eqref{eq:ideal-containment} for all $a\in C(F)$ for which $\bar h_F(a)=\ell$. By induction,~\eqref{eq:ideal-containment} holds for every $a\in C(F)$.
\end{proof}

\begin{proof}[Proof of Lemma~\ref{lem:T^+obstruction}] Let $\ell^+=\ell^+(P,Q)$.
    Since $\ell^+<\infty$, there is an $x\in \supp(P)^\NN$ and a $b\in \supp(Q)$ such that $b\notin I^+_{\ell^+}(x)$.

    Let $C_0=C(T^+_0)$ and  $C_\ell=C(T^+_\ell)\setminus C(T^+_{\ell-1})$ for $1\leq\ell\leq \ell^+$. Then, $C_0,\ldots, C_{\ell^+}$ is a partition of $C(T^+_{\ell^+})$. Similarly, let $V_0=V(T^+_0)$ and $V_\ell=V(T^+_\ell)\setminus V(T^+_{\ell-1})$ for $1\leq\ell\leq \ell^+$. Let $a_0$ be the unique constraint in $C_0$. For each $\ell=0,\ldots,\ell^+(P,Q)$ and edge $av\in E(T^+_{\ell^+})$ with $a\in C_\ell$, set
        \begin{equation*}
            w(av)=\begin{cases}
                x_{2(\ell^+-\ell)}&v\in V_\ell\\
                x_{2(\ell^+-\ell)+1}&v\notin V_\ell.
            \end{cases}
        \end{equation*}

    Define constraint vertex weights on $T^+_{\ell^+}$ by setting $w(a_0)=b$ and setting $w(a)\in \supp(Q)$ arbitrarily for $a\in C(T^+_{\ell^+})\setminus\{a_0\}$. To show that $(T^+_{\ell^+},w)$ is unsatisfiable, suppose $y\in R^{V(T^+_{\ell^+})}$ satisfies every constraint in $C(T^+_{\ell^+})\setminus\{a_0\}$, i.e.
    \begin{equation*}
        \sum_{v\in N(a)} w(av)y_v=w(a),\quad a\in C(T^+_{\ell^+})\setminus\{a_0\},
    \end{equation*}
    where $N(a)$ denotes the set of neighbours of $a$ in $T^+_{\ell^+}$.
    We show that for all $0\leq \ell \leq \ell^+$ 
    \begin{equation}
    w(av)y_v\in I^+_{\ell^+-\ell}(x)\quad \text{for all $a\in C_\ell$, $v\in N(a)\cap V_\ell$} \label{eq:inductionI}
    \end{equation}
    by reverse induction on $\ell$. Note that~\eqref{eq:inductionI} implies that $\sum_{v\in N(a_0)} w(a_0v)y_v\in I^+_{\ell^+}(x)$. Since $w(a_0)=b\notin I^+_{\ell^+}(x)$, this implies that $y$ does not satisfy $(T^+_{\ell^+},w)$.

It remains to prove~\eqref{eq:inductionI}.
    For the base case $\ell=\ell^+$, suppose $a\in C_{\ell^+}$ and $v\in N(a)\cap V_{\ell^+}$. Then as $w(av)=x_0$ by definition of $x$,  $w(av)y_v\in (x_0)=I_0(x)$.

    Now suppose $0\leq \ell <\ell^+$ and suppose that~\eqref{eq:inductionI} holds for all greater values of $\ell$. Let $a\in C_\ell$ and $v\in N(a)\cap V_\ell$. Since $\ell<\ell^+$, there is an $a'\in N(v)\cap C_{\ell+1}$. Let $v_1,\ldots, v_{k-1}$ be the $k-1$ neighbours of $a'$ other than $v$. In other words,  $\{v_1,\ldots, v_{k-1}\}=N(a')\cap V_{\ell+1}$. By the induction hypothesis, $w(a'v_i)y_{v_i}\in I^+_{\ell^+-\ell-1}(x)$ for every $i\in[k-1]$. Since $a'\neq a_0$, $y$ satisfies $a'$, and so we have 
    \begin{equation*}
        w(a'v)y_v=w(a')-\sum_{i\in[k-1]}w(a'v_i)y_{v_i}\in w(a')+I_{\ell^+-\ell-1}(x).
    \end{equation*}
Since $\ell^+-\ell-1<\ell^+$, by definiton of $\ell^+=\ell^+(P,Q)$, $w(a')\in \supp(Q)\subset I^+_{\ell^+-\ell-1}(x)$. Therefore $w(a'v)y_v\in I^+_{\ell^+-\ell-1}(x)$. Since $a\in C_\ell$, $v\in V_\ell$, we have $w(av)=x_{2(\ell^+-\ell)}$, and since $a'\in C_{\ell+1}$, $v\notin V_{\ell+1}$, we have $w(a'v)= x_{2(\ell^+-\ell)-1}$. Therefore
    \begin{equation*}
        w(av)y_v=x_{2(\ell^+-\ell)}y_v\in \{x_{2(\ell^+-\ell)}t:x_{2(\ell^+-\ell)-1}t\in I^+_{\ell^+-\ell-1}(x)\}=I^+_{\ell^+-\ell}(x). 
    \end{equation*}
    This completes the induction step and confirms that~\eqref{eq:inductionI} holds for all $0\leq \ell \leq \ell^+$.
    \end{proof}

    \begin{proof}[Proof of Lemma~\ref{lem:Tobstruction}] Let $\bar \ell=\ell(P,Q)$.
    Since $\bar \ell<\infty$, there is a $b\in \supp(Q)$ such that
    \begin{equation*}
        b\notin I_{\bar \ell}(P)=\bigcap_{r\in\supp(P)} r\lt(\bigcap_{s\in\supp(P)}\{t\in R: st\in I_{\bar \ell-1}(P)\}\rt).
    \end{equation*}
    Let $r_0\in\supp(P)$ be such that
    \begin{equation}\label{lem:Tobstruction-bvalue}
        b\notin r_0\lt(\bigcap_{s\in\supp(P)}\{t\in R: st\in I_{\bar \ell-1}(P)\}\rt).
    \end{equation}

    Let $C_0=C(T_0),V_0=V(T_0)$ and let $C_\ell=C(T_\ell)\setminus C(T_{\ell-1}),V_\ell=V(T_\ell)\setminus V(T_{\ell-1})$ for $1\leq\ell\leq \bar \ell$.

    Assign  weights $w$ to the edges of $T_{\ell(P,Q)}(P)$ as follows.
    \begin{itemize}
        \item $w(a_0v)=r_0$ for all $v\in N(a_0)$.
        \item For all $v\in V_\ell$ with $0\leq \ell<\bar \ell$ and all $(r,s)\in \supp(P)^2$, there is a constraint $a\in N(v)\cap C_{\ell+1}$ such that $w(av)=s$ and $w(av')=r$ for all $v'\in N(a)\cap V_{\ell+1}$.
    \end{itemize}
    Since each variable vertex $v\in V_\ell$ with $0\leq\ell< \bar \ell$ is adjacent to precisely $|\supp(P)|^2$ constraint vertices in $C_{\ell+1}$ by definition of $T_{\bar\ell}$, the edge weights $w$ are well defined. Define constraint vertex weights by setting $w(a_0)=b$ and setting $w(a)\in \supp(Q)$ arbitrarily for $a\in C(T_{\bar\ell})\setminus \{a_0\}$.

    To show that $(T_{\bar\ell},w)$ is unsatisfiable, suppose $y\in R^{V(T_{\bar\ell})}$ satisfies
    \begin{equation*}
        \sum_{v\in N(a)}w(av)y_v=w(a),\quad a\in C(T_{\bar\ell})\setminus \{a_0\}.
    \end{equation*}
    We show by reverse induction on $\ell$ that 
    \begin{equation}\label{lem:TobstructionIH}
        y_v\in \bigcap_{s\in \supp(P)}\{t\in R: st\in I_{\bar\ell-\ell-1}(P)\},\quad \text{for all}\ v\in V_\ell
    \end{equation}
    for all $0\leq \ell\leq \bar\ell$, where $I_{-1}(P)$ is defined to be $R$. We show that~\eqref{lem:TobstructionIH} implies that $y$ does not satisfy $a_0$. We have $N(a_0)\subset V_0$. By~\eqref{lem:TobstructionIH} with $\ell=0$,
    \begin{equation*}
        \sum_{v\in N(a_0)}w(a_0v)y_v=r_0\sum_{v\in N(a_0)}y_y\in r_0 \lt(\bigcap_{s\in\supp(P)}\{t\in R: st\in I_{\bar\ell-1}(P)\}\rt).
    \end{equation*}
    By~\eqref{lem:Tobstruction-bvalue}, this shows that $y$ does not satisfy $a_0$. Therefore $(T_{\bar\ell},w)$ is not satisfiable.

    Now we prove~\eqref{lem:TobstructionIH}. For the base case  $\ell=\bar\ell$, we have
    \begin{equation*}
        \bigcap_{s\in \supp(P)}\{t\in R: st\in R\}=R,
    \end{equation*}
and    so~\eqref{lem:TobstructionIH} holds trivially.

    Now suppose $0\leq \ell <\bar\ell$ and that~\eqref{lem:TobstructionIH} holds for $\ell+1$. Let $v\in V_\ell$ and let $(r',s')\in \supp(P)^2$. Then there is an $a\in N(v)\cap C_{\ell+1}$ such that $w(av)=s'$ and $w(av')=r'$ for all $v'\in N(a)\cap V_{\ell+1}$. By the induction hypothesis
    \begin{equation*}
        y_{v'}\in \bigcap_{s\in \supp(P)}\{t\in R: st\in I_{\bar\ell-\ell-2}(P)\}
    \end{equation*}
    for all $v'\in N(a)\cap V_{\ell+1}$. Since $y$ satisfies $a\neq a_0$, we have
    \begin{equation*}
        w(av)y_v=w(a)-r'\sum_{v'\in N(a)\setminus \{v\} }y_{v'}\in w(a)+r'\bigcap_{s\in \supp(P)}\{t\in R: st\in I_{\bar\ell-\ell-2}(P)\}.
    \end{equation*}
    Since $w(av)=s'$ and by definition of $\bar\ell=\ell(P,Q)$ that
    \begin{equation*}
        w(a)\in \supp(Q)\subset I_{\ell-1}(P)\subset r'\bigcap_{s\in \supp(P)}\{t\in R: st\in I_{\bar\ell-\ell-2}(P)\},
    \end{equation*}
    we obtain that
    \begin{equation*}
        s'y_v\in r'\bigcap_{s\in \supp(P)}\{t\in R: st\in I_{\bar\ell-\ell-2}(P)\}.
    \end{equation*}
    Since this holds for all $r'\in \supp(P)$, we have $s'y_v\in I_{\bar\ell-\ell-1}(P)$ for all $s'\in \supp(P)$. Therefore $y_v\in \bigcap_{s\in \supp(P)} \{t\in R: st\in I_{\bar\ell-\ell-1}(P)\}$. This completes the induction step, and so~\eqref{lem:TobstructionIH} holds for $0\leq \ell\leq \bar\ell$.
    \end{proof}

\section{Linear SAT threshold for principal rings: proof of Theorem~\ref{thm:principal}}\label{sec:linearprincipal}

We prove Theorem~\ref{thm:principal} by a sequence of reductions to Theorem~\ref{thm:units}.
First, we reduce the problem to considering equations over finite local rings.

Recall from Lemma~\ref{lem:localproduct} that any finite ring $R$ can be writen as a product $R\cong R_1\times \cdots \times R_h$ where $R_1,\ldots, R_h$ are local rings (not necessarily principal). Let $Bx=b$ be a system of linear equations over $R$. For each $i\in [h]$, let $\pi_i:R\to R_i$ be the natural projection map onto the $i$-th coordinate. For each $i\in [h]$, applying $\pi_i$ entrywise to both $B$ and $b$ defines a system of linear equations over $R_i$, which we denote by $B^{(i)}\, x=b^{(i)}$. It follows immediately that
\begin{equation}\label{eq:reduction}
    Bx=b  \ \text{is satisfiable over $R$} \Longleftrightarrow B^{(i)}\, x=b^{(i)}\ \text{is satisfiable over $R_i$, for every $i\in[h]$.} 
\end{equation}

Consider the linear system $\B x=\b$ given in~\eqref{eq:system}. By~\eqref{eq:reduction}, it suffices to determine the critical $m$ around which one of the systems $\B^{(i)}x=\b^{(i)}$ becomes unsatisfiable. This leads us to study the distribution of $(\B^{(i)},\b^{(i)})$.
Given $P$ and $Q$, both of which are  distributions over $R=R_1\times\cdots\times R_h$, let $P_i$ and $Q_i$ be the marginal distributions of $P$ and $Q$ respectively on $R_i$. In other words, for each $i\in[h]$, $P_i:=P\circ\pi_i^{-1}$ and $Q_i:=Q\circ\pi_i^{-1}$.
Clearly, $(\B^{(i)},\b^{(i)})$ has the same distribution as $(\B,\b)$ with $(P,Q)$  replaced by $(P_i,Q_i)$.

Our next lemma shows that the marginal distributions of $P$ 
on each local ring $R_1,\ldots, R_h$ have a very simple form if $\alpha_{P,Q}=1$.

\begin{lemma}\label{lem:oneideal}
    Suppose $R=R_1\times\cdots\times R_h$, where $R_1,\ldots, R_h$ are local principal rings. Suppose $Q,P$ are distributions on $R$ such that $\alpha_{P,Q}=1$. Let $P_i:=P\circ\pi_i^{-1}$ and $Q_i:=Q\circ\pi_i^{-1}$. If $i\in[h]$ is such that $(P_i^{\otimes k}, Q_i)$ satisfies {\bf (A2)}, then $|\{rR_i:r\in\supp(P_i)\}|=1$. 

\end{lemma}

\begin{proof}[Proof of Lemma~\ref{lem:oneideal}]
    Let $M_i$ be the unique maximal ideal of $R_i$, and suppose for a contradiction that $|rR_i:r\in \supp(P_i)\}|\geq 2$ and $(P_i^{\otimes k},Q)$ satisfies {\bf (A2)}. Then there are $r_1,r_2\in \supp(P)$ such that $\pi_i(r_1)R_i\neq \pi_i(r_2)R_i$, and  some $b\in \supp(Q)$ such that $\pi_i(b)\neq 0$. Since  $R_i$ is a local principal ring,  we have without loss of generality that $\pi_i(r_1)R_i\subsetneq \pi_i(r_2)R_i$. It follows then that there is an $s\in R_i$ such that  $\pi_i(r_1)=\pi_i(r_2)s$. Moreover, $s$ cannot be a unit as otherwise it contradicts with $\pi_i(r_1)R_i\neq \pi_i(r_2)R_i$. Hence, $s\in M_i$. By Observation~\ref{obs:I^-monotone}, there exists an $\ell^*\geq 0$  such that $I_{\ell^*}(P)=I_{\ell^*-1}(P)$. Then we have
    \begin{align*}
        I_{\ell^*}(P)\subset \{r_1t:r_2t\in I_{\ell^*-1}(P),\ t\in R\},
    \end{align*}
    and so
    \begin{align*}
        \pi_i(I_{\ell^*}(P))\subset \{\pi_i(r_1)t:\pi_i(r_2)t\in \pi_i(I_{\ell^*-1}(P)),\ t\in R_i\}\subset s\pi_i(I_{\ell^*-1}(P))=s\pi_i(I_{\ell^*}(P)).
    \end{align*}

 By Lemma~\ref{lem:rings-ideals}, there is a $\nu\geq 0$ such that $M^\nu =0$. Since $s\in M$, we then have $\pi_i(I_{\ell^*}(P))=s^\nu\pi_i(I_{\ell^*}(P))=0$. Since $\pi_i(b)\neq 0$, this implies $\ell(P,Q)\leq \ell^*$, which contradicts the assumption that $\alpha_{P,Q}=1$.
\end{proof}

Finally, we complete the proof of Theorem~\ref{thm:principal} by reducing it to Theorem~\ref{thm:units} via Lemma~\ref{lem:oneideal}. 

\begin{proof}[Proof of Theorem~\ref{thm:principal}]
    By Lemma~\ref{lem:localproduct}, there exist principal local rings $R_1,\ldots, R_h$ such that $R\cong R_1\times\cdots \times R_h$. 
    By~\eqref{eq:reduction} and the discussions below it, it suffices to show that 
    the set of systems $\B^{(i)}x=\b^{(i)}$ is a.a.s.\ simultanously satisfiable if $m<(d_k-\eps)n/k$, and there exists some $i\in [h]$ such that $\B^{(i)}x=\b^{(i)}$ is 
   a.a.s.\ unsatisfiable if $m>(d_k+\eps)n/k$. 

We first consider the subcritical regime where $m<(d_k-\eps)n/k$. Let $i\in [h]$. If $(P^{\otimes k}_i,Q_i)$ admits a constant solution $r_i\in R_i$ then $\B^{(i)}x=\b^{(i)}$ is satisfiable trivially. Thus we may assume that $(P^{\otimes k},Q_i)$ satisfies {\bf (A2)}.
    \remove{If every $(P^{\otimes k}_i,Q_i)$ admits a constant solution $r_i\in R_i$, then $(r_1,\ldots, r_h)\in R_1\times \ldots\times R_h$ is a constant solution for $\B x=\b$. Since $(P^{\otimes k},Q)$ satisfies {\bf (A2)}, there therefore must be an $i\in[h]$ such that $(P_i^{\otimes k},Q)$ satisfies {\bf (A2)}. We now show that

    \begin{equation}\label{localthresholds}
        \lim_{n\to\infty}\Pr[\B^{(i)}x=\b^{(i)}\text{ is satisfiable}] =\begin{cases}
            1&\text{ if }m<(c_i-\eps)n/k\\
            0&\text{ if }m>(c_i+\eps)n/k.
        \end{cases}
    \end{equation}

    holds with $c_i=d_k$ if $(P_i^{\otimes k},Q_i)$ satisfies {\bf (A2)}, and $c_i=\infty$ if $(P_i^{\otimes k},Q_i)$ does not satisfy {\bf (A2)}. Since there is an $i\in [h]$ such that $(P_i^{\otimes k},Q_i)$ satisfies {\bf (A2)}, the result then follows from Lemma~\ref{lem:reductiontolocal}.

    It is clear that~\eqref{localthresholds} holds with $c_i=\infty$ if $(P_i^{\otimes k},Q_i)$ does not satisfy {\bf (A2)}. Suppose $(P_i^{\otimes k},Q_i)$ satisfies {\bf (A2)}.} Let  $g\in R_i$ be a generator of the maximal ideal $M_i$. Since $R_i$ is local and principal, by Lemma~\ref{lem:rings-ideals}(b), there is an $\ell\geq 0$ such that $rR_i=M_i^\ell=(g^{\ell})$ for all $r\in \supp(P_i)$, where the uniformity of $\ell$ for every $r\in \supp(P_i)$ is a consequence of the fact that $|\{rR_i:r\in \supp(P_i)\}|=1$ which follows by Lemma~\ref{lem:oneideal}. Define a map $\phi:M_i^\ell\to R_i/\Ann(M_i^\ell)$ by  $\phi(sg^{\ell})=s+\Ann(M_i^\ell)$ for all $s\in R_i$. For a matrix $B$ and vector $b$ over $R_i$, let $\phi(B)$ and $\phi(b)$ be the matrix and vector obtained by applying $\phi$ entrywise.

    \begin{claim}\label{claim:reduction}
    \begin{enumerate}[(a)]
        \item $\phi(r)$ is a unit of $R_i/\Ann(M_i^\ell)$ for all $r\in \supp(P_i)$.
\item Suppose $B\in R_i^{C\times V}$ and $b\in R_i^C$ both have entries in $M^\ell_i$. Then $Bx=b$ is satisfied by some $x\in R^V$ if and only if $\phi(B)y=\phi(b)$ is satisfied by some $y\in (R_i/M_i^\ell)^V$.
    \end{enumerate}
    \end{claim}

       We have shown that $\supp(P_i)\subset M_i^\ell$. We also have $\supp(Q_i)\subset M_i^\ell$, as otherwise 
     $\supp(Q_i)\not\subset M_i^\ell=\pi_i(I_0(P))$, which would imply that $\ell(P,Q)=0$, contradicting with the assumption that  
 $\alpha_{P,Q}=1$. Thus, Claim~\ref{claim:reduction}(b) implies that the system $\B^{(i)}\bar x=\b^{(i)}$ is satisfiable if and only if $\phi(\B^{(i)})\bar x=\phi(\b^{(i)})$ is satisfiable. By part (a) of the claim and Theorem~\ref{thm:units}, $\phi(\B^{(i)})\bar x=\phi(\b^{(i)})$ is a.a.s.\ satisfiable. This completes the proof for the subcritical regime.

       The proof for the supercirical regime $m>(d_k+\eps)n/k$ is similar. Since $(P^{\otimes k}, Q)$ satisfies {\bf (A2)}, there must exists $i\in [h]$ such that $(P_i^{\otimes k}, Q_i)$ satisfies {\bf (A2)}. By Claim~\ref{claim:reduction}(a) and 
Theorem~\ref{thm:units}, a.a.s.\ $\phi(\B^{(i)})\bar x=\phi(\b^{(i)})$ is unsatisfiable. Thus, a.a.s.\ $\B^{(i)}\bar x=\b^{(i)}$ is unsatisfiable by Claim~\ref{claim:reduction}(b), which, together with~\eqref{eq:reduction}, further implies that a.a.s.~\eqref{eq:system} is unsatisfiable, completing the proof of Theorem~\ref{thm:principal}.
    
    \begin{proof}[Proof of Claim~\ref{claim:reduction}]
If $r\in \supp(P)$, then $rR=M_i^{\ell}\supsetneq M_i^{\ell+1}$, and so $s\notin M_i$ for any $s\in R_i$ such that $r=sg^{\ell}$. Therefore $\phi(r)\notin M_i/\Ann(M_i^\ell)$ for all $r\in \supp(P)$, and so $\phi(r)$ is a unit of $R/\Ann(M_i^\ell)$ for all $r\in \supp(P)$. This confirms part (a).

For part (b), note that $\phi$ is an isomorphsim of $R_i$-modules. Applying $\phi$ coordinate wise gives an isomorphism between $(M^\ell_i)^C$ and $(R_i/\Ann(M_i^\ell))^C$. Therefore $\phi(b)$ is in the column space of $\phi(B)$ if and only if $b$ is in the column space of $B$. Moreover, for any $y\in R_i^V$ and $\bar y=y+\bigoplus_{v\in V} \Ann(M_i^\ell)$ we have $\phi(B)y=\phi(B)\bar y$. Therefore $Bx=b$ is satisfied by some $x\in R_i^V$ if and only if $\phi(B)\bar y=\phi(b)$ is satisfied by some $\bar y\in \bigoplus_{v\in V}R_i/\Ann(M_i^\ell)$.
\end{proof}
\end{proof}


\section{Linear SAT threshold for nonprincipal rings: proof of Theorem~\ref{thm:nonprincipal-linear}}\label{sec:nonprincipal-linear}

Throughout this section, we let $R=\FF_q[X,Y]/(X^2,Y^2)$. Given any $B\in \{\bar X,\bar Y\}^{C_m\times V_n}$ and $b\in \{\bar X\bar Y,0\}^{C_m}$, define $B_X,B_Y\in \FF_q^{C_m\times V_n}$ by
\begin{equation*}
    (B_X)_{a,v}=\ind{B_{a,v}=\bar X},\quad (B_Y)_{a,v}=\ind{B_{a,v}=\bar Y},\quad b_{XY}=\ind{b_a=\bar X\bar Y},
\end{equation*}
where $\ind{B_{a,v}=p}$ is the indicator function on $(a,v)\in C_m\times V_n$, ie. $\ind{B_{a,v}=p}=1$ if $B_{a,v}=p$ and $\ind{B_{a,v}=p}=0$ otherwise. Let $B_{XY}=[B_X\mid B_Y]$. To prove Theorem~\ref{thm:nonprincipal-linear}, we first show that the system $Bx=b$ over $R$ is equivalent to $B_{XY}x=b_{XY}$ as a system of equations over $\FF_q$. This holds deterministically for any $B\in \{\bar X,\bar Y\}^{C_m\times V_n}$ and $b\in \{\bar X\bar Y,0\}^{C_m}$. We then use a coupling argument together with Theorem~\ref{thm:asympequiv} to apply the results of~\cite{coja2024full}, which gives the satisfiability threshold.

\begin{lemma}\label{lem:SATinp^n}
Suppose that $B\in \{\bar X,\bar Y\}^{C_m\times V_n}$ and $b\in \{\bar X\bar Y,0\}^{C_m}$. If $Bx=b$ is satisfiable then there are $c,c'\in \FF_q^{V_n}$ such that $B(c\bar X+c'\bar Y)=b$.
\end{lemma}
\begin{proof}
    
    Let $x\in R^{V_n}$ be such that $Bx=b$. Let $c^{(1)},c^{(2)},c^{(3)},c^{(4)}\in \FF_q^{V_n}$ be such that $x= c^{(1)}\bar X\bar Y+c^{(2)}\bar X+c^{(3)}\bar Y+c^{(4)}$. We then have
    \begin{align*}
        \bar X\bar YBc^{(1)}+\bar XBc^{(2)}+\bar YBc^{(3)}+B_{a,v}c^{(4)}=b\in (\bar X\bar Y)^C
    \end{align*}
    Since $B_{a,v}\in \{\bar X,\bar Y\}$ for all $a\in C_m$, $v\in V_n$, we have $\bar X\bar YBc^{(1)},\bar XBc^{(2)},\bar YBc^{(3)}\in (\bar X\bar Y)^{C_m}$. Thus, for any $a\in C_m$,
    \begin{align*}
        (Bc^{(4)})_a=\bar X\sum_{v\in V_n: a(v)=\bar X}c^{(4)}_v+\bar Y\sum_{v\in V_n: a(v)=\bar Y}c^{(4)}_v\in (\bar X\bar Y)
    \end{align*}
    But the only $u,u'\in\FF_q$ such that $\bar Xu+ \bar Yu'\in (\bar X\bar Y)$ are $u=u'=0$. Therefore $(Bc^{(4)})_a=0$ for all $a\in C_m$. We also have $\bar X\bar YBc^{(1)}=0$, as $\bar X\bar YB_{a,v}=0$ for all $a\in C_m$ and $v\in V$. Therefore $B(c^{(2)}\bar X+c^{(3)}\bar Y)=b$.
\end{proof}

\begin{lemma}\label{lem:reduction-nonprincipal}
Suppose that $B\in \{\bar X,\bar Y\}^{C_m\times V_n}$ and $b\in \{\bar X\bar Y,0\}^{C_m}$.     Then,
    the system $B_{XY}\,x=b_{XY}$ is satisfiable over $\FF_q^{V}\times \FF_q^V$ if and only if $Bx=b$ is satisfiable over $R^{V_n}$.
\end{lemma}
\begin{proof}
    Given $z=(x,y)\in \FF_q^{V_n}\times \FF_q^{V_n}$, define $z'\in R^V$ by
    \begin{align*}
        z'=x\bar Y+y\bar X.
    \end{align*}
    By Lemma \ref{lem:SATinp^n}, it suffices to show that $B_{XY}z=b_{XY}$ if and only if $Bz'=b$. For any $a\in C_m$, we have
    \begin{align*}
        (Bz')_a&=\sum_{v\in V_n}B_{a,v}z'_v\\
        &=\sum_{v\in V: B_{a,v}=\bar X} \bar X(x_v\bar Y+y_{v}\bar X)+\sum_{v\in V_n: B_{a,v}=\bar Y} \bar Y(x_v\bar Y+y_{v}\bar X)\\
        &= \bar X\bar Y\lt(\sum_{v\in V_n: B_{a,v}=\bar X} x_v+\sum_{v\in V_n: B_{a,v}=\bar Y} y_{v}\rt)\\
        &=\bar X\bar Y(B_{XY}z)_a.
    \end{align*}
    Since $b=\bar X\bar Yb_{XY}$, this shows that $B_{XY}z=b_{XY}$ if and only if $Bz'=b$.
\end{proof}

\subsection{Satisfiability}\label{sec:linearnonprincipal-SAT}
In this section we prove satisfiability of $\B_{XY}x=\b_{XY}$ in the subcritical case. Due to Lemma~\ref{lem:reduction-nonprincipal}, it suffices to prove the following.
\begin{lemma}\label{lem:sat-nonprincipal}
    Suppose $m<(d_{k,\rho}-\eps)n/k$ for some fixed $\eps>0$. Then $\B_{XY}$ has full row rank a.a.s.
\end{lemma}
We prove Lemma~\ref{lem:sat-nonprincipal} using a coupling argument which compares $\B_{XY}$ with a different model $\A$ from~\cite{coja2024full} whose full row-rank condition was fully characterised. The factor graph $F(\A)$ is distributed as $G_{n,\rv m,\rv k,\rv d}^{(1)}$ introduced in Section~\ref{sec:contiguity}. To construct $\A$ from $G_{n,\rv m,\rv k,\rv d}^{(1)}$, let $\chi$ be a random variable  on $\FF_q\setminus \{0\}$. Choose $G_{n,\rv m,\rv k,\rv d}^{(1)}$ and $\{\chi_{a_i,v_j}\}_{i,j\in\NN}$ independently, such that $\{\chi_{a_i,v_j}\}_{i,j\in\NN}$ are i.i.d.\ copies of $\chi$. Then set
\begin{equation*}
\A_{a_i,v_j}=\ind{a_iv_j\in E(G_{n,\rv m,\rv k,\rv d}^{(1)})}\chi_{a_i,v_j}.
\end{equation*}

The following theorem is the main result of~\cite{coja2024full}.
\begin{theorem}[\cite{coja2024full}]\label{thm:rank-fields} Given integer random variables $\rv k$, $\rv d$ such that $\ex[\rv d^r],\ex[\rv k^r]<\infty$ for some $r>2$,
    let $K(z)$ and $D(z)$ be the probability generating functions of $\rv k$ and $\rv d$ respectively. Define
    \begin{equation}
        \Phi_{\rv k,\rv d}(z)=D(1-K'(z)/\ex[k])-\frac{\ex[d]}{\ex [k]}(1-K(z)-(1-z)K'(z)).
    \end{equation}
    If $q$ and $\gcd(\rv d)$ are relatively prime and $\Phi_{\rv k,\rv d}(z)<\Phi_{\rv k,\rv d}(0)$ for all $0<z\leq 1$, then $\A$ has full row rank a.a.s.
\end{theorem}

We wish to couple $\A$ and $\B_{XY}$ such that $\A$ and $\B_{XY}$ have a large intersection. More precisely, we wish to specify a set of rows $I$ 
and a set of columns $S$ such that $\A$ has slightly more rows  than those in $I$, $\B_{XY}$ has slightly more columns than those in $S$, and $\A_{av}=(\B_{XY})_{av}$ for all $(a,v)\in I\times S$. However, directly coupling $\A$ and $\B_{XY}$ with these properties are difficult. So, we will construct matrices $\C'$ and $\C''$ such that $\A \contig \C'$ (for some choice of $\rv k,\rv d,\chi $), $\C''\contig \B_{XY}$ (up to permuting columns) and then couple $\C'$ and $\C''$ such that they have the properties mentioned above. 

To construct $(\C',\C'')$ jointly, we construct a random $\{0,1\}$-matrix $\C$ with two submatrices $\C',\C''$ such that
\begin{itemize}
    \item every row of $\C$ is in $\C'$ a.a.s,
    \item every column of $\C$ is in $\C''$.
\end{itemize}

Since $\C$, $\C'$, and $\C''$ are $\{0,1\}$-matrices, it suffices to couple their factor graphs $F(\C)$, $F(\C')$, $F(\C'')$. Fix $d^*>0$ such that $(d^*-\delta)n/k>m$ for some $\delta>0$, and fix $\eps>0$. Recall that $n,m,k,\rho$ are parameters given in Theorem~\ref{thm:nonprincipal-linear}. We first define three coupled multigraphs $M=M_{n,m,k,\rho,d^*,\eps}$, $M'=M_{n,k,\rho,d^*,\eps}'$, $M''=M_{n,m,k,\rho,\eps}''$, which will be used later to couple $F(\C)$, $F(\C')$, and $F(\C'')$. 
\begin{itemize}
    \item Let $\rv m\sim \Po(d^*n/k)$, and let $\rv {\hat m}=\max \{m,\rv m\}$. The vertex set of $M$ is $V_{2n}\cup C_{\rv {\hat m}}$.

To define the edge set of $M$, we first choose a random partition $(A,A',B,B')$ of $V_{2n}$ independent of $\rv {\hat m}$. Choose a set $S$ uniformly from the subsets of $V_{2n}$ of size $(2-2\eps)n$. Conditioned on $S$, let $A$ be a uniformly chosen random subset of $S$, and let $B=S\setminus A$. Conditioned on $A$ and $S$, choose $A'$ uniformly from the set of subsets $A'\subset V_{2n}\setminus S$ that minimize $||A\cup A'|-n|$. Finally, let $B'=V_{2n}\setminus (S\cup A')$.

Now, conditioned on $\rv {\hat m}$ and $(A,A',B,B')$, let $\{\bb_{a,i}\}_{a\in C_{\rv {\hat m}},i\in[k]}$ be i.i.d.\ random variables on $V_{2n}$, where $\pr[\bb_{a,i}=v]=\rho/|A\cup A'|$ for $v\in A\cup A'$ and $\pr[\bb_{a,i}=v]=(1-\rho)/|B\cup B'|$ for $v\in B\cup B'$. The edge set of $M$ is given by the multiset $\{a\bb_{a,i}:i\in[k],a\in C_{\rv {\hat m}}\}$, so that there is an edge from $a\in C_{\rv {\hat m}}$ to $v\in V_{2n}$ for each $i\in[k]$ such that $\bb_{a,i}=v$.

\item Let $M'$ be the subgraph of $M$ induced by $C_{\rv m}\cup S$.
\item Let $M''$ be the subgraph of $M$ induced by $C_{m}\cup V_{2n}$.
\end{itemize}

Note that $\bb_{a,1},\ldots, \bb_{a,k}\in V_{2n}$ may not be distinct, and so $M$ is random bipartite multigraph. Let $\mc S$ be the set of simple finite factor graphs. We have
\begin{align}
        \pr[M\in \mc S\mid \rv {\hat m}]\geq\lt(1-\frac{k}{n}\rt)^{\rv{\hat m}}\label{eq:prS-M}
\end{align}
and $\rv{\hat m}=O(n)$ a.a.s., and hence $\pr[M\in \mc S]=\Omega(1)$. Choose the random graphs $(F(\C),F(\C'),F(\C''))$ from the conditional joint distribution of $(M,M',M'')$ given $M\in \mc S$.

  

The following lemma shows that $F(\C'')$ is contiguous with $F(\B_{XY})$ up to permuting columns. Let $\pi$ be a uniform random permutation of $V_{2n}$, and let $\tilde \B_{XY}$ be the matrix obtained by permuting columns of $\B_{XY}$ according to $\pi$.
\begin{lemma}\label{lem:BXYmcontigF(C'')}
    For any $\eps>0$ and $m=m(n)$, we have $F(\tilde \B_{XY})\mcontig F(\C'')$
\end{lemma}
We also show that $F(\C')$ is contiguous with $G^{(1)}_{(2-2\eps)n,\rv m,\rv k,\rv d}$ for an appropriate choice of parameters $\rv k$ and $\rv d$. Given $\mu_1,\mu_2>0$ and $\lambda\in [0,1]$, we use 
$\lambda \Po(\mu_1)+(1-\lambda) \Po(\mu_2)
$ to denote the distribution obtained by choosing from $\Po(\mu_1)$ with probability $\lambda$, and choosing from $\Po(\mu_2)$ with probability $1-\lambda$.
\begin{lemma}\label{lem:G(1)mcontigF(C')}
    Let $\eps>0$ and let $d^*>0$. Let $\rv k\sim \Bin(k,1-\eps)$, and let  $\rv d\sim \Po(\rho d^*)/2+\Po((1-\rho)d^*)/2$. Then $G^{(1)}_{(2-2\eps)n,\rv m,\rv k,\rv d}\contig F(\tilde \C')$.
\end{lemma}
We defer the proofs of Lemmas~\ref{lem:BXYmcontigF(C'')} and~\ref{lem:G(1)mcontigF(C')} to Section~\ref{sec:contiguityproof}, which also contains the proofs of the contiguity results discussed in Section~\ref{sec:contiguity}. Using the lemmas, we now prove Lemma~\ref{lem:sat-nonprincipal}.
\begin{proof}[Proof of Lemma~\ref{lem:sat-nonprincipal}]
Suppose $m<(d_{k,\rho}-\eps')n/k$ for some fixed $\eps'>0$. Set $d^*=(1-\eps'/2)d_{k,\rho}$, and set $\rv d\sim \Po(\rho d^*)/2+\Po((1-\rho)d^*)$. Setting $\rv k_\eps\sim\Bin(k,1-\eps)$ for $\eps>0$, we have
\begin{equation*}
    \lim_{\eps\to 0}\Phi_{\rv k_\eps,\rv d}(z)=e^{-\rho d^*z^{k-1}}+e^{-(1-\rho)d^*z^{k-1}}-d^*(1-1/k)z^k+d^*z^{k-1}-d^*/k=\Phi_{d^*,k,\rho}(z)
\end{equation*}
Since $d^*<d_{k,\rho}$, we have $\Phi_{d^*,k,\rho}(z)<\Phi_{d^*,k,\rho}(0)$ for all $z\in (0,1]$. Taking derivatives shows that $\Phi_{d^*,k,\rho}^{(i)}(0)=0$ for $i\in[k-1]$ and $\Phi_{d^*,k,\rho}^{(k)}(0)<0$, so it follows that $\Phi_{\rv k_\eps,\rv d}(z)<\Phi_{\rv k_\eps,\rv d}(0)$ for all $z\in (0,1]$ for some sufficiently small $\eps>0$. Set $\rv k= \rv k_\eps$, and take this choice of $d^*$, $\eps$ in the definition of $\C$, $\C'$, and $\C''$.

Theorem~\ref{thm:rank-fields} shows that the $\{0,1\}$-matrix $\A$ with $F(\A)\sim G^{(1)}_{(2-2\eps)n,\rv m,\rv k,\rv d}$ has full row rank a.a.s. for this choice of $\rv k,\rv d$. By Lemma~\ref{lem:G(1)mcontigF(C')}, it follows that $\C '$ has full row rank a.a.s. On the event that $\rv m>m$ and $\C'$ has full row rank, $\C'$ contains every row of $\C$, and so $\C$ also has full row rank. Since $\rv m\sim \Po(d^*n/k)$ and $d^*/k>(d_{k,\rho}-\eps')/k>m/n$, we have $\rv m>m$ a.a.s. Therefore $\C$ has full row rank a.a.s., and the submatrix $\C''$ has full row rank a.a.s. Lemma~\ref{lem:BXYmcontigF(C'')} then shows that $\B_{XY}$ has full row rank a.a.s.
\end{proof}

\subsection{Unsatisfiability}
In this section we prove that $\B_{XY}x=\b_{XY}$ is a.a.s.\ unsatisfiable in the supercritical case. 
Let $\B'$ and $\b'$ be the submatrix and subvector of $\B_{XY}$ and $\b$ respectively by taking the first $d^*n/k$ rows for some $d^*\approx d_{k,\rho}$ which will be specified later. Recall from Section~\ref{sec:proofoverview} that $(\B')_{\text{2-core}}$ is the largest submatrix of $\B'$ where every column contains at least two non-zero entries and every row containts exactly $k$ non-zero entries. As in the proof of Theorem~\ref{thm:units},  we will bound the size of $|\ker (\B')_{\text{2-core}}|$ and then show that a.a.s.\ all the non-zero kernel vectors in $\ker (\B')_{\text{2-core}}$ will no longer be in the kernel of a new matrix by by adding $\Omega(n)$ more rows whose supports are contained in the columns of $(\B')_{\text{2-core}}$. We use a theorem in~\cite[Theorem 1.2]{coja2020rank}, restated below in a form that is easy to apply in our setting, to estimate the size of $(\B')_{\text{2-core}}$.  To state the theorem we need a few definitions.

     Suppose $\boldsymbol{d}_n=(d_1,\ldots,d_n)$ and $\boldsymbol{k}_n=(k_1,\ldots, k_{m(n)})$ are sequences of tuples, where $m(n)$ is a function of $n$. Suppose $(\boldsymbol{d}_n,\boldsymbol{k}_n)_{n\in\NN}$ satisfies the following condition 
\begin{equation}
\exists r>2, C>0:\
\sum_{i=1}^n d^r_i\le Cn, \quad \sum_{i=1}^{m(n)} k^r_i\le Cm(n) \quad \text{for all $n$.}\label{eq:dkassumptions}
\end{equation} 
Let $\pi$ and $\mu$ be defined by
\begin{equation}
\pi(j)=\lim_{n\to\infty} \sum_{i=1}^n \frac{\ind{d_i=j}}{n},\quad \mu(j)=\lim_{n\to\infty} \sum_{i=1}^m \frac{\ind{k_i=j}}{m}.\label{eq:dkempirical}
\end{equation}
Then it is easy to see that~\eqref{eq:dkassumptions} implies that $\pi$ and $\mu$ are well-defined and that both of them are probability distribution functions. Let $\bar d=\ex_{\d\sim \pi}[\d]$ and let $\bar k=\ex_{\boldsymbol{k}\sim\mu}[\boldsymbol{k}]$. Let $D$ be the probability generating function of $\pi$ and let $K$ be the probability generating function of $\mu$. Let
\begin{equation*}
    \phi(\alpha)=1-\alpha-D'(1-K'(\alpha)/k)/d
\end{equation*}
and let
\begin{equation*}
    \eta=\eta_{\pi,\mu}=\max\{x\in [0,1]:\phi(x)=0\}.
\end{equation*}

For each tuple $\boldsymbol{d}_n=(d_1,\ldots, d_n)$ and $\boldsymbol{k}_n=(k_1,\ldots, k_m)$, let $B_{\boldsymbol{k}_n,\boldsymbol{d}_n}$ be a random binary matrix chosen uniformly from the set of matrices with $k_i$ nonzero entries in row $i$ and $d_j$ nonzero entries in column $j$ for each $i\in[m(n)],j\in[n]$.

\begin{theorem}\label{thm:2core-size} Suppose that $(\boldsymbol{d}_n,\boldsymbol{k}_n)$ satisfies~\eqref{eq:dkassumptions}. Let $B\sim B_{\boldsymbol{d}_n,\boldsymbol{k}_n}$.
Let $n^*$ and $m^*$ be the number of columns and rows of $B_{\text{2-core}}$. Then,
\[
\frac{n^*}{n}\to1-D\lt(1-\frac{K'(\eta)}{\bar k}\rt)-\frac{K'(\eta)}{\bar k}D'\lt(1-\frac{K'(\eta)}{\bar k}\rt),\quad \frac{m^*}{n}\to\frac{\bar d}{\bar k}K(\eta).
\]
\end{theorem}
\proof The proof is identical to the proof of~\cite[Theorem 1.2]{coja2020rank}.\qed

We can now prove Theorem~\ref{thm:nonprincipal-linear}.

\begin{proof}[Proof of Theorem~\ref{thm:nonprincipal-linear}]

Let $\eps>0$. By Lemma~\ref{lem:reduction-nonprincipal}, it suffices to show that $\B_{XY}x=\b_{XY}$ is satisfiable a.a.s. for $m<(d_{k,\rho}-\eps)n/k$, and unsatisfiable a.a.s. for $m>(d_{k,\rho}+\eps)n/k$.
    
The system $\B_{XY}=\b_{XY}$ is satisfiable a.a.s. in the subcritical case $m<(d_{k,\rho}-\eps)n/k$ by Lemma~\ref{lem:sat-nonprincipal}.

For the supercritical case, $m>(d_{k,\rho}+\eps)n/k$, consider the subsystem $\B'x=\b'$ consisting of the first $d^*n/k$ rows, where $d^*=d_{k,\rho}-\eps_n$ and $\eps_n=o(1)$. By choosing $\eps_n$ so that $\eps_n\to 0$ sufficiently slowly, we can ensure that $\B'$ has full row rank by Lemma~\ref{lem:sat-nonprincipal}.

To use a similar approach to Theorem~\ref{thm:units}, we first determine the asymptotic size of $(\B ')_{\text{2-core}}$. Every row of $\B'$ has $k$ nonzero entries, and the number of nonzero entries in each column of $\B'$ converges to a sequence of i.i.d. copies of $\rv d$, where $\rv d\sim \Po(\rho d^*)/2+\Po((1-\rho) d^*)/2$ and $\ex[\rv d]=d^*/2$. Therefore, if $n^*$ and $m^*$ are the number of rows and columns in the 2-core of $\B'$, applying Theorem~\ref{thm:2core-size} shows that

    \begin{equation}\label{eq:BXY-2core}
        \frac{n^*}{n}\to 2-e^{-\eta^{k-1}}-e^{-(1-\rho)d^*\eta^{k-1}}-d^*\eta^{k-1}+d^*\eta^k,\quad \frac{m^*}{n}\to \frac{d^*}{k}\eta^k
    \end{equation}
    in probability. Hence,
    \begin{align*}
        \frac{n^*-m^*}{n}&\to 1-e^{-\rho d_{k,\rho}\eta_{d_{k,\rho},k}^{k-1}}-e^{-(1-\rho)d_{k,\rho}\eta_{d_{k,\rho},k}^{k-1}}-d_{k,\rho}\eta_{d_{k,\rho},k}^{k-1}+d_{k,\rho}\eta_{d_{k,\rho},k}^k-\frac{d_{k,\rho}}{k}\eta_{d_{k,\rho},k}^k\\
        &=\Phi_{d,k,\rho}(0)-\Phi_{d,k,\rho}(\eta_{d_{k,\rho},k})=0
    \end{align*}
    The remainder of the proof follows a similar argument to the proof of Theorem~\ref{thm:units}, which we outline below. 
    
    Since $\B'$ has full row rank a.a.s., $(\B')_{\text{2-core}}$ also has full row rank a.a.s. Therefore the number of solutions to $(\B')_{\text{2-core}}x=(\b')_{\text{2-core}}$ is $q^{(n^*-m^*)}=q^{o(n)}$ a.a.s. Let $R'$ be the set of rows in $\B_{XY}$ that are not in $\B'$ and have all of their non-zero entries in the columns of $(\B')_{\text{2-core}}$. Since $n^*=\Theta(n)$ and $m-(d_{k,\rho}-\eps_n)n/k=\Omega(n)$ a.a.s., we can show that $|R'|=\Omega(n)$ a.a.s. Since $\supp(Q)=\{\bar X\bar Y,0\}$, the assumption {\bf{(A2)}} is satisfied for $(P^{\otimes k}, Q)$, and we can show that $\pr[(\B_{XY})_ax=(\b_{XY})_a]<1-\delta$ for all $x\in R^{V_{2n}}$ and $a\in R'$, where $\delta>0$ does not depend on $x$ or $a$. Therefore, the event that one of the $q^{o(n)}$ solutions to $(\B')_{\text{2-core}}x=(\b')_{\text{2-core}}$ also satisfies $(\B_{XY})_ax=(\b_{XY})_a$ for all $a\in R'$ has probability $q^{o(n)}(1-\delta)^{\Omega(n)}$. Hence,
    \begin{align*}
        \pr[\exists x\in R^n:(\B_{XY})_{\text{2-core}}x=(\b_{XY})_{\text{2-core}}]\leq q^{o(n)}(1-\delta)^{\Omega(n)}=o(1).
    \end{align*}
    Since the subsystem $(\B_{XY})_{\text{2-core}}x=(\b_{XY})_{\text{2-core}}$ is unsatisfiable a.a.s., it follows that $\B_{XY}x=\b_{XY}$ is unsatisfiable a.a.s.
\end{proof}

\section{Proofs of Theorems~\ref{thm:contiguity} and~\ref{thm:asympequiv} and Lemmas~\ref{lem:BXYmcontigF(C'')} and~\ref{lem:G(1)mcontigF(C')}}\label{sec:contiguityproof}

Recall the following models from Section~\ref{sec:contiguity}.
\begin{align*}
        G_{\rv m} &= G_{n,\rv m,\boldsymbol{k}}\\
    M_{\rv m} &= M_{n,\rv m,\boldsymbol{k}}\\
    M^{(2)}_{\rv m} &=M^{(2)}_{n,\rv m,\rv k,\rv d}\\
    G^{(2)}_{\rv m}&=G^{(2)}_{n,\rv m,\rv k,\rv d}\\
    G^{(1)}_{\rv m}&=G^{(1)}_{n,\rv m,\rv k,\rv d}.
\end{align*}
For simplicity we suppress the subscripts except for $\rv m$. We write $G_m, M_m, M_m^{(2)}, G_m^{(2)}$ and $G_m^{(1)}$ for the models with deterministic $m=m(n)$.

Recall the event $\mc D_n$ from Section~\ref{sec:contiguity}. Let $\mc S$ be the set of simple finite factor graphs.
For all $\mc M\in \{G_{\rv m},M_{\rv m}, M_{\rv m}^{(2)}, G_{\rv m}^{(2)}, G_{\rv m}^{(1)}\}$, let 
\begin{align*}
&c({\mc M})\ \text{denote the number of constraint vertices in $\mc M$}\\
&e({\mc M})\ \text{denote the number of edges in $\mc M$}\\
&\text{deg}({\mc M})\ \text{denote the degree sequence of $\mc M$}.
\end{align*}
Observe that $c(G_{\rv m}), c(M_{\rv m})$ has the same distribution as $\rv m$; $c(M_{\rv m}^{(2)}), c(G_{\rv m}^{(1)})$ has the same distribution as $\rv m\mid \mc D_n$, and $c(G_{\rv m}^{(2)})$ has the same distribution as $\rv m\mid (\mc D_n\cap \{M_{\rv m}^{(2)}\in\mc S\})$. The same subtle change in distribution resulting from conditioning on  $\mc D_n$ or $\mc D_n\cap \{\mc M\in\mc S\}$ occurs to $e(\mc M)$ and $\deg(\mc M)$ as well. For this reason we drop the subscript $\rv m$ from $C_{\rv m}$, since the number of constraint vertices $c(\mc M)$ does not necessarily have the same distribution as $\rv m$, and so the subscript no longer carries useful information. For symmetry we drop $n$ from $V_n$ as well.

Given a model $\mc M$, we may write $\deg(\mc M)$ as $(\ka,\da)$ or more specifically $((\ka_a)_{a\in C},(\da_v)_{v\in V})$ to intentionally distinquish it from $(\rv k_a)_{a\in C}$ and $(\rv d_v)_{v\in V}$ (without conditioning on $\mc D_n$ and/or $\mc M\in \mc S$), as their distributions are different.

\subsection{Proofs of the first statements of Theorems~\ref{thm:contiguity} and~\ref{thm:asympequiv} }

We begin this section by proving the first statement of Theorem~\ref{thm:asympequiv}, as stated in the following lemma. \begin{lemma}\label{lem:asympequivfixed}
    Let $\{\mc E_n\}_{n\in\NN}$ be a sequence of sets of factor graphs. Then
    \begin{equation*}
        \lim_{n\to\infty}\lt|\pr[G^{(1)}_{m}\in \mc E_n]-\pr[G^{(2)}_{m}\in \mc E_n]\rt|=0.
    \end{equation*}
\end{lemma}
The distributions of $G^{(1)}_{m}$ and $G^{(2)}_{m}$ are the same when conditioned on having a fixed degree sequence, so the only task lies in showing that the distributions of $\text{deg}(G^{(1)}_{m})$ and $\text{deg}(G^{(2)}_{m})$ are sufficiently close to each other. Observe that $\text{deg}(G^{(1)}_{m})$ has the same distribution as $\deg(M^{(2)}_{m})$. But conditioning on the event $M^{(2)}_m\in \mc S$  causes the distribution of $\deg(G^{(2)}_{m})$ to drift from being i.i.d.\ copies of $\rv d$ conditioned only on $\mathcal{D}_n$. We use a method-of-moments argument similar to~\cite[Lemma 4.3]{coja2020rank} to calculate the probability that $M^{(2)}_{m}\in \mc S$ (for typical degree sequences), and then use this to show that  this drift of distribution is small.

The first statement of Theorem~\ref{thm:contiguity} follows from Lemma~\ref{lem:asympequivfixed} and the following lemma.

\begin{lemma}\label{lem:contiguity}
    Suppose $\rv d$ is a Poisson random variable. Then $G_{m}\mcontig G^{(2)}_{m}$.
\end{lemma}

The proofs of the second statements of Theorems~\ref{thm:contiguity} and~\ref{thm:asympequiv}, where $m$ is replaced by a random variable $\rv m$, are presented in Section~\ref{sec:contiguity-random}.
In Section~\ref{sec:contiguity-couplinglemmas}, we use an adaptation of the proofs for Theorems~\ref{thm:contiguity} and~\ref{thm:asympequiv} to prove Lemmas~\ref{lem:BXYmcontigF(C'')} and~\ref{lem:G(1)mcontigF(C')}.

We use following local limit theorem, which is a special case of~\cite[Theorem 3.5.3]{Durrett2019}, as a tool in the following subsections.
\begin{theorem}[\cite{Durrett2019}]\label{thm:LLT}
    Suppose that $\{X_i\}_{i\in\NN}$ is a sequence of i.i.d.\ integer-valued random variables such that the greatest common divisor of the support of $X_i$ is one, and  $\Var[X_i]=\sigma^2\in(0,\infty)$. Then
    \begin{equation*}
    \lim_{n\to\infty}\sup_{j\in\ZZ}\lt|\sqrt{n}\pr\lt[\sum_{i=1}^nX_i=j\rt]-\frac{e^{\frac{-(j-n\ex[X_i])^2}{2n\sigma^2}}}{\sqrt{2\pi}\sigma}\rt|=0.
    \end{equation*}
\end{theorem}

\subsection{Proof of Lemma~\ref{lem:asympequivfixed}}\label{sec:asympequivfixed-proof}
Recall from~\eqref{eq:D} that $\mc D_n$ is the event
\begin{equation*}
    \sum_{i=1}^n\rv d_{v_i}=\sum_{i=1}^{ m}\rv k_{a_i},
\end{equation*}
where $\{\rv k_{a_i}\}_{i\in\NN}$ and $\{\rv d_{v_i}\}_{i\in\NN}$ are mutually independent, such that $\{\rv k_{a_i}\}_{i\in\NN}$ are i.i.d.\ copies of $\rv k$, $\{\rv d_{v_i}\}_{i\in\NN}$ are i.i.d.\ copies of $\rv d$, and $(\rv d,\rv k)$ satisfies condition ({\bf H}1). Recall that $m=m(n)$ is a sequence of positive integers satisfying ({\bf H}2)--({\bf H}4). 

We begin by estimating the probability of the event $\mc D_n$. 
\begin{lemma}\label{lem:prD}

    If $\Var(\rv k)=\Var(\rv d)=0$, then $\pr[\mc D_n]=1$. Otherwise $\pr[\mc D_n]=\Theta(n^{-1/2})$.
\end{lemma}
\begin{proof}

    If $\Var(\rv k)=\Var(\rv d)=0$, then $nd=mk$ by ({\bf H}4) and we have $\pr[\mc D_n]=1$ trivially.

    If $\Var (\rv k)>0$, then by Theorem~\ref{thm:LLT},
    \begin{align*}
        \pr\lt[\mc D_n\rt]&\geq \sum_{j:|n  d-j|\leq \sqrt{n}}\pr\Big[\sum_{a\in C_m}\rv k_a=j\Big]\pr\Big[\sum_{v\in V_n}\rv d_v=j\Big]\\
        &=\Omega(1/\sqrt{n})\pr\Big[\Big|\sum_{v\in V_n}\rv d_v-dn\Big|<\sqrt{n}\Big]\\
        &=\Omega(1/\sqrt{n}),
    \end{align*}
    and the uniform upper bound from Theorem~\ref{thm:LLT} gives
    \begin{equation*}
        \sum_{j=0}^\infty\pr\Big[\sum_{a\in C_m}\rv k_a=j\Big]\pr\Big[\sum_{v\in V_n}\rv d_v=j\Big]=O(1/\sqrt{n})
    \end{equation*}
    Therefore $\pr[\mc D_n]=\Theta(1/\sqrt{n})$. A similar argument applies in the case where $\Var(\rv d)>0$.
\end{proof}

Let $\mc M'_n$ be the event that $\{\rv k_a\}_{a\in C}$, $\{\rv d_v\}_{v\in V}$ satisfies
\begin{equation}
        \max_{a\in C}\rv k_a+\max_{v\in V} \rv d_v\leq \sqrt{n}/\log n.\label{eq:defeventM}
\end{equation}

Since $\ex[\rv k^{2+\eps}],\ex[\rv d^{2+\eps}]<\infty$ for some $\eps>0$, it follows easily that
\begin{equation}
\pr[\rv k>\sqrt{n}/(2\log n)]=o(1/n),\quad \pr[\rv d>\sqrt{n}/(2\log n)]=o(1/n),\label{eq:max-k-d}
\end{equation}
which immediately implies that
\begin{equation}
\pr[\mc M'_n]=1-o(1).
\end{equation}
We also need the following conditional probability.

\begin{lemma}\label{lem:M'aas} 
    $\pr[\mc M_n'\mid \mc D_n]=1-o(1)$.
\end{lemma}
\begin{proof}
We show that 
    \begin{equation}
    \pr[\rv k_a>\sqrt{n}/(2\log n)\mid \mc D_n]=o(1/n)\label{eq:M'aasWTS}
    \end{equation}
    for all $a\in C_m$, and a similar argument shows that $\pr[\rv d_v>\sqrt{n}/(2\log n)\mid \mc D_n]=o(1/n)$ for all $v\in V_n$. The claim then follows by a union bound.

    If $\Var(\rv k)=0$, then clearly $\pr[\rv k_a>\sqrt{n}/(2\log n)\mid \mc D_n]=o(1/n)$. Suppose $\Var(\rv k)>0$. If $\Var(\rv d)=0$, then by Theorem~\ref{thm:LLT}
    \begin{align*}
        \pr[\{\rv k_a>\sqrt{n}/(2\log n)\}\cap \mc D_n]&=\sum_{j=\sqrt{n}/(2\log n)}^\infty\pr[\rv k_a=j]\pr\lt[\sum_{a'\in C_m\setminus\{a\}}\rv k_a=dn-j\rt]\\
        &=O(1/\sqrt{n})\pr[\rv k>\sqrt{n}/(2\log n)].
    \end{align*}
    Similarly, if $\Var(\rv d)>0$, then by Theorem~\ref{thm:LLT}
    \begin{align*}
        \pr[\{\rv k_a> \sqrt{n}/(2\log n)\}\cap\mc D_n]&=\sum_{j=0}^\infty\pr\lt[\rv k_a>\sqrt{n}/(2\log n),\sum_{a\in C_m}\rv k_a=j\rt]\pr[\sum_{v\in V_n}\rv d_v=j]\\
        &=O(1/\sqrt{n})\pr[\rv k>\sqrt{n}/(2\log n)].
    \end{align*}
    By Lemma~\ref{lem:prD}, $\pr[\mc D_n]=\Theta(1/\sqrt{n})$, and so by Bayes' rule and~\eqref{eq:max-k-d}, in both cases above, 
    \[
    \pr[\rv k_a> \sqrt{n}/(2\log n)\mid \mc D_n]=O(\pr[\rv k>\sqrt{n}/(2\log n)])=o(1/n),
    \]
    which verifies~\eqref{eq:M'aasWTS}. 
\end{proof}
As a consequence of Lemmas~\ref{lem:prD} and~\ref{lem:M'aas}, 
\begin{equation}\label{eq:conditionalprD}
    \pr[\mc D_n\mid \mc M_n']=\begin{cases}
        1-o(1)&\Var(\rv k)=\Var(\rv d)=0\\
        \Theta(1/\sqrt{n})&\text{Otherwise}.
    \end{cases}
\end{equation}
The proof of the following lemma is the same as~\cite[Claim 4.4]{coja2020rank}, except that~\eqref{eq:conditionalprD} is used in place of~\cite[Lemma 4.2]{coja2020rank}. Thus, we omit its proof.
\begin{lemma}\label{lem:empvar}
    Conditioned on $\mc M_n'\cap \mc D_n$, we have $\frac{1}{n}\sum_{v\in V_n} (\rv d_v)^2\to \ex [\rv d^2]$ and \newline $\frac{1}{n}\sum_{a\in C_m} (\rv k_a)^2\to d\ex [\rv k^2]/k$ in probability.
\end{lemma}

We also know that $\frac{1}{n}\sum_{v\in V_n} \rv d_v\to\ex[\rv d]$ in probability conditioned on $\mc D_n$, which follows by the fact that $\sum_{v\in V_n} \rv d_v=n\ex [\rv d] +O(n^{2/3})$ with probability $1-o(n^{-1})$ and Lemma~\ref{lem:prD}. Therefore there is a function $r(n)=o(1)$ such that  the degree sequence $(\ka,\da)$  of $M^{(2)}_{m}$ a.a.s.\ satisfies
\begin{gather}
    \max_{a\in C}\ka_a+\max_{v\in V} \da_v\leq \sqrt{n}/\log{n},\label{eq:maxdeg}\\
     \lt|\frac{1}{n}\sum_{v\in V} \da_v-\ex[\rv d]\rt|,\ \lt|\frac{1}{n}\sum_{a\in C} \ka_a- d\ex[\rv k]/ k\rt|,\ \lt|\frac{1}{n}\sum_{v\in V} \da_v^2-\ex [\rv d^2]\rt|,\ \lt|\frac{1}{n}\sum_{a\in C} \ka_a^2- d\ex [\rv k^2]/ k\rt|\leq r(n).\label{eq:varcvg}
\end{gather}
 Let $\mc M_n$ be the event that~\eqref{eq:maxdeg} and~\eqref{eq:varcvg} hold.

To calculate the probability that $M^{(2)}_{m}$ is simple, we show that the number of parallel edges of $M^{(2)}_{m}$ is asymptotically Poisson on $\mc M_n$. We use an argument by the method of moments  similar to the one used in~\cite{coja2020rank}.

\begin{lemma}\label{lem:moments}
    Let $Y$ be the number of parallel edges in $M^{(2)}_{m}$ and let
    \begin{equation*}
    \lambda=\frac{(\ex[\rv d^2]- d)(\ex[\rv k^2]- k)}{2 d\,  k}.
\end{equation*}
Then for all fixed $\ell\geq 1$, on $\mc M_n$ we have
    \begin{equation}
        \ex\lt[ \prod_{i=0}^{\ell-1}(Y-i)\,\middle|\,\{\ka_a\}_{a\in C},\{\da_v\}_{v\in V}\rt]=\lambda^\ell+o(1).\label{eq:moments}
    \end{equation}
\end{lemma}
\begin{proof}
    Let $U_\ell$ be the set of $\ell$-tuples $(c_i,u_i)_{i\in[\ell]}\in (C\times V)^\ell$ such that $\ka_{c_i},\da_{u_i}\geq 2$ for all $i\in [\ell]$ and such that $(c_1,u_1),\ldots, (c_\ell,u_\ell)$ are distinct. For $(c_i,u_i)_{i\in[\ell]}\in U_\ell$, let $Y[(c_i,u_i)_{i\in[\ell]}]$ be the indicator of the event that there are exactly two parallel edges between $c_i$ and $u_i$ in $M^{(2)}_{m}$ for all $i\in[\ell]$. Let $Z$ be the number of $\ell$-tuples of distinct multiedges in $M^{(2)}_{m}$ where at least one multiedge has multiplicity at least three. We then have
    \begin{equation*}
        \prod_{i=0}^{\ell-1}(Y-i)=\sum_{(c_i,u_i)_{i\in[\ell]}\in U_\ell}Y[(c_i,u_i)_{i\in[\ell]}]+O(Z).
    \end{equation*}
    By a standard first-moment argument it is easy to show that $\ex [Z\mid \{\ka_a\}_{a\in C},\{\da_v\}_{v\in V}]=o(1)$. For $(c_i,u_i)_{i\in[\ell]}\in U_\ell$, we have
    \begin{align*}
        \ex[Y[(c_i,u_i)_{i\in[\ell]}]\mid \{\ka_a\}_{a\in C},\{\da_v\}_{v\in V}]&=\frac{\lt(\sum_{v\in V}\da_v-2\ell\rt)!}{\lt(\sum_{v\in V}\da_v\rt)!}\prod_{i\in[\ell]}2\binom{\da_{u_i}}{2}\binom{\ka_{c_i}}{2}
    \end{align*}
    where $\prod_{i\in[\ell]}2\binom{\da_{u_i}}{2}\binom{\ka_{c_i}}{2}$ counts the number of ways to match two vertex-copies in $\{c_i\}\times [\ka_{c_i}]$ with two vertex-copies in $\{u_i\}\times [\da_{v_i}]$ for $i\in[\ell]$, and $\frac{\lt(\sum_{v\in V}\da_v-2\ell\rt)!}{\lt(\sum_{v\in V}\da_v\rt)!}$ is the probability that all of these pairs are contained in a uniform random matching $\Gamma$. On the event~\eqref{eq:maxdeg}, summing over $(c_i,u_i)_{i\in[\ell]}\in U_\ell$ gives
    \begin{equation*}
        \sum_{(c_i,u_i)_{i\in[\ell]}\in U_\ell}\prod_{i\in[\ell]}2\binom{\da_{u_i}}{2}\binom{\ka_{c_i}}{2}=\lt(\frac{1}{2}\sum_{a\in C,v\in V}\da_v(\da_v-1)\ka_a(\ka_a-1)\rt)^\ell+o(n^{2\ell}),
    \end{equation*}
    and on the event~\eqref{eq:varcvg}, 
    \begin{equation*}
        \frac{\lt(\sum_{v\in V}\da_v-2\ell\rt)!}{\lt(\sum_{v\in V}\da_v\rt)!}\sim\frac{1}{( dn)^{2\ell}},\quad \sum_{a\in C,v\in V}\da_v(\da_v-1)\ka_a(\ka_a-1)\sim\frac{ dn^2}{ k}(\ex[\da^2]- d)(\ex [\ka^2]- k).
    \end{equation*}
    Therefore
    \begin{equation*}
        \ex\lt[\sum_{(c_i,u_i)_{\in[\ell]}\in U_\ell} Y[(c_i,u_i)_{\in[\ell]}]\mid \{\ka_a\}_{a\in C},\{\da_v\}_{v\in V}\rt]\sim\lt(\frac{(\ex[\da^2]- d)(\ex[\ka^2]- k)}{2 d\,  k}\rt)^\ell=\lambda^\ell.
    \end{equation*}
    Since $\ex [Z\mid \{\ka_a\}_{a\in C},\{\da_v\}_{v\in V}]=o(1)$, this proves~\eqref{eq:moments}.
\end{proof}

\begin{lemma}\label{lem:prS}
    On $\mc M_n$, we have $\pr[M^{(2)}_{m}\in \mc S\mid \{\ka_a\}_{a\in C},\{\da_v\}_{v\in V}]\sim e^{-\lambda}$. As a consequence, $\pr[M^{(2)}_{m}\in \mc S]\sim e^{-\lambda}$.
\end{lemma}
\begin{proof}
    By Lemma~\ref{lem:moments} and the method of moments (eg. \cite[Theorem 1.22]{Bollobas2001}),
    \begin{equation*}
        \pr[M^{(2)}_{m}\in \mc S\mid \{\rv k_a\}_{a\in C},\{\da_v\}_{v\in V}]=\pr[Y=0\mid \{\ka_a\}_{a\in C},\{\da_v\}_{v\in V}]\sim e^{-\lambda}
    \end{equation*}
    on $\mc M_n$. Since $\pr[\mc M_n]=1-o(1)$, we also get $\pr[M^{(2)}_{m}\in \mc S]\sim e^{-\lambda}$.
\end{proof}

We can now prove Lemma~\ref{lem:asympequivfixed}.
\begin{proof}[Proof of Lemma~\ref{lem:asympequivfixed}]
    Let $H$ be a simple factor graph with degree sequence $\{k_a\}_{a\in C}$, $\{d_v\}_{v\in V}$. Given a fixed degree sequence, the conditional distributions of $M^{(2)}_{m}\mid \{M^{(2)}_{m}\in \mc S\}$ and $G^{(1)}_{m}$ are both uniform on simple factor graphs with the given degree sequence. Therefore
    \begin{align}
        &\pr[M^{(2)}_{m}=H\mid M^{(2)}_{m}\in\mc S]\nonumber\\
&\hspace{2cm}=\pr[M^{(2)}_{m}=H\mid M^{(2)}_{m}\in\mc S, \{\ka_a\}=\{k_a\},\{\da_v\}=\{d_v\}] \nonumber\\
&\hspace{2.5cm}\times \pr(\{\ka_a\}=\{k_a\},\{\da_v\}=\{d_v\}\mid M^{(2)}_{m}\in\mc S)
\nonumber\\     &\hspace{2cm}=\pr[G^{(1)}_{m}=H\mid \{\ka_a\}=\{k_a\},\{\da_v\}=\{d_v\}] \nonumber\\
&\hspace{2.5cm} \frac{\pr[M^{(2)}_{m}\in\mc S\mid \{\ka_a\}=\{k_a\},\{\da_v\}=\{d_v\}]}{\pr[M^{(2)}_{m}\in\mc S]} \pr_{M^{(2)}_{m}}[\{\ka_a\}=\{k_a\},\{\da_v\}=\{d_v\}]\nonumber\\
&\hspace{2cm}=\pr[G^{(1)}_{m}=H\mid \{\ka_a\}=\{k_a\},\{\da_v\}=\{d_v\}] \nonumber\\
&\hspace{2.5cm} \frac{\pr[M^{(2)}_{m}\in\mc S\mid \{\ka_a\}=\{k_a\},\{\da_v\}=\{d_v\}]}{\pr[M^{(2)}_{m}\in\mc S]} \pr[\{\rv k_a\}=\{k_a\},\{\rv d_v\}=\{d_v\}\mid {\mc D}_n]\nonumber\\
&\hspace{2cm}=\pr[G^{(1)}_{m}=H]\frac{\pr[M^{(2)}_{m}\in\mc S\mid \{\ka_a\}=\{k_a\},\{\da_v\}=\{d_v\}]}{\pr[M^{(2)}_{m}\in\mc S]}
.\label{eq:equiv1}
    \end{align}
 Lemma~\ref{lem:prS} shows that on $\mc M_n$,
    \begin{align*}
        \pr[M^{(2)}_{m}\in\mc S\mid \{\ka_a\},\{\da_v\}]\sim \pr[M^{(2)}_{m}\in\mc S].
    \end{align*}
    Therefore, for any sequence of sets of graphs $\{\mc E_n\}_{n\in \NN}$,
    \begin{equation*}
        \lt|\pr[\{M^{(2)}_{m}\in \mc E_n\}\cap \mc M_n\mid M^{(2)}_{m}\in \mc S]-\pr[\{G^{(1)}_{m}\in \mc E_n\}\cap \mc M_n]\rt|=o(1).
    \end{equation*}
    By Lemma~\ref{lem:prS}, $\pr[M^{(2)}_{m}\in\mc S]=\Omega(1)$, and so $\pr[\mc M_n\mid M^{(2)}_{m}\in\mc S]=1-o(1)$. Therefore $\pr[G^{(2)}_{m}\in \mc E_n]=\pr[G^{(1)}_{m}\in \mc E_n]+o(1)$.
\end{proof}

\subsection{Proof of Lemma~\ref{lem:contiguity}}\label{sec:contiguityproof-lemma}

As in the previous section, we assume that $(\rv d,\rv k)$ satisfies condition ({\bf H}1), and assume that $m=m(n)$ is a sequence of positive integers satisfying ({\bf H}2)--({\bf H}4). Observe that for $\mc M\in \{G_m, M_m\}$, $\deg(\mc M)$ has exactly the same distribution as $((\rv d_v)_{v\in V}, (\rv k_a)_{a\in C})$. Thus, for these two models, we directly write $(\rv k,\rv d)$ for $(\da,\ka)$. 
We first prove the following asymptotic equivalence statement.
\begin{lemma}\label{lem:G-Mcontig}
    For any sequence of sets of factor graphs $\{\mc E_n\}_{n\in\NN}$, we have 
    \begin{equation}\label{eq:G-Mcontig}
        \lim_{n\to\infty} \lt|\pr[G_{m}\in \mc E_n]-\pr[M_{m}\in \mc E_n\mid M_{m}\in \mc S]\rt|=0.
    \end{equation}
    Moreover, $\pr[M_{m}\in \mc S]=\Omega(1)$.
\end{lemma}
\begin{proof}
Let $\hat k=\ex[(\rv k-1)\ind{\rv k\geq 1}]$, and let $\rv m(j)=|\{a\in C:\rv k_a=j\}|$ for $j\in\NN$. Let $\mc T_n$ be the event 
\begin{equation*}
    |\sum_{a\in C}(\rv k_a-1)\ind{\rv k_a\geq 1}-\hat kdn/k|<n^{2/3},\quad \max _{a\in C} \rv k_a\leq \sqrt{n}/\log n.
\end{equation*}
Since $m=dn/k+O(\sqrt{n})$ by ({\bf H}2) and $\ex[\rv k^{2+\eps}]<\infty$ for some $\eps>0$, it is easy to show that $\pr[\mc T_n]=1-o(1)$ by standard concentration arguments. For each $a\in C$, the probability that $a$ is not incident to any parallel edges in $M_{m}$ is $\prod_{j=1}^{\rv k_a-1}(1-\frac{j}{n})$. Hence, on $\mc T_n$,
\begin{align*}
    \pr[M_{m}\in \mc S\mid \{\rv k_a\}_{a\in C}]&=\prod_{a\in C_m}\prod_{j=1}^{\rv k_a-1}(1-\frac{j}{n})=\prod_{j\ge 1}(1-\frac{j}{n})^{\rv m(j+1)}\\
    &=\prod_{j\ge 1}\exp\left(-\rv m(j+1)\frac{j}{n}+O(\rv m(j+1) j^2/n^2)\right)\\
    &=\exp\left(-\sum_{j\ge 1}\rv m(j+1)\frac{j}{n} +o(1) \right),
\end{align*}
since $\sum_{j\ge 1}\rv m(j+1) j^2/n^2 =O(n)\cdot (n/\log^2 n)/n^2=o(1)$ on ${\mc T}_n$.

Since $\ex[\rv k^{2+\eps}]<\infty$ for some $\eps>0$, we have
\begin{equation*}
    \frac{\sum_{a\in C}\rv (\rv k_a-1)\ind{\rv k_a\geq 1}}{n}\to  d\hat k/ k
\end{equation*}
in probability by a standard concentration argument. Therefore $\pr[M_{m}\in \mc S\mid \{\rv k_a\}_{a\in C}]=(1-o(1))e^{- d\hat k/ k}$ on $\mc T_n$. Since $\pr[\mc T_n]=1-o(1)$, we also have $\pr[M_{m}\in \mc S]=(1-o(1))e^{- d\hat k/ k}$.

We now prove~\eqref{eq:G-Mcontig}. Note that conditioned on a fixed degree sequence, $G_{m}$ and $M_{m}\mid \{M_{m}\in \mc S\}$ are identically distributed. Following a similar argument as in~\eqref{eq:equiv1}, for any simple graph $H$ with degree sequence $\{k_a\}_{a\in C}$, we have
\begin{equation*}
    \pr[M_{m}=H\mid M_{m}\in \mc S]=\frac{\pr[M_{m}\in \mc S\mid \{\rv k_a\}_{a\in C}=\{k_a\}_{a\in C}]}{\pr[M_{m}\in\mc S]}\pr[G_{m}=H],
\end{equation*}
and on $\mc T_n$,
\begin{equation*}
    \pr[M_{m}\in \mc S\mid \{\rv k_a\}_{a\in C}]\sim \pr[M_{m}\in\mc S].
\end{equation*}
Therefore for all sequences of sets of graphs $\{\mc E_n\}_{n\in\NN}$,
\begin{equation*}
    \pr[\{M_{m}\in \mc E_n\}\cap \mc T_n\mid M_{m}\in \mc S]=(1-o(1))\pr[\{G_{m}\in \mc E_n\}\cap \mc T_n].
\end{equation*}
We have $\pr[\mc T_n]=1-o(1)$, and since $\pr[M_{m}\in \mc S]=\Theta(1)$ we also have $\pr[\mc T_n\mid M_{n}\in \mc S]=1-o(1)$. Therefore $\pr[M_{m}\in \mc E_n\mid M_{m}\in \mc S]=\pr[G_{m}\in \mc E_n]+o(1)$.
\end{proof}

\begin{lemma}\label{lem:MmcontigMmidD}
    Let $Z_n\sim\Po(dn)$ be independent of $M_{m}$. Then 
    $$M_{m}\mcontig \Big(M_{m}\mid \sum_{a\in C}\rv k_a=Z_n\Big).$$
\end{lemma}
\begin{proof}
Let $\mc T_n(K)$ be the event $\{|\sum_{a\in C}\rv k_a- dn|\leq K\sqrt{n}\}$. We prove the following two claims.
\begin{claim}
\begin{equation*}\label{claim:Mclaim1}
    \lim_{K\to\infty}\liminf_{n\to\infty}\pr[\mc T_n(K)]=\lim_{K\to\infty}\liminf_{n\to\infty}\pr\Big[\mc T_n(K)\mid \sum_{a\in C}\rv k_a=Z_n\Big]=1.
\end{equation*}
\end{claim}
Recall that $e(H)$ denotes the number of edges in a graph $H$.
\begin{claim}\label{claim:Mclaim2}
Let $K>0$ be fixed. There exists a constant $\delta=\delta(K)$ for which the following holds:  If $H_n$ is a sequence of factor graphs such that $|e(H_n)- d n| \le K\sqrt{n}$ then  
$$\delta\, \pr[M_{m}=H_n] \le \pr[M_{m}=H_n\mid \sum_{a\in C}\rv k_a=Z_n]\le \delta^{-1} \pr[M_{m}=H_n].$$
\end{claim}
We postpone the proofs of the two claims above. By Claim~\ref{claim:Mclaim1}, for any fixed $K>0$ and any sequence of sets of graphs $\{\mc E_n\}_{n\in \NN}$,
\begin{equation}
    \pr[\{M_{m}\in \mc E_n\}\cap \mc T_n(K)\mid \sum_{a\in C}\rv k_a=Z_n]=\Theta_K(1)\pr[\{M_{m}\in \mc E_n\}\cap \mc T_n(K)],\label{eq:equiv2}
\end{equation}
where $\Theta_K(1)$ denotes the set of sequences $(a_n)$ such that there exist positive-valued functions functions $f,g$ for which $f(K)<|a_n|<g(K)$ for all $n$. We show that $M_{m} \mcontig (M_{m}\in \mc E_n\mid \sum_{a\in C} \rv k_a=Z_n)$ as desired follows from~\eqref{eq:equiv2} and Claim~\ref{claim:Mclaim1}. For every $\eps>0$ let $K>0$ be chosen such that $\liminf_{n\to\infty} \pr(\mc T_n(K)), \liminf_{n\to\infty} \pr(\mc T_n(K)\mid \sum_{a\in C} \rv k_a=Z_n)>1-\eps$. Then, for every sequence of sets of graphs $\{\mc E_n\}_{n\in \NN}$, if $\pr[M_{m}\in \mc E_n\mid \sum_{a\in C} \rv k_a=Z_n]=o(1)$ then $\pr[\{M_{m}\in \mc E_n\}\cap \mc T_n(K)\mid \sum_{a\in C}\rv k_a=Z_n]=o(1)$ and so by~\eqref{eq:equiv2},  $\pr[\{M_{m}\in \mc E_n\}\cap \mc T_n(K) ]=o(1)$, which implies that $\pr[M_{m}\in \mc E_n]\le \eps+o(1)$. Thus, by sending $\eps\to 0$ we obtain that $(M_{m}\in \mc E_n\mid \sum_{a\in C} \rv k_a=Z_n) \contig M_{m}$. The proof for $  M_{m}\contig (M_{m}\in \mc E_n\mid \sum_{a\in C} \rv k_a=Z_n)$ is by a symmetric argument. 

\begin{proof}[Proof of Claim~\ref{claim:Mclaim1}]
If $\Var(\rv k)=0$, then $\sum_{a\in C} \rv k_a= km= dn+O(\sqrt{n})$, and so 
\[
\pr[T_n(K)]=\pr[\mc T_n(K)\mid \sum_{a\in C} \rv k_a=Z_n]=1
\]
for sufficiently large $n$ and $K$. 

Now suppose $\Var(\rv k)>0$. Since $\sum_{a\in C} \rv k_a$ is a sum of $m=dn/k+O(\sqrt{n})$ i.i.d.\ random variables by ({\bf H}2), we have $\lim_{K\to\infty}\liminf_{n\to\infty}\pr[T_n(K)]=1$ by the central limit theorem.

Apply Theorem~\ref{thm:LLT} to $\sum_{a\in C_m} \rv k_a$ and use the probability mass function of $Z_n$, it follows that $\pr[\sum_{a\in C} \rv k_a=Z_n]=\Theta(1/\sqrt{n})$.
 Therefore
\begin{equation*}
    \pr[\mc T_n(K)\mid \sum_{a\in C} \rv k_a=Z_n]= 1-\Theta(\sqrt{n})\sum_{\substack{j\in\NN\\|j-dn|\geq K\sqrt{n}}}\pr[\sum_{a\in C} \rv k_a=j]\pr[Z_n=j].
\end{equation*}

By Theorem~\ref{thm:LLT} (applied to  $\sum_{a\in C}\rv k_a$) and the fact that 
$\lim_{K\to \infty}\pr[|Z_n-dn|\geq K\sqrt{n}]=0$, it follows immediately that 

\begin{equation*}
    \lim_{K\to\infty}\liminf_{n\to\infty}\pr[\mc T_n(K)\mid \sum_{a\in C} \rv k_a=Z_n]=1.
\end{equation*}
\end{proof}
\begin{proof}[Proof of Claim~\ref{claim:Mclaim2}]
    If $|e(H_n)-dn|\le K\sqrt{n}$, then $\pr[Z_n=e(H_n)]=\Theta_K(1/\sqrt{n})$. We have already shown that $\pr[\sum_{a\in C} \rv k_a=Z_n]=\Theta(1/\sqrt{n})$ (the implicit constant bounds here are independent of $K$).  Then we have
    \begin{align*}
        \pr[M_{m}=H_n\mid \sum_{a\in C} \rv k_a=Z_n]&=\frac{\pr[M_{m}=H_n]\pr[e(H_n)=Z_n]}{\pr[\sum_{a\in C} \rv k_a=Z_n]}=\Theta_K(1)\pr[M_{m}=H_n].
    \end{align*}
\end{proof}
\end{proof}

Nest, we relate the random multigraphs $M_{m}$ and $M^{(2)}_{m}$. We need the following definition, which relates the distribution of  $\deg(M_{m})$.
\begin{definition}
    Let $S$ be a finite set and fix $Z\geq 0$. Let $\{X_i\}_{i\in[Z]}$ be i.i.d.\ uniform random variables on $S$. Define the distribution $\Mult(S,Z)$ on $\ZZ^S$ to be the distribution of $\{|\{i\in[Z]:X_i=s\}|\}_{s\in S}$.
\end{definition}
The following claim is used to compare the distributions of  $\deg(M_{m})$ and $\deg(M^{(2)}_{m})$.
\begin{claim}\label{claim:multinomialdist}
    Let $S$ be a finite set and let $Z\geq 0$. Let $\{\rv g_s\}_{s\in S}$ be i.i.d.\ $\Po(\lambda)$ random variables for some $\lambda>0$. Then the conditional distribution $\{\rv g_s\}_{s\in S}\mid \{\sum_{s\in S}\rv g_s=Z\}$ is $\Mult(S,Z)$.
\end{claim}
\begin{proof}
    Let $\{h_s\}_{s\in S}\in \NN^S$ be such that $\sum_{s\in S}h_s=Z$. Then
    \begin{align*}
        \pr[\Mult(S,Z)=\{h_s\}_{s\in S}] = \binom{Z}{h_s,s\in S}n^{-Z}=\frac{Z!}{\prod_{s\in S}h_s!}|S|^{-Z}.
    \end{align*}
    On the other hand, since $\sum_{s\in S}\rv g_s\sim \Po(\lambda|S|)$, we have
    \begin{align*}
        \pr\Big[\{\rv g_s\}_{s\in S}=\{h_s\}_{s\in S}\mid \sum_{s\in S} \rv g_s=Z\Big]=\prod_{s\in S} \frac{\lambda^{h_s}e^{-\lambda}}{h_s!}\cdot \frac{Z!}{(\lambda|S|)^Ze^{-\lambda|S|}}=\frac{Z!}{\prod_{s\in S}h_s!}|S|^{-Z}.
    \end{align*}
    Therefore $\{\rv g_i\}_{i\in[n]}\mid \{\sum_{i=1}^n\rv g_i=Z\}\sim\Mult(S,Z)$.
\end{proof}

Recall that we denote the degree sequence of a factor graph $G$ by $\deg(G)$ and the number of edges by $e(G)$. 
For any $a\in C$, $v\in V$, we denote by $e(a,v)$ the number of edges between $a$ and $v$.

\begin{lemma}\label{lem:MmidD=M'}
    Suppose $\rv d$ is a Poisson random variable, and let $Z_n\sim \Po(dn)$ be independent of $M_{m}$. Then $M^{(2)}_{m}$ and $M_{m}\mid \{e(M_m)=Z_n\}$ are identically distributed.
\end{lemma}
\begin{proof}

Let $G$ be a bipartite graph on $C\cup V$. Note that the distribution of the degree sequence $(\ka_a)_{a\in C}$ of the constraint vertices of  $M_{m}$ is exactly that of $(\rv k_a)_{a\in C}$. Since $Z_n$ is independent of $M_{m}$, we have
    \begin{align}
        &\pr[M_{m}=G\mid e(M_m)=Z_n]=\pr[M_{m}=G\mid  \deg(M_{n})=\deg(G)]\nonumber\\
        &\hspace{0.2cm}\times \pr_{M_m}[\da_v=\deg_G(v)\ \forall v\in V\mid \ka_a=\deg_G(a)\ \forall a\in C]\cdot        \pr_{M_m}[\ka_a=\deg_G(a)\ \forall a\in C\mid  e(M_m)=Z_n].\label{eq:telescrop1}
    \end{align}
and
   \begin{align}
  &\pr[M^{(2)}_{m}=G]=\pr[M_{m}^{(2)}=G\mid \deg(M^{(2)}_{m})=\deg(G)]\nonumber\\
        &\hspace{0.5cm}\times
\pr_{M_m^{(2)}}[\da_v=\deg_G(v)\ \forall v\in V\mid \ka_a=\deg_G(a)\ \forall a\in C]\cdot \pr_{M_m^{(2)}}[\ka_a=\deg_G(a)\ \forall a\in C].\label{eq:telescrop2}
    \end{align}
    
    Therefore it suffices to prove the following three claims which verify that each of the three probabilities in~\eqref{eq:telescrop1} is equal to the corresponding probability in~\eqref{eq:telescrop2}. 
    \begin{claim}\label{M-M2-conditional}
        \begin{align*}
        &\pr[M_{m}=G\mid  \deg(M_{n})=\deg(G)]=\pr[M_{m}^{(2)}=G\mid \deg(M^{(2)}_{m})=\deg(G)].
    \end{align*}
    \end{claim}
    \begin{proof}
        We first calculate $\pr[M_{m}=G\mid \ka_a=\deg_G(a)\ \forall a\in C]$. The number of ways to choose $(u_1,\ldots, u_{\deg_G(a)})$ so that $\{au_i:i\in[\deg_G(a)]\}$ is precisely the multi-set of edges incident to $a\in C$ in $G$ is $\deg_G(a)!/\prod_{v\in V}e(a,v)!$, and the probability of making each of these choices is $n^{-\deg_G(a)}$. Since $\{(u_1,\ldots, u_{\deg_G(a)})\}_{a\in C}$ are chosen independently by $M_{m}$, we get
    \begin{equation*}
        \pr[M_{m}=G\mid \ka_a=\deg_G(a)\ \forall a\in C]=\pr[M_{m}=G\mid \rv k_a=\deg_G(a)\ \forall a\in C]=\frac{\prod_{a\in C} \deg_G(a)!}{n^{e(G)}\prod_{a\in C,v\in V}e(a,v)!}.
    \end{equation*}
    For $\pr[\da_v=\deg_G(v)\ \forall v\in V\mid \ka_a=\deg_G(a)\ \forall a\in C]$, the number of ways to choose $\{(u_1,\ldots, u_{\deg_G(a)})\}_{a\in C}$ so that $\da_v=\deg_G(v)$ for all $v\in V$ is $e(G)!/\prod_{v\in V} \deg_G(v)!$. So
    \begin{align*}
        \pr[\da_v=\deg_G(v)\ \forall v\in V\mid \ka_a=\deg_G(a)\ \forall a\in C]=\frac{e(G)!}{n^{e(G)}\prod_{v\in V_n} \deg_G(v)!}.
    \end{align*}
    Therefore
    \begin{align*}
        \pr[M_{m}=G\mid \deg(M_{m})=\deg(G)]&=\frac{\prod_{a\in C}\deg_G(a)!\prod_{v\in V}\deg_G(v)!}{e(G)!\prod_{a\in C,v\in V}e(a,v)!}.
    \end{align*}
    Now consider $M^{(2)}_{m}$, which is chosen according to the pairing model. Let $c(a)=\{a\}\times [\deg_G(a)]$ and $c(v)=\{v\}\times [\deg_G(v)]$ be the sets of vertex copies of $a\in C$ and $v\in V$. Let $\mc G$  be the set of matchings $\Gamma$ between $\bigcup_{a\in C} c(a)$ and $\bigcup_{v\in V} c(v)$ that contain exactly $e(a,v)$ edges between $c(a)$ and $c(v)$ for each $a\in C$, $v\in V$. Then the set of $\prod_{a\in C}\deg_G(a)!\prod_{v\in V}\deg_G(v)!$ of permutations of vertex copies acts transitively on $\mc G$, and every matching in $\mc G$ is stabilized by the $\prod_{a\in C,v\in V}e(a,v)!$ permutations that map edges between $c(a)$ and $c(v)$ to other edges between $c(a)$ and $c(v)$ for each $a\in C, v\in V$. Therefore
    \begin{equation*}
        |\mc G|=\frac{\prod_{a\in C}\deg_G(a)!\prod_{v\in V}\deg_G(v)!}{\prod_{a\in C,v\in V}e(a,v)!}.
    \end{equation*}
    Since there are $e(G)!$ matchings in total, the claim follows.
    \end{proof}
    Recall again that  
    \begin{equation}
    \deg(M_{m}^{(2)})\sim ((\rv d_v)_{v\in V_n},\ (\rv k_a)_{a\in C_m})\mid \mc D_n.\label{eq:Mm2}
    \end{equation}

    \begin{claim}
        \begin{equation*}
            \pr_{M_m}[\da_v=\deg_G(v)\ \forall v\in V\mid \ka_a=\deg_G(a)\ \forall a\in C]=\pr_{M_m^{(2)}}[\da_v=\deg_G(v)\ \forall v\in V\mid \ka_a=\deg_G(a)\ \forall a\in C].
        \end{equation*}
    \end{claim}
    \begin{proof}
        Since the endpoints $\{(u_1,\ldots, u_{\ka_a})\}_{a\in C}=\{(u_1,\ldots, u_{\deg_G(a)})\}_{a\in C}$ are chosen independently $M_{m}$, we have $(\da_v)_{v\in V}\mid \{\ka_a=\deg_G(a)\ \forall a\in C\}\sim \Mult(V,e(G))$ in $M_m$. On the other hand, in $M_m^{(2)}$, 
        \begin{align*}
   (\da_v)_{v\in V}\mid \{\ka_a=\deg_G(a)\ \forall a\in C\} 
            &\sim \{\rv d_v\}_{v\in V}\mid \lt\{\sum_{v\in V} \rv d_v=e(G)\rt\},
        \end{align*}
        by~\eqref{eq:Mm2},
        and $\{\rv d_v\}_{v\in V}\mid \sum_{v\in V} \rv d_v=e(G)\sim \Mult(V,e(G))$ by Claim~\ref{claim:multinomialdist}.
    \end{proof}
    \begin{claim}
    \begin{equation*}
\pr_{M_m}[\ka_a=\deg_G(a)\ \forall a\in C\mid  e(M_m)=Z_n]=\pr_{M_m^{(2)}}[\ka_a=\deg_G(a)\ \forall a\in C].
    \end{equation*}
    \begin{proof}
        Since $\{\rv d_v\}_{v\in V}$ are i.i.d.\ $\Po(d)$ random variables independent of $\{\rv k_a\}_{a\in C}$, we have 
        \begin{equation*}
            \Big(\{\rv k_a\}_{c\in C},\sum_{v\in V}\rv d_v\Big)\sim (\{\rv k_a\}_{a\in C},Z_n),
        \end{equation*} 
        where $\{\rv k_a\}_{a\in C}$ and $Z_n$ are mutually independent.
        Therefore, in $M_m$,
        \begin{align*}
            (\ka_a)_{a\in C}\mid \{e(M_{m})=Z_n\}\sim (\rv k_a)_{a\in C}\mid \{e(M_{m})=Z_n\}\sim \{\rv k_a\}_{a\in C}\mid \lt\{\sum_{a\in C}\rv k_a=\sum_{v\in V} \rv d_v\rt\},
        \end{align*}
        whereas the right hand side has exactly the same distribution as $(\ka_a)_{a\in C}$ in $M_m^{(2)}$.
    \end{proof}
    \end{claim}
\end{proof}

We can now complete the proof of Lemma~\ref{lem:contiguity}.
\begin{proof}[Proof of Lemma~\ref{lem:contiguity}]
    Let $\{\mc E_n\}_{n\in\NN}$ be a sequence of sets of factor (multi)graphs, and let $\mc E_n'=\mc E_n\cap  \mc S$. We have $\pr[G_m\in \mc E_n']=\pr[G_m\in \mc E_n]$ and $\pr[G^{(2)}_{n,\rv m,\rv k,\rv d}\in \mc E_n']=\pr[G^{(2)}_{n,\rv m,\rv k,\rv d}\in \mc E_n]$. By Lemma~\ref{lem:G-Mcontig},
    \begin{align*}
        \pr[G_m\in \mc E_n']=\frac{\pr[M_m\in \mc E_n']}{\pr[M_m\in \mc S]}+o(1)
    \end{align*}
    and $\pr[M_m\in \mc S]=\Theta(1)$. Hence $\pr[G_m\in \mc E_n']=o(1)$ if and only if $\pr[M_m\in \mc E_n']=o(1)$. By Lemmas~\ref{lem:MmcontigMmidD} and \ref{lem:MmidD=M'}, $\pr[M_m\in \mc E_n']=o(1)$ if and only if $\pr[M_m^{(2)}\in \mc E_n']=o(1)$. Finally, by Lemma~\ref{lem:prS},
    \begin{align*}
        \pr[G_m^{(2)}\in \mc E_n']=\frac{\pr[M_m^{(2)}\in \mc E_n']}{\pr[M_m^{(2)}\in \mc S]}=\Theta(1)\pr[M_m^{(2)}\in \mc E_n'].
    \end{align*}
    Hence $\pr[M_m^{(2)}\in \mc E_n']=o(1)$ if and only if $\pr[G_m^{(2)}\in \mc E_n']=o(1)$.
\end{proof}

\subsection{Proofs of the second statements of Theorems~\ref{thm:contiguity} and~\ref{thm:asympequiv}}
\label{sec:contiguity-random}
In this section we prove the statements of Theorems~\ref{thm:contiguity} and~\ref{thm:asympequiv} about $G_{\rv m}$, $G^{(1)}_{\rv m}$, and $G^{(2)}_{\rv m}$. Recall that $c(F)$ is the number of constraint vertices in a factor graph $F$. Note that the distribution of $c(G_{\rv m})$ and $c(M_{\rv m})$ is $\Po(dn/k)$, but $c(M^{(2)}_{\rv m})$ and $c(G^{(1)}_{\rv m})$ are not Poisson due to the conditioning on $\mc D_n$. Moreover, the distribution of $c(G^{(2)}_{\rv m})$ is different from $c(G_{\rv m})$, $c(M_{\rv m})$, $c(M^{(2)}_{\rv m})$, $c(G^{(1)}_{\rv m})$, as $G^{(2)}_{\rv m}$ is conditioned on both $\mc D_n$ and $\{M^{(2)}_{\rv m}\in\mc S\}$, i.e.\ the event that $M^{(2)}_{\rv m}$ has no parallel edges.


\begin{lemma}\label{lem:mconcentration}
    $\lim_{K\to\infty}\limsup_{n\to\infty}\pr[|\rv m-dn/k|>K\sqrt{n}\mid\mc D_n]=0$.
\end{lemma}
\begin{proof}

We consider multiple cases based on the variance of $\rv k$ and $\rv d$. If $\Var(\rv k)=\Var(\rv d)=0$, then $\pr[\sum_{i=1}^{m}\rv k_{a_i}=\sum_{i=1}^n \rv d_{v_i}]=0$ for $m\neq dn/k$, and so $\pr[\rv m= dn/ k\mid\mc D_n]=1$.

Now suppose $\Var (\rv d)=0$ and $\Var(\rv k)>0$. Then
    \begin{align*}
        \pr[\{|\rv m-dn/k|>K\sqrt{n}\}\cap \mc D_n]&=\sum_{m:|m-dn/k|>K\sqrt{n}}\pr[\sum_{i=1}^{m}\rv k_{a_i}=dn]\pr[\rv m=m],
    \end{align*}
and Theorem~\ref{thm:LLT} gives
\begin{align*}
    \sum_{m:|m-dn/k|>K\sqrt{n}}\pr[\sum_{i=1}^{m}\rv k_{a_i}=dn]\pr[\rv m=m]&=O(1/\sqrt{n})\pr[|\rv m-dn/k|>K\sqrt{n}],
\end{align*}
where the implicit constant in $O(1/\sqrt{n})$ does not depend on $K$. Since $\Var(\rv k)>0$, we also have $\pr[\mc D_n]=\Omega(1/\sqrt{n})$ by conditioning on $\rv m$ and applying Lemma~\ref{lem:prD}. Hence $$\pr[|\rv m-dn/k|>K\sqrt{n}\mid \mc D_n]=\frac{\pr[\{|\rv m-dn/k|>K\sqrt{n}\}\cap \mc D_n]}{\pr[\mc D_n]}=O(1)\pr[|\rv m-dn/k|>K\sqrt{n}].$$ Since $\rv m'\sim \Po(dn/k)$, the claim of the lemma follows by the upper tail bound of a Poisson variable.

Similarly, in the case where $\Var(\rv d)>0$, Theorem~\ref{thm:LLT} gives
    \begin{align*}
        \pr[\{|\rv m-dn/k|>K\sqrt{n}\}\cap \mc D_n]&= \sum_{j\in\NN}\pr[\{|\rv m-dn/k|>K\sqrt{n}\}\cap\{\sum_{i=1}^{\rv m}\rv k_{a_i}=j\}]\pr[\sum_{i=1}^n\rv d_{v_i}=j]\\
        &=O(1/\sqrt{n})\pr[\{|\rv m-dn/k|>K\sqrt{n}\}],
    \end{align*}
and the rest of the proof follows exactly as in the previous case.
\end{proof}

\begin{lemma}\label{lem:mequiv}
    Let $\{\mc E_n\}_{n\in \NN}$ be a sequence of sets of integers. Then
    \begin{equation*}
        \lim_{n\to\infty}\lt|\pr[c(M^{(2)}_{\rv m})\in \mc E_n]-\pr[c(M^{(2)}_{\rv m})\in \mc E_n\mid M^{(2)}_{\rv m}\in \mc S]\rt|=0.
    \end{equation*}
\end{lemma}
\begin{proof}
    By conditioning on $c(M^{(2)}_{\rv m})$ and applying Lemma~\ref{lem:prS}, we have that for any $K>0$,
    \begin{align*}
        |\pr[M^{(2)}_{\rv m}\in \mc S]-e^{-\lambda}|&\leq \pr[|c(M^{(2)}_{\rv m})-dn/k|>K\sqrt{n}]\\
        &\hspace{1cm}+\sum_{m:|m-dn/k|\leq K\sqrt{n}}\lt|\pr[M^{(2)}_{m}\in \mc S]-e^{-\lambda}\rt|\pr [c(M^{(2)}_{\rv m}) = m]\\
        &=\pr[|c(M^{(2)}_{\rv m})-dn/k|> K\sqrt{n}]+o(1)\pr[|c(M^{(2)}_{\rv m})-dn/k|\leq K\sqrt{n}].
    \end{align*}
    Hence, by Lemma~\ref{lem:mconcentration}, $\pr[M^{(2)}_{\rv m}\in \mc S]\sim e^{-\lambda}$.

    Now let $\eps>0$. Since $\pr[M^{(2)}_{\rv m}\in \mc S]=\Theta(1)$, by Lemma~\ref{lem:mconcentration} there is a $K>0$ such that $\pr[|c(M^{(2)}_{\rv m})-dn/k|>K\sqrt{n}]+\pr[|c(M^{(2)}_{\rv m})-dn/k|>K\sqrt{n}\mid M^{(2)}_{\rv m}\in \mc S]<\eps+o(1)$. Hence
    \begin{align*}
        &\lt|\pr[c(M^{(2)}_{\rv m}) \in \mc E_n]-\pr[c(M^{(2)}_{\rv m}) \in \mc E_n\mid M^{(2)}_{\rv m}\in \mc S]\rt|\\
        &\hspace{1cm}\leq \eps+o(1)+\sum_{\substack{c(M^{(2)}_{\rv m})\in \mc E_n\\|m-dn/k|\leq K\sqrt{n}}}\lt(1-\frac{\pr[M^{(2)}_{\rv m}\in \mc S\mid c(M^{(2)}_{\rv m})=m]}{\pr[M^{(2)}_{\rv m}\in \mc S]}\rt)\pr[c(M^{(2)}_{\rv m})=m]
    \end{align*}
    By Lemma~\ref{lem:prS}, we have $\sup_{m:|m-dn/k|\leq K\sqrt{n}}|\pr[M^{(2)}_{\rv m}\in \mc S\mid c(M^{(2)}_{\rv m})=m]-e^{-\lambda}|=o(1)$. Since $\pr[M^{(2)}_{\rv m}\in \mc S]\sim e^{-\lambda}$, this shows that $\lt|\pr[c(M^{(2)}_{\rv m}) \in \mc E_n]-\pr[c(M^{(2)}_{\rv m}) \in \mc E_n\mid M^{(2)}_{\rv m}\in \mc S]\rt|=\eps+o(1)$, which proves the claim.
\end{proof}

\begin{lemma}\label{lem:mcontig}
 If $\Var(\rv d)>0$ and $\gcd(\rv d)=1$ then $\rv m\mcontig \{\rv m\mid \mc D_n\}$.
\end{lemma}
\begin{proof}
By Lemma~\ref{lem:mconcentration}, it suffices to show that, for all sequences $m(n)=dn/k+O(\sqrt{n})$, 
\begin{equation}\label{eq:mcontiguity}
\pr[\mc D_n\mid \rv m=m(n)]=\Theta(\pr [\mc D_n]),
\end{equation}
as this implies that
\[
    \pr[\rv m=m(n)\mid \mc D_n]=\frac{\pr[\mc D_n\mid \rv m=m(n)]}{\pr [\mc D_n]}\pr[\rv m=m(n)]=\Theta(1)\pr[\rv m=m(n)].
\]
Since $\Var(\rv d)>0$ and $\gcd(\rv d)=1$, by Lemma~\ref{lem:prD} we have $\pr[\mc D_n\mid \rv m=m(n)]=\Theta(1/\sqrt{n})$ for all sequences $m(n)=dn/k+O(\sqrt{n})$. Moreover, by conditioning on $\rv m$, Lemma~\ref{lem:prD} shows that $\pr[\mc D_n]=\Theta(1/\sqrt{n})$. Therefore~\eqref{eq:mcontiguity} holds.
\end{proof}

We prove the second statements in Theorems~\ref{thm:contiguity} and~\ref{thm:asympequiv} related to the models $G_{\rv m}$, $G_{\rv m}^{(1)}$, and $G_{\rv m}^{(2)}$ in the following two lemmas. This completes the proofs of Theorems~\ref{thm:contiguity} and~\ref{thm:asympequiv}.

\begin{lemma}\label{thm:contiguityrandom}
   Suppose $\rv d$ is a Poisson random variable. Then $G_{\rv m}\mcontig G^{(2)}_{\rv m}$.
\end{lemma}

\begin{proof}

Suppose $\{\mc E_n\}$ is a sequence of factor graphs such that $\pr[G_{\rv m}\in \mc E_n]=o(1)$ and let $\eps>0$. Then there is a sequence of sets of integers $\{\mc C_n\}_{n\in\NN}$ such that $\pr[\rv m\in\mc C_n]=1-o(1)$ and $\pr[G_{m(n)}\in \mc E_n]=o(1)$ for all sequences $m(n)$ such that $m(n)\in \mc C_n$. By Lemmas~\ref{lem:mconcentration} and~\ref{lem:mequiv}, there is a $K>0$ such that $\pr[|\rv m-dn/k|>K\sqrt{n}]<\eps+o(1)$ and $\pr[|c(M^{(2)}_{\rv m})-dn/k|>K\sqrt{n}\mid M^{(2)}_{\rv m}\in\mc S]<\eps+o(1)$. By Lemmas~\ref{lem:mcontig} and~\ref{lem:mequiv}, we have $\pr[c(M^{(2)}_{\rv m})\in \mc C_n\mid M^{(2)}_{\rv m}\in \mc S]=1-o(1)$, and hence
\begin{align*}
    \pr[G^{(2)}_{\rv m}\in \mc E_n]&\leq \eps+o(1)+\sum_{\substack{m\in \mc C_n\\|m-dn/k|\leq K\sqrt{n}}}\pr[G^{(2)}_{m}\in \mc E_n]\pr[c(M^{(2)}_{\rv m}) = m\mid M^{(2)}_{\rv m}\in \mc S]
\end{align*}
By the first statement of Theorem~\ref{thm:contiguity}, we have $\pr[G^{(2)}_{m(n)}\in \mc E_n]=o(1)$ for all sequences $m(n)$ with $m(n)\in \mc C_n\cap \{m:|m-dn/k|\leq K\sqrt{n}\}$.  Therefore
\begin{equation*}
    \sum_{\substack{m\in \mc C_n\\|m-dn/k|\leq K\sqrt{n}}}\pr[G^{(2)}_{m}\in \mc E_n]\pr[c(M^{(2)}_{\rv m}) = m\mid M^{(2)}_{\rv m}\in \mc S]
=o(1).
\end{equation*}
Since $\eps>0$ was arbitrary, this shows that $\{G_{\rv m}\}\contig\{G^{(2)}_{\rv m}\}$. A similar argument shows that $\{G^{(2)}_{\rv m}\}\contig\{G_{\rv m}\}$.
\end{proof}

\begin{lemma}\label{thm:asympequivrandom}
    Let $\{\mc E_n\}_{n\in\NN}$ be a sequence of sets of factor graphs. Then
    \begin{equation*}
        \lim_{n\to\infty}\lt|\pr[G^{(1)}_{\rv m}\in \mc E_n]-\pr[G^{(2)}_{\rv m}\in \mc E_n]\rt|=0.
    \end{equation*}
\end{lemma}

\begin{proof}

Let $\eps>0$. By Lemma~\ref{lem:mconcentration}, there is a $K>0$ be such that $\pr[|c(M^{(2)}_{\rv m})-dn/k|>K\sqrt{n}]+\pr[|\rv m-dn/k|>K\sqrt{n}]<\eps$. We have
\begin{align*}
    &\lt|\pr[G^{(1)}_{\rv m}\in \mc E_n]-\pr[G^{(2)}_{\rv m}\in \mc E_n]\rt|\leq\sum_{\substack{m\in\NN\\|m-dn/k|\leq K\sqrt{n}}}\lt|\pr[G^{(1)}_{m}\in \mc E_n]-\pr[G^{(2)}_{m}\in \mc E_n]\rt|\pr[c(G^{(1)}_{\rv m})=m]\\
    &\hspace{3cm}+\sum_{\substack{m\in\NN\\|m-dn/k|\leq K\sqrt{n}}}\pr[G^{(2)}_{m}\in \mc E_n]\lt|\pr[c(M^{(2)}_{\rv m})=m]-\pr[c(M^{(2)}_{\rv m})=m\mid M^{(2)}_{\rv m}\in \mc S]\rt|+\eps.
\end{align*}

By the first statement of Theorem~\ref{thm:asympequiv}, we have
\begin{equation*}
    \sum_{m:|m-dn/k|\leq C\sqrt{n}}\lt|\pr[G^{(1)}_{m}\in \mc E_n]-\pr[G^{(2)}_{m}\in \mc E_n]\rt|\pr[c(G^{(1)}_{\rv m})=m]=o(1),
\end{equation*}
and by Lemma~\ref{lem:mequiv},
\begin{equation*}
    \sum_{m:|m-dn/k|\leq C\sqrt{n}}\pr[G^{(2)}_{m}\in \mc E_n]\lt|\pr[c(M^{(2)}_{\rv m})=m]-\pr[c(M^{(2)}_{\rv m})=m\mid M^{(2)}_{\rv m}\in \mc S]\rt|=o(1).
\end{equation*}
Since $\eps>0$ was arbitrary, this shows that $\lt|\pr[G^{(1)}_{\rv m}\in \mc E_n]-\pr[G^{(2)}_{\rv m}\in \mc E_n]\rt|=o(1)$.
\end{proof}

\section{Proof of Lemmas~\ref{lem:BXYmcontigF(C'')} and~\ref{lem:G(1)mcontigF(C')}}\label{sec2:contiguityproof}

\subsection{Proof of Lemmas~\ref{lem:BXYmcontigF(C'')} and~\ref{lem:G(1)mcontigF(C')}}\label{sec:contiguity-couplinglemmas}

In this section, we use Lemma~\ref{thm:asympequivrandom} and adaptations of arguments from Section~\ref{sec:contiguityproof-lemma} to prove Lemmas~\ref{lem:BXYmcontigF(C'')} and~\ref{lem:G(1)mcontigF(C')}. In the following, a matching of $S$ with $T$ is a set of ordered pairs  $\Gamma\subset S\times T$ such that $|\{(s,t)\in \Gamma,t\in T\}|\leq 1$ for all $s\in S$ and $|\{(s,t)\in \Gamma,s\in S\}|\leq 1$ for all $t\in T$. A matching $\Gamma\subset S\times T$ is perfect if $|\Gamma|=|S|=|T|$. 
    Recall that $e_G(a,v)$ denotes the number of edges between vertices $a$ and $v$ in graph $G$. For any two sets of vertices $A,B$ in $G$, let $e_G(A,B)$ denote the number of edges in $G$ with at least one endpoint in $A$ and at least one endpoint in $B$. Recall $A,A',B,B'$ and 
    \begin{eqnarray*}
     M &=& M_{n,m,k,\rho,d^*,\eps};\\
     M' &=& M_{n,k,\rho,d^*,\eps}';\\
     M'' &=& M_{n,m,k,\rho,\eps}''.
    \end{eqnarray*}
    defined in Section~\ref{sec:linearnonprincipal-SAT}. Recall also that random variables $\rv m\sim \Po(d^*n/k)$ and $\rv{\hat m}=\max\{\rv m, m\} $ were involved in defining $M,M',M''$. 
We will also use the models 
\begin{align*}
    M^{(2)}&=M^{(2)}_{(2-2\eps)n,\rv m,\rv k,\rv d}\\
    G^{(1)}&=G^{(1)}_{(2-2\eps)n,\rv m,\rv k, \rv d}\\
    G^{(2)}&=G^{(2)}_{(2-2\eps)n,\rv m,\rv k, \rv d}
\end{align*}
that were defined in Section~\ref{sec:contiguity}, in the proof.

\begin{lemma}\label{lem:BXY-C''-contig} Let $\Gamma\subset (A\cup A')\times (B\cup B')$ be a maximum size matching of $A\cup A'$ with $B\cup B'$ chosen uniformly at random. Let $\mc S'$ be the event that $e_{M''}(a, \{v,v'\})\leq 1$ for all $a\in C$ and $(v,v')\in \Gamma$, and let $\mc C$ be the event that $|A\cup A'|=|B\cup B'|=n$. Let $\pi$ be a uniform random permutation of $V_{2n}$ chosen independently of $\B_{XY}$, and let $\tilde \B_{XY}$ be the matrix obtained from $\B_{XY}$ by permuting columns by $\pi$. Then 
$$F(\tilde \B_{XY})\sim M''\mid \{\mc S'\cap \mc C\},\ \quad \pr[\mc S'\cap \mc C]=\Omega(1).$$ 
\end{lemma}
\begin{proof}
    Let $\Gamma_\pi\subset \pi(V_n)\times \pi(V_{2n}\setminus V_n)$ be the matching given by $\Gamma_\pi=\{(\pi(v_i),\pi(v_{i+n}):i\in[n]\}$. Conditioned on $\mc C$, $A\cup A'$ is a uniformly chosen subset of $V_{2n}$ of size $n$. Hence, conditioned on $\mc C$, the distributions of $\Gamma$ and $\Gamma_\pi$ are identical. Therefore it suffices to show that
    \begin{equation}
        \pr[M''=G\mid \mc S',\mc C,\Gamma=\gamma]=\pr[ F(\tilde \B_{XY})=G\mid \Gamma_{\pi}=\gamma]\label{eq:BXY-WTS}
    \end{equation}
    for any perfect matching $\gamma\subset T\times(V_{2n}\setminus T)$ with $T\subset V_n$, and any factor graph $G$ with $e_G(a,\{v,v'\})\leq 1$ for all $(v,v')\in \gamma$, $a\in C$. For any such $\gamma$, $G$, we have
    \begin{align}
        \pr[M''=G\mid \Gamma=\gamma]&=\prod_{a\in C}\binom{k}{e_G(a,v),v\in V_{2n}}\rho^{e_G(a,T)}(1-\rho)^{e_G(a,V_{2n}\setminus T)}n^{-k}\nonumber\\
        &=\prod_{a\in C}k!\rho^{e_G(a,T)}(1-\rho)^{e_G(a,V_{2n}\setminus T)}n^{-k}\label{eq:prM''}
    \end{align}
    where $\binom{k}{e_G(a,v),v\in V_{2n}}$ is the number of ways to assign $b_{a,i}$, $i\in[k]$ so that $e_G(a,v)=|\{i\in [k]: b_{a,i}=v\}|$ for all $v\in V_{2n}$, and $\rho^{e_G(a,T)}(1-\rho)^{e_G(a,V_{2n}\setminus T)}n^{-k}$ is the probability of each assignment. The last line holds by the assumption that $e_G(a,\{v,v'\})\leq 1$ for all $(v,v')\in \gamma$, $a\in C_m$. Moreover, by considering each $a\in C_m$ independently, we have
    \begin{equation}
        \pr[\mc S'\mid \Gamma=\gamma] = \prod_{a\in C}\prod_{i=1}^k\lt(1-\frac{\rho(i-1)}{n}-\frac{(1-\rho)(i-1)}{n}\rt)=\lt(\frac{n!}{(n-k)!n^k}\rt)^m=\Omega(1).\label{eq:prS'}
    \end{equation}
    Since $\gamma$ is assumed to be a perfect matching, we have $\{\Gamma=\gamma\}\cap \mc C=\{\Gamma=\gamma\}$ as an equality of events. Thus, by combining~\ref{eq:prM''} and~\ref{eq:prS'},
    \begin{align}
        \pr[M''=G\mid \mc S',\mc C,\Gamma=\gamma]&=\pr[M''=G\mid \mc S',\Gamma=\gamma]\nonumber\\
        &= \lt(\frac{n!}{(n-k)!n^k}\rt)^{-m}\prod_{a\in C}k!\rho^{e_G(a,T)}(1-\rho)^{e_G(a,V_{2n}\setminus T)}n^{-k}\nonumber\\
        &=\prod_{a\in C}\binom{n}{k}^{-1}\rho^{e_G(a,T)}(1-\rho)^{e_G(a,V_{2n}\setminus T)}.\label{eq:BXY-M''conditional}
    \end{align}
    On the other hand, for any permutation $\sigma$ such that $\Gamma_\sigma=\gamma$,
    \begin{align}
        \pr[F(\tilde \B_{XY})=G\mid \pi=\sigma]=\prod_{a\in C}\binom{n}{k}^{-1}\rho^{e_G(a,T)}(1-\rho)^{e_G(a,V_{2n}\setminus T)}\label{eq:prTildeB}
    \end{align}
    where $\binom{n}{k}^{-1}$ is the probability of choosing the neighbourhood of $a$ in $F(\B)$ so that 
    \[
    \{i\in [n]: v_i\in N_{F(\B)}(a)\}=\{i\in[n]: \sigma(v_i)\in N_G(a)\text{ or }\sigma(v_{i+n})\in N_G(a)\},
    \] 
    and $\rho^{e_G(a,T)}(1-\rho)^{e_G(a,V_{2n}\setminus T)}$ is the probability of then assigning the coefficients $\B_{a,v}$ for $v\in N_{F(\B)}(a)$ so that $\sigma(\{v_i: \B_{a,v_i}=\bar X\})=N_G(a)\cap T$ and $\sigma(\{v_{i+n}: \B_{a,v_i}=\bar Y\})=N_G(a)\setminus T$. Summing~\eqref{eq:prTildeB} over all $\sigma$ with $\Gamma_\sigma=\gamma$ gives
    \begin{align*}
    \pr[F(\tilde \B_{XY})=G\mid \Gamma_{\pi}=\gamma]=\prod_{a\in C}\binom{n}{k}^{-1}\rho^{e_G(a,T)}(1-\rho)^{e_G(a,V_{2n}\setminus T)}.
    \end{align*}
    Combined with~\ref{eq:BXY-M''conditional}, this verifies~\ref{eq:BXY-WTS}. 
    
    To show that $\pr[\mc S']=\Omega(1)$, first note that summing over $\gamma$ in~\eqref{eq:prS'} gives $\pr[\mc S'\mid \mc C]=\Omega(1)$. Since $A'$ is chosen to minimize $||A\cup A'|-n|$, we have $|A\cup A'|=|B\cup B'|=n$ whenever $|A|,|B|\leq n$. Since $A$ is a uniformly chosen subset of $S$, where $|S|=(2-2\eps)n$, this occurs with probability $1-o(1)$. Therefore $\pr[\mc C]=1-o(1)$, and hence $\pr[\mc S'\cap \mc C]=\Omega(1)$.
\end{proof}

We now show that $M'$ is contiguous with the random graph $M^{(2)}=M^{(2)}_{(2-2\eps)n,\rv m,\rv k,\rv d}$, where $\rv k\sim \Bin(k,(1-\eps))$, $\rv d\sim \Po(\rho d^*)/2+\Po((1-\rho) d^*)/2$. Note that $\ex[\rv d](2-2\eps)n/\ex [\rv k]= d^*n/k$, so the distribution of $\rv m\sim \Po(d^*n/k)$ that appears in the definition for $M'$ coincides with the distribution of $\rv m\sim \Po(\ex[\rv d](2-2\eps)n/\ex [\rv k])$ that appears in the definition of $M^{(2)}$.

    Let $\rv \delta _v''\sim Po(d^*\rho)$ and $\rv \delta _v'''\sim Po(d^*(1-\rho))$ for $v\in V_{2n}$ be such that $\{\rv \delta _v''\}_{v\in V_{2n}}$, $\{\rv \delta _v'''\}_{v\in V_{2n}}$, and $M'$ are mutually independent. Let $\rv \delta_v'=\rv \delta_v''$ for $v\in A$ and $\rv \delta _v'=\rv \delta _v'''$ for $v\in B$. To make the distribution of the degree sequence of $M'$ match that of $M^{(2)}$, we condition $M'$ on the event
    \begin{equation*}
        \mc D_n'=\lt\{\sum_{a\in C} \rv \kappa_a=\sum_{v\in A}\rv \delta_v''+\sum_{v\in B} \rv \delta_v'''\rt\},
    \end{equation*}
where $\rv \kappa_a$ is the degree of $a$ in $M'$. The following Lemma, which is analogous to Lemma~\ref{lem:MmcontigMmidD}, shows that conditioning on $\mc D_n'$ does not significantly change the distribution of $M'$. The proof is similar to that of Lemma~\ref{lem:MmcontigMmidD}.
\begin{lemma}\label{lem:M'conditioning-contiguity}
    $M'\mcontig (M'\mid \mc D_n')$.
\end{lemma}
\begin{proof}
    Let $\mc T_n(C)$ be the event
    \begin{equation*}
        \mc T_n(C)=\lt\{\lt||A|-(1-\eps)n\rt |\leq C\sqrt{n},|\rv m-d^*n/k|\leq C\sqrt{n},\lt|\sum_{a\in C}\rv \kappa_a-(1-\eps)d^*n\rt|\leq C\sqrt{n}\rt\}
    \end{equation*}
    Similar to Lemma~\ref{lem:MmcontigMmidD}, the lemma follows from the following two claims, which are analogous to Claims~\ref{claim:Mclaim1} and~\ref{claim:Mclaim2}.
    \begin{claim}
        \begin{equation*}
            \lim_{K\to\infty}\liminf_{n\to\infty}\pr[\mc T_n(K)]=\lim_{K\to\infty}\liminf_{n\to\infty}\pr[\mc T_n(K)\mid \mc D_n']=1.
        \end{equation*}
    \end{claim}
    \begin{claim}
        Suppose $G_n$ is a sequence of factor graphs such that $|e(G_n)-(1-\eps)d^*n|=O(\sqrt{n})$ and $|c(G_n)-d^*n/k|=O(\sqrt{n})$. Suppose $a_n$ is such that $|a_n-(1-\eps)n|=O(\sqrt{n})$. Then
        \begin{equation*}
            \pr[M'=G_n,|A|=a_n\mid \mc D_n']=\Theta(1)\pr[M', |A|=a_n].
        \end{equation*}
    \end{claim}
\end{proof}

 The following lemma can be viewed as an analog of Lemma~\ref{lem:MmidD=M'}, which expresses $M^{(2)}$ as a conditional probability space (recall that $M_m^{(2)}$ in Lemma~\ref{lem:MmidD=M'} is the same as $M^{(2)}$ in this section). Recall that $\mc C$ is the event that $|A\cup A'|=|B\cup B'|=n$.  With slight abuse of notation, we define $M^{(2)}$ with $A\cup B$ (instead of $V_{(2-2\eps)n}$) being the set of $(2-2\eps) n$ variable vertices. Now $M^{(2)}$ and $M'$ are both defined on the same set of variable vertices.

\begin{lemma}\label{lem:M'-M2-equiv}
$(M'\mid \mc D_n',\mc C)\sim M^{(2)}$.
\end{lemma}

\begin{proof}
    Choose $(\hat A,\hat A',\hat B,\hat B',\{\rv {\hat\kappa} _a\}_{a\in C}, \{\rv {\hat \delta}_v\}_{v\in \hat A\cup \hat B}, \hat M'_{n,k,\rho,d^*,\eps})$ from the conditional distribution of $(A,A',B,B',\{\rv \kappa_a\}_{a\in C}, \{\rv \delta_v'\}_{v\in A\cup B}, M'_{n,k,\rho,d^*,\eps})$ given $\mc D_n'$. We first show that the distributions of the degree sequences of $\hat M'$ and $M^{(2)}$ agree.
    \begin{claim}\label{claim:M2degreesequence}
        \[(\{\rv {\hat \kappa_a}\}_{a\in C},\{\rv {{\hat\delta }}_{v}\}_{v\in \hat A\cup \hat B})\sim (\{\deg_{M^{(2)}}(a)\}_{a\in C},\{\deg_{M^{(2)}}(v)\}_{v\in A\cup B}).\]
    \end{claim}
    \begin{proof}
    First note that the union $A\cup B$ is independent of $\mc D_n'$, and thus $\hat A\cup \hat B\sim A\cup B$. 
    
    Conditioned on the event $\{A\cup B=T\}$ for any $T\subset V_{2n}$ of size $(2-2\eps)n$, the sequence $\{\rv \delta' _{v}\}_{v\in T}$ is a sequence of i.i.d. $\Po(\rho d^*)/2+\Po((1-\rho)d^*)/2$ random variables that is independent of $\{\rv \kappa_a\}_{a\in C}$. Moreover, $\{\rv \kappa_a\}_{a\in C}|\{A\cup B=T\}$ is a sequence of i.i.d. $\Bin(k,1-\eps)$ random variables. Therefore $(\{\rv \kappa_a\}_{a\in C},\{\rv \delta'_{\sigma_{A\cup B}(v)}\}_{v\in T})\mid \{A\cup B=T\}$ has the same joint distribution as $(\{\rv k_a\}_{a\in C},\{\rv d_v\}_{v\in T})$ in the notation of~\eqref{eq:D}. Therefore 
    \begin{align*}
        (\{\rv {\hat \kappa_a}\}_{a\in C},\{\rv {{\hat\delta }}_{v}\}_{v\in \hat A\cup \hat B})\mid \{\hat A\cup \hat B=T\}&\sim (\{\rv k_a\}_{a\in C},\{\rv d_v\}_{v\in T})\mid \Big\{\sum_{a\in C}\rv k_a=\sum_{v\in T}\rv d_v\Big\},\\
        &\sim (\{\deg_{M^{(2)}}(a)\}_{a\in C},\{\deg_{M^{(2)}}(v)\}_{v\in A\cup B})\mid \{A\cup B=T\}.
    \end{align*}
    Since $\hat A\cup \hat B\sim A\cup B$, this proves the claim.
    \end{proof}
    \begin{claim}\label{claim:M'degreesequence}
        \begin{equation*}
            (\{\rv {\hat\kappa_{a}}\}_{a\in C},\{\rv {{\hat\delta} }_{v}\}_{v\in \hat A\cup \hat B})\sim (\{\rv {\hat \kappa_{a}}\}_{a\in C},\{\deg_{\hat M'}(v)\}_{v\in \hat A\cup \hat B})\mid \mc C.
        \end{equation*}
    \end{claim}
    \begin{proof}
        Let $\alpha,\beta\subset V_{2n}$ be disjoint sets such that $|\alpha\cup \beta|=(2-2\eps)n$ and $|\alpha|,|\beta|\leq n$. Let $\{k_a\}_{a\in C}\in \NN^{C}$ and let $\{d_v\}_{v\in \alpha\cup \beta }\in\NN^{\alpha \cup \beta}$ such that $\sum_{v\in \alpha\cup \beta} d_v=\sum_{a\in C}k_a$. Let $M=\sum_{a\in C}k_a$. Then
        \begin{align*}
            &\pr[\{\rv {\hat \delta}_v\}_{v\in \hat A\cup \hat B}=\{d_v\}_{v\in \alpha\cup \beta}\mid \{\rv{\hat\kappa_a}\}_{a\in C}=\{k_a\}_{a\in C},\hat A=\alpha,\hat B=\beta]\\
            &=\frac{\pr[\{\rv \delta_v'\}_{v\in \hat A\cup \hat B}=\{d_v\}_{v\in \alpha\cup \beta},\{\rv\kappa_a\}_{a\in C}=\{k_a\}_{a\in C},A=\alpha,B=\beta,\mc D'_n]}{\pr[\{\rv\kappa_a\}_{a\in C}=\{k_a\}_{a\in C},A=\alpha,B=\beta,\mc D'_n]}\\
            &=\frac{\pr[\{\rv \delta_v''=d_v\}_{v\in \alpha},\{\rv \delta_v'''=d_v\}_{v\in \beta},\{\rv\kappa_a\}_{a\in C}=\{k_a\}_{a\in C},A=\alpha,B=\beta]}{\pr[\sum_{v\in \alpha}\rv \delta ''_v+\sum_{v\in\beta} \rv \delta'''_v=M,\{\rv\kappa_a\}_{a\in C}=\{k_a\}_{a\in C},A=\alpha,B=\beta]}\\
            &=\frac{\pr[\{\rv \delta_v''=d_v\}_{v\in \alpha},\{\rv \delta_v'''=d_v\}_{v\in \beta}]}{\pr[\sum_{v\in \alpha}\rv \delta ''_v+\sum_{v\in\beta} \rv \delta'''_v=M]}
        \end{align*}
        where the last equality follows from the independence of $\{\rv \delta ''_v\}_{v\in V_{2n}},\{\rv \delta'''_v\}_{v\in V_{2n}}$ and $\{\rv \kappa\}_{a\in C},A,B$. Since the $\{\rv \delta''_v\}_{v\in V_{2n}}$, $\{\rv \delta'''_v\}_{v\in V_{2n}}$ are i.i.d. Poisson, we have
        \begin{align*}
            \pr[\{\rv \delta_v''=d_v\}_{v\in \alpha},\{\rv \delta_v'''=d_v\}_{v\in \beta}]=\prod_{v\in \alpha}\frac{e^{-\rho d^*}(\rho d^*)^{d_v}}{d_v!}\prod_{v\in \beta}\frac{e^{-(1-\rho) d^*}((1-\rho) d^*)^{d_v}}{d_v!}.
        \end{align*}
        The sum $\sum_{v\in \alpha}\rv \delta ''_v+\sum_{v\in\beta} \rv \delta'''_v$ is Poisson with mean $\rho|\alpha| d^*+(1-\rho)|\beta|d^*$, so we get
        \begin{align}
            \frac{\pr[\{\rv \delta_v''=d_v\}_{v\in \alpha},\{\rv \delta_v'''=d_v\}_{v\in \beta}]}{\pr[\sum_{v\in \alpha}\rv \delta ''_v+\sum_{v\in\beta} \rv \delta'''_v=M]}&=\frac{e^{\rho|\alpha| d^*+(1-\rho)|\beta|d^*}M!}{(\rho|\alpha| d^*+(1-\rho)|\beta|d^*)^M}\\
            &\times\prod_{v\in \alpha}\frac{e^{-\rho d^*}(\rho d^*)^{d_v}}{d_v!}\prod_{v\in \beta}\frac{e^{-(1-\rho) d^*}((1-\rho) d^*)^{d_v}}{d_v!}\nonumber\\
            &=\frac{\rho ^{\sum_{v\in \alpha}d_v}(1-\rho)^{\sum_{v\in \beta}d_v}M!}{\prod_{v\in \alpha\cup \beta}d_v!(\rho|\alpha| +(1-\rho)|\beta|)^M}.\label{eq:priidPo}
        \end{align}
        Now consider the distribution of $\{\deg_{\hat M'}(v)\}_{v\in \hat A\cap \hat B}$. Let $\rv{\hat \gamma}_v=\deg_{\hat M'_{n,k,\rho,d^*,\eps}}(v)$ for $v\in \hat A\cup \hat B$, and let $\rv{\gamma}_v=\deg_{ M'_{n,k,\rho,d^*,\eps}}(v)$ for $v\in A\cup B$. Using the independence of $\{\rv \delta ''_v\}_{v\in V_{2n}},\{\rv \delta'''_v\}_{v\in V_{2n}}$ and $\{\rv \kappa\}_{a\in C},A,B$, we have
        \begin{align*}
            &\pr[\{\rv {\hat \gamma}_v\}_{v\in \hat A\cup \hat B}=\{d_v\}_{v\in \alpha\cup \beta}\mid \{\rv{\hat\kappa_a}\}_{a\in C}=\{k_a\}_{a\in C},\hat A=\alpha,\hat B=\beta]\\
            &=\frac{\pr[\{\rv {\gamma}_v\}_{v\in A\cup B}=\{d_v\}_{v\in \alpha\cup \beta},\{\rv\kappa_a\}_{a\in C}=\{k_a\}_{a\in C},A=\alpha,B=\beta,\mc D'_n]}{\pr[\{\rv\kappa_a\}_{a\in C}=\{k_a\}_{a\in C},A=\alpha,B=\beta,\mc D'_n]}\\
            &=\pr[\{\rv {\gamma}_v\}_{v\in A\cup B}=\{d_v\}_{v\in \alpha\cup \beta}\mid \{\rv\kappa_a\}_{a\in C}=\{k_a\}_{a\in C},A=\alpha,B=\beta].
        \end{align*}
        Let $\mc B=\{(a,i)\in C\times [k]:b_{a,i}\in A\cup B\}$, and let $\mathscr{B}$ be the set of subsets $D\subset C\times [k]$ such that $|\{i\in[k]:(a,i)\in D\}|=k_a$ for all $a\in C$. We can then write
        \begin{align}
            &\pr[\{\rv {\gamma}_v\}_{v\in A\cup B}=\{d_v\}_{v\in \alpha\cup \beta}\mid \{\rv\kappa_a\}_{a\in C}=\{k_a\}_{a\in C},A=\alpha,B=\beta]\nonumber\\
            &=\sum_{D\in \mathscr{B}}\big(\pr[\{\rv {\gamma}_v\}_{v\in A\cup B}=\{d_v\}_{v\in \alpha\cup \beta}\mid \mc B=D,A=\alpha,B=\beta]\nonumber\\
            &\hspace{2cm}\times \pr[\mc B=D\mid \{\rv\kappa_a\}_{a\in C}=\{k_a\}_{a\in C},A=\alpha,B=\beta]\big)\label{eq:prsum}
        \end{align}
        To calculate this probability, first note that the $\{b_{a,i}\}_{a\in C,i\in[k]}$ in the definition of $M'$ remain independent conditioned on $\mc B$, $A$, and $B$. Then, for any $(a,i)\in D$ and $v\in \alpha$, we have 
        \begin{align*}
            \pr[b_{a,i}=v\mid \mc B=D, A=\alpha, B=\beta]&=\pr[b_{a,i}=v\mid b_{a,i}\in \alpha\cup \beta,A=\alpha, B=\beta]\\
            &=\frac{\pr[b_{a,i}=v\mid A=\alpha,B=\beta]}{\pr[b_{a,i}\in \alpha\cup \beta\mid  A=\alpha,B=\beta]}
        \end{align*}
        Since $|\alpha|,|\beta|\leq n$, the event $A=\alpha,B=\beta$ implies $|A\cup A'|=|B\cup B'|=n$. Hence,
        \begin{equation*}
            \frac{\pr[b_{a,i}=v\mid A=\alpha,B=\beta]}{\pr[b_{a,i}\in \alpha\cup \beta\mid  A=\alpha,B=\beta]}=\frac{\rho/n}{\rho |\alpha|/n+(1-\rho)|\beta|/n}.
        \end{equation*}
        Similarly, we have $\pr[b_{a,i}=v\mid \mc B=D, A=\alpha, B=\beta]=(1-\rho)/(\rho|\alpha|+(1-\rho)|\beta|)$ for $v\in \beta$. Therefore, we get
        \begin{align*}
            \pr[\{\rv {\gamma}_v\}_{v\in A\cup B}=\{d_v\}_{v\in \alpha\cup \beta}\mid \mc B=D,A=\alpha,B=\beta]&=\binom{M}{d_v,v\in \alpha\cup \beta}\prod_{v\in \alpha}\lt(\frac{\rho}{\rho|A|+(1-\rho)|\beta|}\rt)^{d_v}\\&\times \prod_{v\in \beta}\lt(\frac{1-\rho}{\rho|A|+(1-\rho)|\beta|}\rt)^{d_v}\\
            &=\frac{\rho^{\sum_{v\in \alpha}d_v}(1-\rho)^{\sum_{v\in \beta}d_v}M!}{\prod_{v\in \alpha\cup\beta}d_v!(\rho|\alpha|+(1-\rho)|\beta|)^M}
        \end{align*}
        
        Substituting into~\eqref{eq:prsum} shows that
        \begin{align}
            &\pr[\{\rv {\gamma}_v\}_{v\in \hat A\cup \hat B}=\{d_v\}_{v\in \alpha\cup \beta}\mid \{\rv\kappa_a\}_{a\in C}=\{k_a\}_{a\in C},A=\alpha,B=\beta]\nonumber\\
            &\hspace{2cm}=\frac{\rho^{\sum_{v\in \alpha}d_v}(1-\rho)^{\sum_{v\in \beta}d_v}M!}{\prod_{v\in \alpha\cup\beta}d_v!(\rho|\alpha|+(1-\rho)|\beta|)^M}\label{eq:prdconditional}
        \end{align}
        which agrees with~\eqref{eq:priidPo}. This proves the claim.
    \end{proof}
    Claims~\ref{claim:M2degreesequence} and~\ref{claim:M'degreesequence} show that the degree sequences of $(M'\mid \mc D'_n,\mc C)$ and $M^{(2)}$ are identically distributed. Therefore the following claim completes the proof of the lemma.
    \begin{claim}
        Let $\{d_v\}_{v\in T}\in \NN^{V_{A\cup B}}$ for some $T\subset V_{2n}$ of size $(2-2\eps)n$, and let $\{k_a\}_{a\in C}\in \NN^{C}$. Then, conditioned on having degree sequences $\{d_v\}_{v\in T}$, $\{k_a\}_{a\in C}$, the distributions of $M'\mid \mc D_n',\mc C$ and $M^{(2)}$ are identical.
    \end{claim}
    \begin{proof}
    Let $\alpha,\beta$ be disjoint such that $\alpha\cup \beta =T$. We have
        \begin{align*}
            &\pr[M'=G\mid \{\rv \kappa_a\}_{a\in C}=\{k_a\}_{a\in C},A=\alpha,B=\beta]\\
            &=\prod_{a\in C}\binom{k_a}{e_G(a,v),v\in V_{T}}\prod_{v\in \alpha}\lt(\frac{\rho}{\rho|\alpha|+(1-\rho)|\beta|}\rt)^{e_G(a,v)}\prod_{v\in \beta}\lt(\frac{1-\rho}{\rho|\alpha|+(1-\rho)|\beta|}\rt)^{e_G(a,v)}\\
            &=\frac{\prod_{a\in C}k_a!\rho^{\sum_{v\in \alpha}\deg_G(v)}(1-\rho)^{\sum_{v\in \beta}\deg_G(v)}}{\prod_{a\in C,v\in T}e(a,v)!(\rho|\alpha|+(1-\rho)|\beta|)^M}.
        \end{align*}
        Let $\rv \gamma_v$ be the degree of $v\in V_{2n}$ in $M'$. Then using~\eqref{eq:prdconditional}, we get
        \begin{align*}
            &\pr[M'=G\mid\{\rv {\gamma}_v\}_{v\in A\cup B}=\{d_v\}_{v\in T} \{\rv \kappa_a\}_{a\in C}=\{k_a\}_{a\in C},A=\alpha,B=\beta]\\
            &=\frac{\pr[M'=G\mid \{\rv \kappa_a\}_{a\in C}=\{k_a\}_{a\in C},A=\alpha,B=\beta]}{\pr[\{\rv {\gamma}_v\}_{v\in  A\cup  B}=\{d_v\}_{v\in T}\mid \{\rv\kappa_a\}_{a\in C}=\{k_a\}_{a\in C},A=\alpha,B=\beta]}\\
            &=\frac{\prod_{v\in T}d_v!\prod_{a\in C}k_a!}{M!\prod_{a\in C,v\in T}e_G(a,v)!}
        \end{align*}
        Since $M^{(2)}$ is drawn from the pairing model, the proof of Claim~\ref{M-M2-conditional} shows that
        \begin{equation*}
            \pr[M^{(2)}=G\mid\{\rv {\delta}_v\}_{v\in A\cup B}=\{d_v\}_{v\in T} \{\rv \kappa_a\}_{a\in C}=\{k_a\}_{a\in C}]=\frac{\prod_{v\in T}d_v!\prod_{a\in C}k_a!}{M!\prod_{a\in C,v\in T}e_G(a,v)!}
        \end{equation*}
        which proves the claim.
    \end{proof}
\end{proof}
\begin{proof}[Proof of Lemma~\ref{lem:BXYmcontigF(C'')}]
Let $\{\mc E_n\}_{n\in \NN}$ be a sequence of sets of simple factor graphs. By Lemma~\ref{lem:BXY-C''-contig}, we have
\begin{align*}
    \pr[F(\tilde B_{XY})\in \mc E_n]=\pr[M''\in \mc E_n\mid \mc S'\cap \mc C]
\end{align*}
and
\begin{align*}
    \pr[\{M''\in \mc E_n\}\cap \mc C]\leq\pr[M''\in \mc E_n\mid \mc S'\cap \mc C]\leq \frac{\pr[M''\in \mc E_n]}{\pr[\mc S'\cap \mc C]}
\end{align*}
where $\pr[\mc S'\cap \mc C]=\Omega(1)$ and $\pr[\mc C]=1-o(1)$. Therefore $M''\in \mc E_n$ a.a.s. if and only if $F(\tilde B_{XY})\in \mc E_n$ a.a.s.

By the definition of $F(\C'')$, we have $\pr[F(\C'')\in \mc E_n]=\pr[M''\in \mc E_n\mid M\in \mc S]$. By~\eqref{eq:prS-M}, $\pr[M\in \mc S]=\Omega(1)$, so we get
\begin{equation*}
    \pr[M''\in \mc E_n]\leq \pr[F(\C'')\in \mc E_n]\leq O(1)\pr[M''\in \mc E_n].
\end{equation*}
Therefore $F(\C'')\in \mc E_n$ a.a.s. if and only if $M''\in \mc E_n$ a.a.s. This completes the proof.
\end{proof}

\begin{proof}[Proof of Lemma~\ref{lem:G(1)mcontigF(C')}]
Let $\{\mc E_n\}$ be a sequence of sets of simple factor graphs, and suppose $\pr[G^{(1)}\in \mc E_n]=o(1)$. By Theorem~\ref{thm:asympequiv},
\begin{equation*}
    \pr[M^{(2)}\in \mc E_n]\leq \pr[G^{(2)}\in \mc E_n]\leq \pr[G^{(1)}\in \mc E_n]+o(1)=o(1).
\end{equation*}
Then by Lemma~\ref{lem:M'-M2-equiv}, $\pr[M'\in \mc E_n\mid \mc D_n']=o(1)$, and by Lemma~\ref{lem:M'conditioning-contiguity}, $\pr[M'\in \mc E_n]=o(1)$. Similar to the proof of Lemma~\ref{lem:BXYmcontigF(C'')}, we have $\pr[F(\C')\in \mc E_n]=\pr[M'\in \mc E_n\mid M'\in \mc S]$ with $\pr[ M'\in \mc S]=\Omega(1)$, so we get $\pr[F(\C')\in \mc E_n]=o(1)$. This completes the proof.
\end{proof}

\bibliographystyle{abbrv}
 \bibliography{theo}

@article{connamacher2012satisfiability,
  title={The satisfiability threshold for a seemingly intractable random constraint satisfaction problem},
  author={Connamacher, Harold and Molloy, Michael},
  journal={SIAM Journal on Discrete Mathematics},
  volume={26},
  number={2},
  pages={768--800},
  year={2012},
  publisher={SIAM}
}

@article{ayre2020satisfiability,
  title={The satisfiability threshold for random linear equations},
  author={Ayre, Peter and Coja-Oghlan, Amin and Gao, Pu and M{\"u}ller, No{\"e}la},
  journal={Combinatorica},
  volume={40},
  number={2},
  pages={179--235},
  year={2020},
  publisher={Springer}
}

@article{duboismandler2002XORSAT,
author = {Dubois, Olivier and Mandler, Jacques},
title = {The 3-XORSAT threshold},
journal = {Comptes Rendus Mathematique},
volume = {335},
number = {11},
pages = {963-966},
year = {2002},
issn = {1631-073X},
doi = {https://doi.org/10.1016/S1631-073X(02)02563-3},
url = {https://www.sciencedirect.com/science/article/pii/S1631073X02025633},
}

@article{pittelsorkin2016XORSATthreshold,
title = {The Satisfiability Threshold for $k$-XORSAT},
journal = {Combinatorics, Probability and Computing},
volume = {25},
issue = {2},
pages = {236-268},
year = {2016},
author = {Pittel, Boris and Sorkin, Gregory}
}

@misc{dietzfelbinger2010XORSATthreshold,
      title={Tight Thresholds for Cuckoo Hashing via XORSAT}, 
      author={Martin Dietzfelbinger and Andreas Goerdt and Michael Mitzenmacher and Andrea Montanari and Rasmus Pagh and Michael Rink},
      year={2010},
      eprint={0912.0287},
      archivePrefix={arXiv},
      primaryClass={cs.DS},
      url={https://arxiv.org/abs/0912.0287}, 
}

@article{achlioptas2006twomomentsNAESAT,
author = { Achlioptas, Dimitris and  Moore, Cristopher},
title = {Random k‐SAT: Two Moments Suffice to Cross a Sharp Threshold},
journal = {SIAM Journal on Computing},
volume = {36},
number = {3},
pages = {740-762},
year = {2006},
doi = {10.1137/S0097539703434231},

URL = {https://doi.org/10.1137/S0097539703434231},
eprint = {https://doi.org/10.1137/S0097539703434231}
}

@article{dingslysun2022kSAT,
author = {Jian Ding and Allan Sly and Nike Sun},
title = {{Proof of the satisfiability conjecture for large $k$}},
volume = {196},
journal = {Annals of Mathematics},
number = {1},
publisher = {Department of Mathematics of Princeton University},
pages = {1 -- 388},
year = {2022},
doi = {10.4007/annals.2022.196.1.1},
URL = {https://doi.org/10.4007/annals.2022.196.1.1}
}

@inproceedings{goerdt2012beyondxorsat,
  title={Satisfiability thresholds beyond $k$-XORSAT},
  author={Goerdt, Andreas and Falke, Lutz},
  booktitle={International Computer Science Symposium in Russia},
  pages={148--159},
  year={2012},
  organization={Springer}
}

@inproceedings{coja2020rank,
  title={The rank of sparse random matrices},
  author={Coja-Oghlan, Amin and Erg{\"u}r, Alperen A and Gao, Pu and Hetterich, Samuel and Rolvien, Maurice},
  booktitle={Proceedings of the Fourteenth Annual ACM-SIAM Symposium on Discrete Algorithms},
  pages={579--591},
  year={2020},
  organization={SIAM}
}

@article{coja2024full,
  title={The full rank condition for sparse random matrices},
  author={Coja-Oghlan, Amin and Gao, Pu and Hahn-Klimroth, Max and Lee, Joon and M{\"u}ller, Noela and Rolvien, Maurice},
  journal={Combinatorics, probability and computing},
  volume={33},
  number={5},
  pages={643--707},
  year={2024},
  publisher={Cambridge University Press}
}

@book{Bini2002,
    author = {Gilberto Bini and Flaminio Flamini},
    title = {Finite commutative rings and their applications},
    publisher = {Springer},
    year = {2002}
}

@book{Bollobas2001,
place={Cambridge},
edition={2}, 
series={Cambridge Studies in Advanced Mathematics}, 
title={Random Graphs}, publisher={Cambridge University Press}, 
author={Bollobás, Béla}, 
year={2001}, 
collection={Cambridge Studies in Advanced Mathematics}

}

@Article{Ding2016,
author={Ding, Jian
and Sly, Allan
and Sun, Nike},
title={Satisfiability Threshold for Random Regular nae-sat},
journal={Communications in Mathematical Physics},
year={2016},
month={Jan},
day={01},
volume={341},
number={2},
pages={435-489},
issn={1432-0916},
doi={10.1007/s00220-015-2492-8},
url={https://doi.org/10.1007/s00220-015-2492-8}
}

@article{Panagiotou_Pasch_2025, title={Satisfiability thresholds for regular occupation problems}, volume={34}, DOI={10.1017/S0963548324000440}, number={4}, journal={Combinatorics, Probability and Computing}, author={Panagiotou, Konstantinos and Pasch, Matija}, year={2025}, pages={491–527}
}

@misc{gao2026satisfiabilitythresholdsolutionspace,
      title={The satisfiability threshold and solution space of random uniquely extendable constraint satisfaction problems}, 
      author={Pu Gao and Theodore Morrison},
      year={2026},
      eprint={2512.13819},
      archivePrefix={arXiv},
      primaryClass={math.CO},
      url={https://arxiv.org/abs/2512.13819}, 
}

@article{friedgut1999sharp,
  title={Sharp thresholds of graph properties, and the $k$-sat problem},
  author={Friedgut, Ehud and Bourgain, Jean},
  journal={Journal of the American mathematical Society},
  volume={12},
  number={4},
  pages={1017--1054},
  year={1999}
}

@article{Achlioptas2001,
author = {Achlioptas, Dimitris and Chtcherba, Arthur and Istrate, Gabriel and Moore, Cristopher},
year = {2001},
month = {05},
pages = {},
title = {The phase transition in 1-in-k SAT and NAE 3-SAT},
journal = {Proceedings of the Annual ACM-SIAM Symposium on Discrete Algorithms},
doi = {10.1145/365411.365760}
}

@article{GOERDT1996469,
title = {A Threshold for Unsatisfiability},
journal = {Journal of Computer and System Sciences},
volume = {53},
number = {3},
pages = {469-486},
year = {1996},
issn = {0022-0000},
doi = {https://doi.org/10.1006/jcss.1996.0081},
url = {https://www.sciencedirect.com/science/article/pii/S0022000096900811},
author = {Andreas Goerdt}
}

@article{Kirousis1998,
author = {Kirousis, Lefteris M. and Kranakis, Evangelos and Krizanc, Danny and Stamatiou, Yannis C.},
title = {Approximating the unsatisfiability threshold of random formulas},
journal = {Random Structures \& Algorithms},
volume = {12},
number = {3},
pages = {253-269},
doi = {https://doi.org/10.1002/(SICI)1098-2418(199805)12:3<253::AID-RSA3>3.0.CO;2-U},
url = {https://onlinelibrary.wiley.com/doi/abs/10.1002/%28SICI%291098-2418%28199805%2912%3A3%3C253%3A%3AAID-RSA3%3E3.0.CO%3B2-U},
eprint = {https://onlinelibrary.wiley.com/doi/pdf/10.1002/%28SICI%291098-2418%28199805%2912%3A3%3C253%3A%3AAID-RSA3%3E3.0.CO%3B2-U},
year={1998}
}

@article{Tikhomirov2020,
author = {Tikhomirov},
year = {2020},
month = {03},
pages = {593},
title = {Singularity of random Bernoulli matrices},
volume = {191},
journal = {Annals of Mathematics},
doi = {10.4007/annals.2020.191.2.6}
}

@inproceedings{Cheeseman1991,
author = {Cheeseman, Peter and Kanefsky, Bob and Taylor, William M.},
title = {Where the really hard problems are},
year = {1991},
isbn = {1558601600},
publisher = {Morgan Kaufmann Publishers Inc.},
address = {San Francisco, CA, USA},
booktitle = {Proceedings of the 12th International Joint Conference on Artificial Intelligence - Volume 1},
pages = {331–337},
numpages = {7},
location = {Sydney, New South Wales, Australia},
series = {IJCAI'91}
}

@inproceedings{Selman1992,
author = {Selman, Bart and Levesque, Hector and Mitchell, David},
title = {A new method for solving hard satisfiability problems},
year = {1992},
isbn = {0262510634},
publisher = {AAAI Press},
booktitle = {Proceedings of the Tenth National Conference on Artificial Intelligence},
pages = {440–446},
numpages = {7},
location = {San Jose, California},
series = {AAAI'92}
}

@article{CRAWFORD1996,
title = {Experimental results on the crossover point in random 3-SAT},
journal = {Artificial Intelligence},
volume = {81},
number = {1},
pages = {31-57},
year = {1996},
note = {Frontiers in Problem Solving: Phase Transitions and Complexity},
issn = {0004-3702},
doi = {https://doi.org/10.1016/0004-3702(95)00046-1},
url = {https://www.sciencedirect.com/science/article/pii/0004370295000461},
author = {James M. Crawford and Larry D. Auton}
}

@inproceedings{Achlioptas2002,
author = {Achlioptas, Dimitris and Moore, Cristopher},
title = {The Asymptotic Order of the Random k -SAT Threshold},
year = {2002},
isbn = {0769518222},
publisher = {IEEE Computer Society},
address = {USA},
booktitle = {Proceedings of the 43rd Symposium on Foundations of Computer Science},
pages = {779–788},
numpages = {10},
series = {FOCS '02}
}

@article{CojaOghlan2016,
author = {Coja-Oghlan, Amin and Panagiotou, Konstantinos},
year = {2016},
month = {01},
pages = {985-1068},
title = {The asymptotic $k$-SAT threshold},
volume = {288},
journal = {Advances in Mathematics},
doi = {10.1016/j.aim.2015.11.007}
}

@article{Bordenave2009,
  title={The rank of diluted random graphs.},
  author={Charles Bordenave and Marc Lelarge and Justin Salez},
  journal={Annals of Probability},
  year={2009},
  volume={39},
  pages={1097-1121},
  url={https://api.semanticscholar.org/CorpusID:115175031}
}

@article{Costello2010,
author = {Costello, Kevin p. and Vu, Van},
title = {On the rank of random sparse matrices},
year = {2010},
journal = {Duke Mathematical Journal},
issue_date = {May 2010},
publisher = {Cambridge University Press},
address = {USA},
volume = {19},
number = {3},
issn = {0963-5483},
url = {https://doi.org/10.1017/S0963548309990447},
doi = {10.1017/S0963548309990447},
month = may,
pages = {321–342},
numpages = {22}
}

@article{glasgow2023,
  title={The exact rank of sparse random graphs},
  author={Glasgow, Margalit and Kwan, Matthew and Sah, Ashwin and Sawhney, Mehtaab},
  journal={arXiv preprint arXiv:2303.05435},
  volume={1},
  number={3},
  pages={4},
  year={2023}
}

@article{Cheong2021,
author = {Gilyoung Cheong and Yifeng Huang},
title = {{Cohen–Lenstra distributions via random matrices over complete discrete valuation rings with finite residue fields}},
volume = {65},
journal = {Illinois Journal of Mathematics},
number = {2},
publisher = {Duke University Press},
pages = {385 -- 415},
year = {2021},
doi = {10.1215/00192082-8939615},
URL = {https://doi.org/10.1215/00192082-8939615}
}

@book{Durrett2019, place={Cambridge}, edition={5}, series={Cambridge Series in Statistical and Probabilistic Mathematics}, title={Probability: Theory and Examples}, publisher={Cambridge University Press}, author={Durrett, Rick}, year={2019}, collection={Cambridge Series in Statistical and Probabilistic Mathematics}}

\end{document}